%% file: CC25.tex
\documentclass[preprint,1p,number,sort&compress]{elsarticle}

\usepackage{enumitem}
\usepackage{graphicx}
\usepackage{amsmath,amsfonts,amsthm}
\usepackage{tikz}
\usepackage{pgfplots}
\pgfplotsset{compat=newest}
\usetikzlibrary{plotmarks}
\usetikzlibrary{arrows.meta}
\usepgfplotslibrary{patchplots}
\usepackage{grffile}

\usepackage{hyperref}

\newtheorem{theorem}{Theorem}
\newtheorem{lemma}{Lemma}
\newtheorem{proposition}{Proposition}
\newtheorem{remark}{Remark}
\newtheorem{assumption}{Assumption}

\newcommand{\rme}{\mathrm{e}}
\newcommand{\rmi}{\mathrm{i}}
\newcommand{\RR}{\mathbb{R}}

\newcommand{\NN}{\mathbb{N}}
\newcommand{\norm}[1]{\lVert #1 \rVert}
\newcommand{\abs}[1]{\lvert #1 \rvert}

\begin{document}
\myfooter[L]{}
\begin{frontmatter}
\title{Exponential quadrature rules for
  problems with time-dependent fractional source} 
\author[1]{Marco Caliari}
\ead{marco.caliari@univr.it}

\author[1,2]{Fabio Cassini\corref{cor1}}
\ead{fabio.cassini@univr.it, cassini@altamatematica.it}

\cortext[cor1]{Corresponding author}

\affiliation[1]{organization={Department of Computer Science,
    University of Verona},addressline={Strada Le Grazie, 15},
  postcode={37134},
  city={Verona},
country={Italy}}

\affiliation[2]{organization={Istituto Nazionale di Alta Matematica},
  addressline={Piazzale Aldo Moro, 5},
  postcode={00185},
  city={Roma},
  country={Italy}}

\begin{abstract}
  In this manuscript, we propose newly-derived exponential
  quadrature rules for stiff linear
  differential equations with time-dependent fractional sources in the
  form $h(t^r)$, with $0<r<1$ and $h$ a sufficiently smooth function.
  To construct the methods,
  the source term is interpolated at $\nu$ collocation points by a suitable
  non-polynomial function, yielding to time marching schemes that we
  call Exponential Quadrature Rules for Fractional sources (EQRF$\nu$).
  We write the integrators in terms of special instances of the
  Mittag--Leffler functions that we call fractional $\varphi$ functions.
  We perform the error analysis of the schemes in the abstract framework
  of strongly continuous
  semigroups. Compared to classical exponential quadrature rules, which in our
  case of interest converge with order $1+r$ at most, we prove that
  the new methods may reach order  $1+\nu r$ for proper choices of
  the collocation points.
  Several numerical experiments demonstrate the theoretical findings
  and highlight the effectiveness of the approach.
\end{abstract}

\begin{keyword}
linear abstract ODEs \sep
stiff equations \sep
order reduction \sep
exponential quadrature \sep
fractional $\varphi$~functions
\end{keyword}
\end{frontmatter}

\section{Introduction}\label{sec:intro}
The numerical time integration of stiff differential equations requires careful
treatment due to stability issues that usually arise when employing standard
explicit time marching methods. In the last years, exponential integrators
proved to be a valuable way to effectively tackle the issue in many contexts.
We mention here, among the
others,~\cite{KT05,LPR19,CCEO24,CC24bis,C24,HL03,ETL17,KLS24} and the
seminal manuscript~\cite{HO10}. In this paper, we focus our attention to the
development of exponential integrators for linear ordinary differential
equations (ODEs) in the form
\begin{equation}\label{eq:ODE}
  \left\{
  \begin{aligned}
    y'(t) &= A y(t) + g(t) = A y(t) + h(t^r), \quad 0<r<1,\\
    y(0) &= y_0.
  \end{aligned}
  \right.
\end{equation}
Here $t\in[0,T]$, $y$ is the unknown function,
$A:\mathcal{D}(A)\subset X \rightarrow X$ is a linear operator on the
Banach space $(X,\norm{\cdot})$, and $g(t)=h(t^r)$ represents the
time-dependent fractional source term for the given parameter~$r$.
Results on the existence and uniqueness of solution
for this problem can be found, for instance, in~\cite[Section 4.2]{P83}.
Stiff equations in the form~\eqref{eq:ODE} typically arise when
considering linear partial differential equations (PDEs) as abstract
ODEs on suitable function spaces~\cite{P83},
or when dealing with systems of ODEs coming from the
semidiscretization in space of linear PDEs. A simple example is the
one-dimensional heat equation on the interval $\Omega=(0,1)$, coupled with
periodic or homogeneous Dirichlet boundary conditions, with heat source that
grows in time at a \emph{sublinear} rate $r$. In this case, the operator $A$
is the second-order differential operator $\partial_{xx}$, together with the
boundary conditions, or a suitable discretization thereof (e.g., by
standard finite differences).

Classical exponential integrators for the linear problem~\eqref{eq:ODE} are the
so-called \emph{exponential quadrature rules}~\cite{HO05}, the simplest
one being
\begin{equation}\label{eq:expquad1}
  y_{n+1} = \rme^{\tau A}y_n + \tau \varphi_1(\tau A)g(t_n+c_1\tau),
  \quad \tau\varphi_1(\tau A) = \int_0^\tau \rme^{(\tau-s)A}ds,
\end{equation}
with $c_1$ a collocation point in $[0,1]$. Here we introduced the time
discretization $t_n=n\tau$, $n=0,1,\ldots,N$,
with constant time step size $\tau$.
Under suitable assumptions (in particular $g'$ being absolutely integrable,
see~\cite[Theorem~1]{HO05}), the time marching method~\eqref{eq:expquad1} is
in general first-order convergent. The scheme is actually second-order
convergent if
$c_1=\frac{1}{2}$ and additional assumptions are verified, in particular $g''$
absolutely integrable (see~\cite[Theorem~2]{HO05}).
In the situation under investigation in this manuscript, however, the latter
is in general
not verified (since $0<r<1$), and therefore we expect an order reduction
when employing scheme~\eqref{eq:expquad1} with $c_1=\frac{1}{2}$.
This can be easily demonstrated numerically by considering the scalar ODE
\begin{equation}\label{eq:scalintro}
  \left\{
  \begin{aligned}
    y'(t) &= -y(t) + t^r, \quad t\in(0,T], \\
    y(0) &= 1.
  \end{aligned}
  \right.
\end{equation}
By performing the time marching with scheme~\eqref{eq:expquad1} up to $T=0.1$
for different number of time steps $N=2^j$ (with $j=2,4,6,8,10$), we
clearly see in
Figure~\ref{fig:motivating} that the order of convergence for the case $c_1=\frac{1}{2}$ drops
to $1.75$, $1.5$, and $1.25$ when $r=\frac{3}{4}$, $r=\frac{1}{2}$ and $r=\frac{1}{4}$, respectively.
On the other hand, as expected, the integrator is first-order convergent for
each $r$ when $c_1=0$.
\begin{figure}[!htb]
  \centering
  \input{plot_motivating1.tex}
  \input{plot_motivating2.tex}
  \input{plot_motivating3.tex}
  \caption{Observed rate of convergence of the error
    $\lvert y_N-y(T)\rvert$, with different values of $r$, for
    problem~\eqref{eq:scalintro}. The time marching is performed with
    scheme~\eqref{eq:expquad1}.
    The slope of the solid line is $-1$,
    of the dashed line is
    $-1.75$, of the dotted line is $-1.5$,
    and of the dashdotted line is $-1.25$.}
  \label{fig:motivating}
\end{figure}
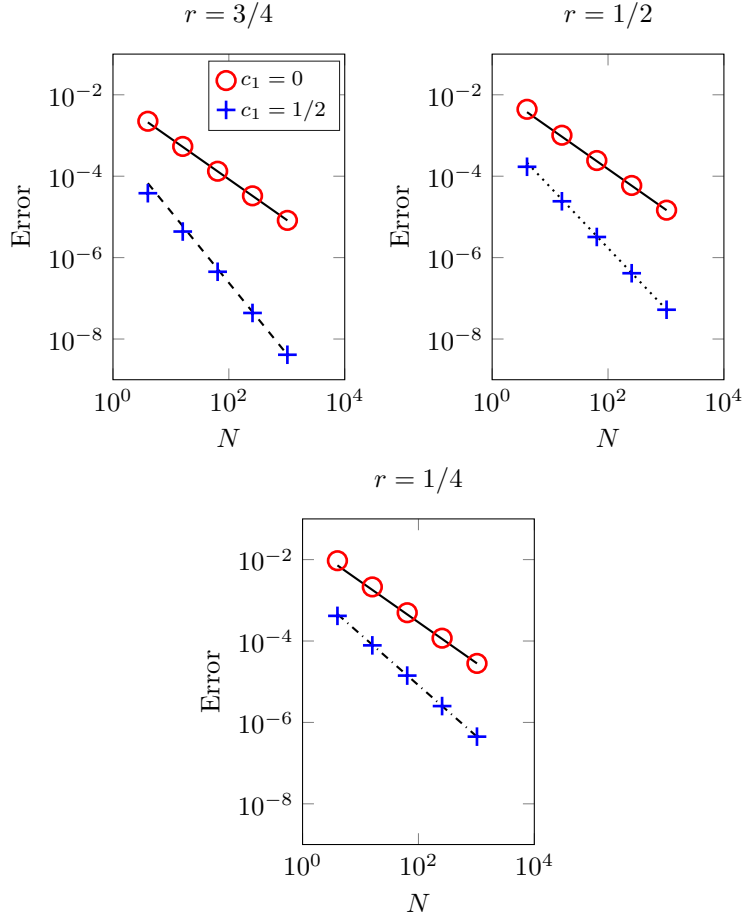
It is worth mentioning that the order reduction phenomenon of exponential
quadrature rules 
has already been studied in the literature in different situations.
We refer, e.g., to~\cite{CM18} in the context of linear equations
with time-dependent boundary conditions.

The simple numerical example just presented motivates the development and
the investigation of
exponential quadrature rules tailored for problems in the form~\eqref{eq:ODE}.
The remaining part of the manuscript is organized as follows.
In Section~\ref{sec:expquad}, the main one, we introduce the abstract framework
that we employ for the analysis and we present our proposed 
Exponential Quadrature Rules for Fractional sources with $\nu$ collocation
points (EQRF$\nu$).
The newly-derived methods are written in terms of linear combinations of 
special instances of the Mittag--Leffler functions, which we will call
fractional $\varphi$ functions.
We investigate in full details the integrators with one and two
non-confluent points which originate by interpolating the source term with a
non-polynomial basis.
The generalization of the procedure to $\nu$ collocation points is addressed
in Section~\ref{sec:nupoints}. In particular, if the points satisfy a peculiar
relation, we show in Theorem~\ref{thm:nupoints} that the proposed
integrators EQRF$\nu$ have convergence
order $1+\nu r$ (compared to classical quadrature rules which have order $1+r$).
We then proceed in Section~\ref{sec:numexp} by presenting several numerical
experiments that demonstrate the theoretical findings.
We finally draw the conclusions and discuss possible further
developments in Section~\ref{sec:conc}.
\section{The proposed exponential quadrature rules}\label{sec:expquad}
Classical
exponential quadrature rules originate from the expression of
the exact solution of problem~\eqref{eq:ODE} by means
of the so-called
\emph{variation-of-constants} formula
\begin{equation}\label{eq:voc}
  y(t_{n+1}) = \rme^{\tau A} y(t_n) + \int_0^\tau \rme^{(\tau-s)A}g((t_n+s))ds
             = \rme^{\tau A} y(t_n) + \int_0^\tau \rme^{(\tau-s)A}h((t_n+s)^r)ds,
\end{equation}
and by approximating the source term appearing in the integral with an
interpolation
\textit{polynomial} of a certain degree. The choice of the interpolation
points
(or equivalently the nodes of a quadrature rule) determines in fact the
numerical
method. For instance, by employing a single point (i.e.,
using the approximation $g(t_n+c_1\tau)\approx g(t_n+s)$)
we get scheme~\eqref{eq:expquad1}.
Under suitable assumptions, the employment of more terms in the polynomial
basis leads to different methods of higher order (see~\cite{HO05} for
a thorough presentation).
Our proposal is to embed the power $r$ directly into the interpolation
basis functions, that is considering $\{(t_n+s)^{jr}\}_{j=0}^{\nu-1}$.

As already mentioned in the introduction, for the analysis we will work
in an abstract
framework. A reader not familiar with the following formalism is invited to
consult, for instance, \cite{H81,P83}.
\begin{assumption}\label{assum:scs}
  Let $X$ be a Banach space with norm $\norm{\cdot}$.
  Let
$A \colon \mathcal{D}(A)\subset X \rightarrow X$ be
a linear operator.
We assume
that $A$ generates a strongly continuous
semigroup $\{\rme^{tA}\}_{t \geq 0}$.
\end{assumption}
Note that, in our cases of interest, $A$ is a differential operator (on
a suitable space) considered with boundary conditions
(e.g., $A=\rmi\partial_{xx}$ or
$A=\partial_{xx}$ with periodic or homogeneous Dirichlet boundary conditions).
Under Assumption~\ref{assum:scs}, the following bound on the semigroup
\begin{equation}\label{eq:expbound}
    \norm{\rme^{tA}} \leq C
\end{equation}
holds in the time interval $t\in[0,T]$.
Here and throughout the paper $C$ is a generic constant
that may have different values at different occurrences, but
it is independent of the stiffness of the problem.
Bound~\eqref{eq:expbound} directly implies that the
$\varphi_\ell$ functions
\begin{equation}\label{eq:phi}
  \varphi_0(z)=\rme^z,\quad 
  \varphi_\ell(z)=\frac{1}{(\ell-1)!}\int_0^1\rme^{(1-\theta)z}\theta^{\ell-1}d\theta=
  \sum_{k=0}^\infty\frac{z^k}{(\ell+k)!},\quad \ell\in\NN,
\end{equation}
are all bounded operators (that is, they satisfy
$\norm{\varphi_\ell(tA)} \leq C$).
We state now some lemmas needed for the following analysis.
\begin{lemma}\label{lem:h}
  Assume that $h\colon [0,+\infty)\to X$ is $k$ times differentiable,
    with bounded derivatives. Then, for $\sigma,\omega\ge0$, we have
  \begin{equation*}
    h(\sigma^r)=\sum_{j=0}^{k-1}
    \frac{h^{(j)}(\omega^r)}{j!}(\sigma^r-\omega^r)^j
    +\frac{h^{(k)}(\gamma_\sigma^r)}{k!}(\sigma^r-\omega^r)^k,
  \end{equation*}
where $\gamma_\sigma$ is between $\sigma$ and $\omega$.
\end{lemma}
\begin{proof}
  We take $x,y\ge0$ and consider the Taylor expansion of $h(y)$ at $x$
  \begin{equation*}
    h(y)=\sum_{j=0}^{k-1}\frac{h^{(j)}(x)}{j!}(y-x)^j
    +\frac{h^{(k)}(\lambda_y)}{k!}(y-x)^k,
  \end{equation*}
  where $\lambda_y$ is between $y$ and $x$.
  Now, we choose
  $y=\sigma^r$ and $x=\omega^r$. Since  $\sigma^r$ is an increasing
  monotone function, we have guarantee of existence
  of $\gamma_\sigma$, between $\sigma$ and $\omega$,
  such that $\gamma_\sigma^r=\lambda_y$.
  This concludes the proof.
\end{proof}
\begin{lemma}\label{lem:hn}
  Let $t_n=n\tau$, for $n=1,\ldots,N$, with $\tau=T/N$. Then
  the following bounds hold
  \begin{equation*}
\sum_{j=1}^n t_j^{-p}<  \begin{cases}
    \tau^{-p}+C\tau^{-1},& \text{if $p>0$, $p\neq 1$},\\
    \tau^{-p}(C+\abs{\log\tau}), & \text{if $p=1$}.
    \end{cases}
  \end{equation*}
\end{lemma}
\begin{proof}
  We have
  \begin{equation*}
    \sum_{j=1}^nt_j^{-p}=\tau^{-p}\sum_{j=1}^n\frac{1}{j^p}=
    \tau^{-p}H_n^{(p)}\le     \tau^{-p}H_N^{(p)},
  \end{equation*}
where $H_n^{(p)}$ denotes the generalized harmonic number
\begin{equation*}
  H_n^{(p)}=\sum_{j=1}^n\frac{1}{j^p}.
  \end{equation*}
Then, by employing the following bounds
(see for instance~\cite[formulas (25)--(26)]{C09}).
\begin{equation*}
  H_n^{(p)}<
  \begin{cases}
    1+Cn^{1-p},& \text{if $p>0,\ p\neq 1$},\\
    1+\log n, & \text{if $p=1$}
    \end{cases}
\end{equation*}
the result follows immediately.
\end{proof}
\subsection{Exponential quadrature rule with one collocation
  point (EQRF1)}\label{sec:EQRF1}
Starting from the variation-of-constants formula~\eqref{eq:voc} and
considering $h((t_n+c_1\tau)^r)\approx h((t_n+s)^r)$, for $s\in[0,\tau]$,
we get the scheme
\begin{equation}\label{eq:cc1q}
  y_{n+1} = \rme^{\tau A}y_n + \tau \varphi_1(\tau A)\alpha_{0,n},
  \quad \alpha_{0,n}=h((t_n+c_1\tau)^r).
\end{equation}
Obviously this method (labeled EQRF1) coincides with the simplest classical quadrature rule
\eqref{eq:expquad1}.
We now prove the following.
\begin{theorem}\label{thm:expquadsp}
Let Assumption~\ref{assum:scs} be valid, and consider the time marching
scheme~\eqref{eq:cc1q} for problem~\eqref{eq:ODE}. Then, the following bounds
on the error $\epsilon_n=y(t_n)-y_n$
hold uniformly on $t_n\in[0,T]$:
\begin{enumerate}
  \item if $h$ is differentiable with bounded
  derivative, then $\norm{\epsilon_n}\leq C\tau$;
\item if  $h$ is
  twice differentiable with bounded derivatives,
  $h'\in\mathcal{D}(A)$, and $c_1=\frac{1}{2}$,
  then $\norm{\epsilon_n}\leq C\tau^{1+r}$.
\end{enumerate}
The constant $C$ may depend on the final time $T$,
but not on $n$.
\end{theorem}
\begin{proof}
The first statement directly follows from~\cite[Theorem~1]{HO05}.
For the second one, setting $c_1=\frac{1}{2}$, we start by equivalently
writing scheme~\eqref{eq:expquad1}
as
\begin{equation*}
  y_{n+1} = \rme^{\tau A} y_n + \int_0^\tau\rme^{(\tau-s)A}
  h\left(\left(t_n+\frac{\tau}{2}\right)^r\right)ds.
\end{equation*}
By subtracting this expression from the variation-of-constants
formula~\eqref{eq:voc} we get
\begin{equation*}
  \epsilon_{n+1} = \rme^{\tau A}\epsilon_n + \delta_{n+1},
\end{equation*}
where we defined
\begin{equation*}
  \epsilon_n=y(t_n)-y_n \quad \text{and} \quad
  \delta_{n+1}= \int_0^\tau \rme^{(\tau-s) A}\left(h((t_n+s)^r)-h\left(\left(t_n+\frac{\tau}{2}\right)^r\right)\right)ds.
\end{equation*}
Since $\epsilon_0=0$, the recursion leads to
\begin{equation*}
  \epsilon_n = \sum_{j=0}^{n-1} \rme^{j\tau A}\delta_{n-j},
\end{equation*}
and by employing the bound on the semigroup~\eqref{eq:expbound} we get
\begin{equation*}
  \norm{\epsilon_n} \leq C\left(\norm{\delta_1}+
  \sum_{j=2}^n\norm{\delta_j}\right).
\end{equation*}
For the term $\delta_1$, by using Lemma~\ref{lem:h}, we get
\begin{equation*}
  \delta_1=\int_0^\tau \rme^{(\tau-s)A}
  \left(h(s^r)-h\left(\left(\frac{\tau}{2}\right)^r\right)\right)ds=
  \int_0^\tau \rme^{(\tau-s)A}
  h'(\gamma_s^r)\left(s^r-\left(\frac{\tau}{2}\right)^r\right)ds,
\end{equation*}
where $\gamma_s$ is between $s$ and $\frac{\tau}{2}$,
and therefore $\norm{\delta_1} \leq C\tau^{1+r}$.
For the remaining terms $\delta_j$, we have
\begin{multline*}
  h((t_{j-1}+s)^r) - h\left(\left(t_{j-1}+\frac{\tau}{2}\right)^r\right) = \\
  h'\left(\left(t_{j-1}+\frac{\tau}{2}\right)^r\right)
  \left((t_{j-1}+s)^r-\left(t_{j-1} + \frac{\tau}{2}\right)^r\right) +
  \frac{h''(\gamma_s^r)}{2}
  \left((t_{j-1}+s)^r-\left(t_{j-1} + \frac{\tau}{2}\right)^r\right)^2.
\end{multline*}
Since $t_{j-1}\ne 0$, by Taylor expansion
we have
\begin{equation*}
  (t_{j-1}+s)^r - \left(t_{j-1}+\frac{\tau}{2}\right)^r =
  r\left(t_{j-1}+\frac{\tau}{2}\right)^{r-1}\left(s-\frac{\tau}{2}\right)
  + \frac{r(r-1)}{2}\xi_{s,1}^{r-2}\left(s-\frac{\tau}{2}\right)^2,
\end{equation*}
and
\begin{equation*}
  \left((t_{j-1}+s)^r - \left(t_{j-1}+\frac{\tau}{2}\right)^r\right)^2 =
  r^2\xi_{s,2}^{2r-2}\left(s-\frac{\tau}{2}\right)^2,
\end{equation*}
with $\xi_{s,1}$ and $\xi_{s,2}$ between
$t_{j-1}+s$ and $t_{j-1}+\frac{\tau}{2}$.
Therefore, exploiting
the boundedness of the derivatives of $h$, the bound on the semigroup,
and the fact that $(t_{j-1}+s)^{r-1}\le t_{j-1}^{r-1}$,
we get
\begin{multline*}
  \norm{\delta_{j}} \leq
  C t_{j-1}^{r-1}\bigg\lVert\int_0^\tau \rme^{(\tau-s)A}\left(s-\frac{\tau}{2}\right)ds
    h'\left(\left(t_{j-1} + \frac{\tau}{2}\right)^r\right)\bigg\rVert\\
  + Ct_{j-1}^{r-2}\int_0^\tau\left(s-\frac{\tau}{2}\right)^2ds
 + Ct_{j-1}^{2r-2}\int_0^\tau \left(s-\frac{\tau}{2}\right)^2ds.
\end{multline*}
By exploiting that
\begin{equation*}
\int_0^\tau \rme^{(\tau-s)A}\left(s-\frac{\tau}{2}\right)dsh'\left(\left(t_{j-1} + \frac{\tau}{2}\right)^r\right) =
\tau^3\left(\varphi_3(\tau A)-\frac{1}{2}\varphi_2(\tau A)\right)Ah'\left(\left(t_{j-1} + \frac{\tau}{2}\right)^r\right),
\end{equation*}
we obtain
\begin{equation*}
  \norm{\delta_{j}} \leq C t_{j-1}^{r-1}\tau^3 + Ct_{j-1}^{r-2}\tau^3 + Ct_{j-1}^{2r-2}\tau^3,
  \end{equation*}
since by assumption $h'\in\mathcal{D}(A)$. Here, we also
exploited the boundedness of the $\varphi$ functions.
Finally, we conclude that
\begin{equation*}
  \begin{aligned}
    \norm{\epsilon_n} &\leq C\tau^{1+r} +
    C\tau^3 \left(\sum_{j=2}^{n}t_{j-1}^{r-1}+ \sum_{j=2}^{n}t_{j-1}^{r-2} 
    + \sum_{j=2}^{n}t_{j-1}^{2r-2}\right) \leq C\tau^{1+r},
  \end{aligned}
\end{equation*}
where we used the bounds in Lemma~\ref{lem:hn}.
\end{proof}
\subsection{Exponential quadrature rule with two collocation points
  (EQRF2)}\label{sec:EQRF2}
The classical exponential quadrature rules with two (or more) collocation
points still exhibit order reduction to $1+r$. This can be
easily proved by using similar reasoning as above (see
Section~\ref{sec:2pexact} for a numerical confirmation).
Here we employ the proposed approach to derive
exponential integrators of
order $1+2r$ using two collocation points. The generalization to $\nu$
points will be addressed in the next section.

We therefore consider the basis $\{1,(t_n+s)^r\}$ containing
two elements and perform the approximation
\begin{equation*}
  \alpha_{0,n} + (t_n+s)^r\alpha_{1,n}\approx h((t_n+s)^r).
  \quad s\in[0,\tau],
\end{equation*}
Here $\alpha_{0,n}$ and $\alpha_{1,n}$ are coefficients determined
by the imposing interpolation conditions on the
non-confluent collocation points $c_1$ and $c_2$ in $[0,1]$.
Simple calculations lead to
\begin{subequations}\label{eq:ccint_int}
\begin{equation}\label{eq:ccint_coeff}
  \alpha_{0,n} = h((t_n+c_1\tau)^r) - (t_n+c_1\tau)^r\alpha_{1,n}, \quad
  \alpha_{1,n} = \frac{h((t_n+c_2\tau)^r)-h((t_n+c_1\tau)^r)}{(t_n+c_2\tau)^r - (t_n+c_1\tau)^r}.
\end{equation}
Substituting the approximation into the variation-of-constants
formula~\eqref{eq:voc}, we obtain
\begin{equation}\label{eq:ccint_march}
  y_{n+1} = \rme^{\tau A} y_n + \tau \varphi_1(\tau A)\alpha_{0,n} +
  \int_0^\tau \rme^{(\tau-s) A}(t_n+s)^rds \alpha_{1,n}.
\end{equation}
\end{subequations}
The final form of the numerical scheme is obtained by employing a
generalization of the $\varphi$ functions appearing in classical exponential
exponential quadrature rules. More in detail, we avail of the following integral
formulation of the so-called Mittag--Leffler functions
\begin{equation}\label{eq:mittint}
  E_{1,\beta}(z) = \frac{1}{\Gamma(\beta-1)}\int_0^1 \rme^{(1-\theta) z}\theta^{\beta-2}d\theta,\quad \beta>1,
\end{equation}
where $\Gamma(\cdot)$ denotes Euler's gamma function
(see~\cite[Formula~1.100]{P99}). These
functions are particular cases of two-parameter Mittag--Leffler functions
$E_{\alpha,\beta}(z)$, 
typically encountered in the context of
fractional differential equations~\cite{P99,GP15,GP11}.
Note also that the Mittag--Leffler functions with $\alpha=1$ are related to
the so-called Miller--Ross functions. Comparing definitions~\eqref{eq:phi}
and \eqref{eq:mittint}, we clearly see that the Mittag--Leffler
functions $E_{1,\beta}$ generalize the $\varphi$ functions, and therefore
we choose to employ the notation
\begin{equation}\label{eq:philambda}
  \varphi_{\lambda}(z)=E_{1,1+\lambda}(z)=
  \frac{1}{\Gamma(\lambda)}\int_0^1\rme^{(1-\theta)z}\theta^{\lambda-1}d\theta,\quad \lambda\in\RR^+.
  \end{equation}
We call
these functions \emph{fractional} $\varphi$ functions.
Note that, as the classical $\varphi$ functions, they are bounded operators,
i.e., they satisfy $\norm{\varphi_{\lambda}(tA)}\leq C$ for $t\in[0,T]$.
We can now state the following.
\begin{proposition}\label{prop:inttomitt}
Let $\lambda\in\RR^+$. Then, we have
\begin{equation*}
  \int_0^\tau \rme^{(\tau-s) A}(t_n+s)^{\lambda-1}ds =
  \Gamma(\lambda)((t_n+\tau)^{\lambda}\varphi_{\lambda}((t_n+\tau)A)-t_n^{\lambda}\rme^{\tau A}\varphi_{\lambda}(t_nA)).
\end{equation*}
\end{proposition}
\begin{proof}
By splitting the integration interval into two parts we get
\begin{equation*}
  \int_0^\tau\rme^{(\tau-s) A}(t_n+s)^{\lambda-1}ds =
  \int_{-t_n}^\tau \rme^{(\tau-s) A}(t_n+s)^{\lambda-1}ds
  - \rme^{\tau A}\int_{-t_n}^0 \rme^{-sA}(t_n+s)^{\lambda-1}ds.
\end{equation*}
For the first integral we employ the change of variable
$(\tau-s) = (1-\theta)(t_n+\tau)$, while for the second one we perform the
substitution $-s = (1-\theta)t_n$. Then, the right hand side becomes
\begin{multline*}
   \int_0^1\rme^{(1-\theta)(t_n+\tau)A}((t_n+\tau)\theta)^{\lambda-1}(t_n+\tau)d\theta
  -\rme^{\tau A}\int_0^1 \rme^{(1-{\color{black}\theta})t_nA}(t_n\theta)^{\lambda-1}t_nd\theta \\
   = (t_n+\tau)^{\lambda}\int_0^1 \rme^{(1-\theta)(t_n+\tau)A}\theta^{\lambda-1} d\theta
  - t_n^{\lambda}\rme^{\tau A}\int_0^1 \rme^{(1-\theta)t_nA}\theta^{\lambda-1} d\theta.
\end{multline*}
The integral representation of the $\varphi_\lambda$
function~\eqref{eq:philambda} concludes the proof.
\end{proof}
By inserting the result of Proposition~\ref{prop:inttomitt} with
$\lambda=1+r$
into~\eqref{eq:ccint_march}, we finally get the proposed numerical scheme
EQRF2, i.e.,
\begin{multline}\label{eq:ccint_ml}
  y_{n+1} = \rme^{\tau A} y_n + \tau \varphi_1(\tau A)\alpha_{0,n}\\
  +\Gamma(1+r)((t_n+\tau)^{1+r}\varphi_{1+r}((t_n+\tau)A)-t_n^{1+r}\rme^{\tau A}\varphi_{1+r}(t_nA))\alpha_{1,n},
\end{multline}
where $\alpha_{0,n}$ and $\alpha_{1,n}$ are defined in
formula~\eqref{eq:ccint_coeff}.
Remark that, by construction, this integrator is exact on problems for which
$h(t^r)=t^rv$, being $v$ time-independent (see the numerical
experiments in Section~\ref{sec:2pexact}).
We now study the convergence of integrator~\eqref{eq:ccint_ml}.
\begin{theorem}\label{thm:2p}
  Let Assumption~\ref{assum:scs} be valid and assume that
  $h$ is twice differentiable with bounded derivatives.
  Consider the non-confluent points $c_1$ and $c_2$ in $[0,1]$.
Then, the following bound
{\color{black}
\begin{equation*}
  \norm{y(t_n)-y_n} \leq
  \begin{cases}
    C\tau^{\min\{1+2r,2\}} ,& \text{if $r\neq\frac{1}{2}$},\\
    C\tau^2(1+\lvert \log \tau \rvert), & \text{if $r=\frac{1}{2}$},
  \end{cases}
\end{equation*}
}
for the time marching
scheme~\eqref{eq:ccint_ml}
holds uniformly on $t_n\in[0,T]$.
The constant $C$ may depend on the final time $T$,
but not on $n$.
\end{theorem}
\begin{proof}
By comparing the exact solution expressed by means of the variation-of-constants
formula~\eqref{eq:voc} and the numerical one (see also formula~\eqref{eq:ccint_int}),
we get the recursion
\begin{equation*}
  \epsilon_{n+1} = \rme^{\tau A}\epsilon_n + \delta_{n+1},
\end{equation*}
where
\begin{equation*}
  \epsilon_{n} = y(t_n)-y_n \quad \mathrm{and} \quad
  \delta_{n+1} = \int_0^\tau \rme^{(\tau-s)A}\rho_n(s)ds,
\end{equation*}
with
\begin{equation}\label{eq:rhon}
  \rho_n(s)=h((t_n+s)^r)- (\alpha_{0,n} + (t_n+s)^r \alpha_{1,n}).
\end{equation}
Since $\epsilon_0=0$, similarly to the proof of Theorem~\ref{thm:expquadsp},
we get
\begin{equation*}
  \norm{\epsilon_n} \leq C\left(\norm{\delta_1} +
  \sum_{j=2}^{n}\norm{\delta_{j}}\right).
\end{equation*}
Expanding $h((t_{j-1}+s)^r)$ around $t_{j-1}+c_1\tau$ according to
Lemma~\ref{lem:h}, and doing simple algebraic manipulations,
we get
\begin{equation*}
    \rho_{j-1}(s)=
    ((t_{j-1}+s)^r - (t_{j-1}+c_1\tau)^r)\left(h'(\gamma_s^r)-
    \frac{h((t_{j-1}+c_2\tau)^r)-h((t_{j-1}+c_1\tau)^r)}{(t_{j-1}+c_2\tau)^r-(t_{j-1}+c_1\tau)^r}\right),
\end{equation*}
where $\gamma_s$ is between $t_{j-1}+c_1\tau$ and $t_{j-1}+s$.
Expanding now $h((t_{j-1}+c_2\tau)^r)$ around $t_{j-1}+c_1\tau$ we obtain
\begin{equation*}
  \rho_{j-1}(s)=
  ((t_{j-1}+s)^r - (t_{j-1}+c_1\tau)^r)(h'(\gamma_{s}^r)-h'(\eta^r)),
\end{equation*}
where $\eta$ is between $t_{j-1}+c_1\tau$ and $t_{j-1}+c_2\tau$.
Then, by using the bound on the semigroup~\eqref{eq:expbound}
and the fact that $h''$ is bounded
(i.e., $h'$ is Lipschitz), we get
\begin{equation*}
  \begin{aligned}
    \norm{\delta_{j}} &\leq C\int_0^\tau \norm{\rho_{j-1}(s)}ds
    \leq C \int_0^\tau \lvert  (t_{j-1}+s)^r-(t_{j-1}+c_1\tau)^r \rvert \lvert \gamma_s^r - \eta^r \rvert ds\\
    &\leq C\tau \max_{0\leq s \leq \tau}\lvert (t_{j-1}+s)^r-(t_{j-1}+c_1\tau)^r\rvert
    \max_{0\leq s \leq \tau} \lvert \gamma_s^r - \eta^r \rvert
    \leq C\tau ((t_{j-1}+\tau)^r-t_{j-1}^r)^2.
  \end{aligned}
\end{equation*}
Therefore $\norm{\delta_1}\leq C\tau^{1+2r}$. For the remaining $\delta_j$, by
Taylor expansion of $(t_{j-1}+\tau)^r$ around $t_{j-1}$ we get
\begin{equation*}
  \norm{\delta_{j}} \leq C\tau^3t_{j-1}^{2r-2}.
\end{equation*}
Finally, using Lemma~\ref{lem:hn}, we conclude that
{\color{black}
\begin{equation*}
    \norm{\epsilon_n} \leq C\tau^{1+2r}+C\tau^3\sum_{j=2}^{n}t_{j-1}^{2r-2}
    \leq
    \begin{cases}
      C\tau^{1+2r} + C\tau^2 ,& \text{if $r\neq\frac{1}{2}$},\\
      C\tau^2 + C\tau^2(1+\lvert \log \tau \rvert), & \text{if $r=\frac{1}{2}$},
  \end{cases}
\end{equation*}
which yields the result}.
\end{proof}

The rate of convergence can actually be improved if we impose additional
conditions on the collocation points and smoothness assumptions.
\begin{theorem}\label{thm:2pg}
  Let Assumption~\ref{assum:scs} be valid and assume that
  $h$ is three times differentiable with bounded derivatives.
  In addition, assume that $h''$ lies in $\mathcal{D}(A)$
  and that the non-confluent points $c_1$ and $c_2$ in $[0,1]$ satisfy
  the relation
\begin{equation}\label{eq:c1c2relation}
  \frac{1}{3} - \frac{1}{2}(c_1+c_2) + c_1c_2 = 0.
\end{equation}
  Then, the following bound
\begin{equation*}
  \norm{y(t_n)-y_n} \leq C\tau^{1+2r}
\end{equation*}
for the time marching
scheme~\eqref{eq:ccint_ml}
holds uniformly on $t_n\in[0,T]$.
The constant $C$ may depend on the final time $T$,
but not on $n$.
\end{theorem}
\begin{proof}
  By comparing the exact solution with the numerical one, from
  the recursion on $\epsilon_n = y(t_n)-y_n$
we get
\begin{equation*}
  \epsilon_n = \sum_{j=0}^{n-1}\rme^{j\tau A}\delta_{n-j},\quad
  \delta_{n+1} = \int_0^\tau \rme^{(\tau-s)A}\rho_n(s)ds,
\end{equation*}
where
$\rho_n(s)$ is defined in formula~\eqref{eq:rhon}.
As in the proof of Theorem~\ref{thm:2p},
 $\norm{\delta_1}\le C\tau^{1+2r}$. Consider now $j>1$.
Employing {\color{black}Lemma~\ref{lem:h}} and expanding the function
$h$ around $t_{j-1}+c_1\tau$ up to the second derivative, we get
\begin{multline*}
  \rho_{j-1}(s) = -\frac{h''(\eta^r)}{2}((t_{j-1}+c_2\tau)^r-(t_{j-1}+c_1\tau)^r)((t_{j-1}+s)^r-(t_{j-1}+c_1\tau)^r)\\
  +\frac{h''(\gamma_s^r)}{2}((t_{j-1}+s)^r-(t_{j-1}+c_1\tau)^r)^2,
\end{multline*}
where $\eta$ is between $t_{j-1}+c_1\tau$ and $t_{j-1}+c_2\tau$, while
$\gamma_s$ is between $t_{j-1}+c_1\tau$ and $t_{j-1}+s$.
Now, by writing $(t_{j-1}+c_2\tau)^r-(t_{j-1}+c_1\tau)^r$ as
$\left((t_{j-1}+s)^r-(t_{j-1}+c_1\tau)^r\right)+\left((t_{j-1}+c_2\tau)^r-(t_{j-1}{\color{black}+}s)^r\right)$,
we obtain the following expression for $\delta_j$
\begin{equation*}
  \begin{aligned}
  \delta_{j} &= \int_0^\tau \rme^{(\tau-s)A}\rho_{j-1}(s)ds\\
  &=
  \underbrace{\int_0^\tau \rme^{(\tau-s)A}
  ((t_{j-1}+s)^r-(t_{j-1}+c_1\tau)^r)^2\left(\frac{h''(\gamma_s^r)}{2}-\frac{h''(\eta^r)}{2}\right)ds}_{\psi_j}\\
  &\phantom{=}+ \underbrace{\int_0^\tau \rme^{(\tau-s)A}((t_{j-1}+s)^r-(t_{j-1}+c_1\tau)^r)((t_{j-1}+s)^r-(t_{j-1}+c_2\tau)^r)ds\frac{h''(\eta^r)}{2}}_{\omega_j}.
  \end{aligned}
\end{equation*}
Then, for the first integral term $\psi_j$ we have
\begin{equation*}
    \norm{\psi_j} \leq C\int_0^\tau ((t_{j-1}+s)^r-(t_{j-1}+c_1\tau)^r)^2
    \norm{h''(\gamma_s^r)-h''(\eta^r)}ds\leq C\tau^4t_{j-1}^{3r-3}.
\end{equation*}
Here{\color{black}, similarly to the proof of the previous theorem, we used the fact that $h''$ is Lipschitz and the Taylor expansion of $(t_{j-1}+\tau)^r$ around
$t_{j-1}$.}
For the second integral term $\omega_j$, using again Taylor expansions, we get
\begin{multline*}
  \norm{\omega_j}  = \bigg\lVert\int_0^\tau\rme^{(\tau-s)A}
    \left(r(t_{j-1}+c_1\tau)^{r-1}(s-c_1\tau)+\frac{r(r-1)}{2}\gamma_s^{r-2}(s-c_1\tau)^2\right)\\
    \times\left(r(t_{j-1}+c_2\tau)^{r-1}(s-c_2\tau)+\frac{r(r-1)}{2}\eta_s^{r-2}(s-c_2\tau)^2\right)    ds\frac{h''(\eta^r)}{2}\bigg\rVert,
\end{multline*}
where $\gamma_s$ is between $t_{j-1}+c_1\tau$ and $t_{j-1}+s$
and $\eta_s$ is between $t_{j-1}+c_2\tau$ and $t_{j-1}+s$.
Then, by using the triangle inequality we obtain
\begin{equation}\label{eq:omega1}
  \begin{aligned}
    \norm{\omega_j}&\leq C\bigg\lVert\underbrace{((t_{j-1}+c_1\tau)(t_{j-1}+c_2\tau))^{r-1}\int_0^\tau\rme^{(\tau-s)A}(s-c_1\tau)(s-c_2\tau)ds h''(\eta^r)}_{\omega_{j,1}}\bigg\rVert  \\
  &+C\bigg\lVert\underbrace{(t_{j-1}+c_2\tau)^{r-1}\int_0^\tau\rme^{(\tau-s)A}\gamma_s^{r-2}(s-c_1\tau)^2 (s-c_2\tau)dsh''(\eta^r)}_{\omega_{j,2}}\bigg\rVert\\
    &+C\bigg\lVert\underbrace{(t_{j-1}+c_1\tau)^{r-1}\int_0^\tau\rme^{(\tau-s)A}\eta_s^{r-2}(s-c_1\tau) (s-c_2\tau)^2dsh''(\eta^r)}_{\omega_{j,3}}\bigg\rVert\\
    &+C\bigg\lVert\underbrace{\int_0^\tau \rme^{(\tau-s)A}\gamma_s^{r-2}\eta_s^{r-2}(s-c_1\tau)^2(s-c_2\tau)^2ds h''(\eta^r)}_{\omega_{j,4}}\bigg\rVert.
  \end{aligned}
\end{equation}
Concerning the first term $\omega_{j,1}$, by exploiting that
\begin{multline*}
\int_0^\tau \rme^{(\tau-s)A}(s-c_1\tau)(s-c_2\tau)ds h''(\eta^r)=
\bigg(\tau^3\left(\frac{1}{3} - \frac{1}{2}(c_1+c_2) + c_1c_2\right)\\
     + \tau^4(2\varphi_4(\tau A)-(c_1+c_2)\varphi_3(\tau A)+c_1c_2\varphi_2(\tau A))A\bigg)h''(\eta^r)
\end{multline*}
we get
\begin{equation*}
  \norm{\omega_{j,1}}\leq Ct_{j-1}^{2r-2}\tau^4,
\end{equation*}
since $h''$ lies in $\mathcal{D}(A)$ and $c_1$ and $c_2$ satisfy
relation~\eqref{eq:c1c2relation}.
Concerning the remaining terms, using similar estimates as
in proof of Theorem~\ref{thm:expquadsp}, we have
\begin{equation*}
  \norm{\omega_{j,2}} \leq Ct_{j-1}^{2r-3}\tau^4,\quad
  \norm{\omega_{j,3}} \leq Ct_{j-1}^{2r-3}\tau^4,\quad
  \norm{\omega_{j,4}} \leq Ct_{j-1}^{2r-4}\tau^5.
\end{equation*}
Summarizing, for $j>1$ we get
\begin{equation*}
  \norm{\delta_{j}} \leq C\tau^4t_{j-1}^{3r-3} + C\tau^4t_{j-1}^{2r-2}
    + C\tau^4t_{j-1}^{2r-3}+C\tau^5t_{j-1}^{2r-4},
\end{equation*}
and therefore
\begin{equation*}
  \begin{aligned}
  \norm{\epsilon_n} &\leq C\left(\norm{\delta_1}+
  \sum_{j=2}^{n} \norm{\delta_{j}}\right) \\
  &\le C \tau^{1+2r}+
  C\tau^4 \sum_{j=2}^nt_{j-1}^{3r-3}
  + C\tau^4\sum_{j=2}^n t_{j-1}^{2r-2}
  + C\tau^4\sum_{j=2}^n t_{j-1}^{2r-3}
  + C\tau^5\sum_{j=2}^nt_{j-1}^{2r-4}.
  \end{aligned}
\end{equation*}
Finally, using Lemma~\ref{lem:hn}, we conclude
\begin{equation*}
  \norm{\epsilon_n} \leq C\tau^{1+2r}
\end{equation*}
which is the statement.
\end{proof}
We notice that relation~\eqref{eq:c1c2relation}  is satisfied, for instance,
by the Gauss and Gauss--Radau quadrature nodes.

\subsection{Exponential quadrature rule with \texorpdfstring{$\nu$}{nu}
  collocation points (EQRF\texorpdfstring{$\nu$}{nu})}\label{sec:nupoints}
We generalize now the procedure to $\nu$ collocation points by considering
the interpolation
\begin{equation*}
\sum_{j=0}^{\nu-1} \alpha_{j,n}(t_n+s)^{jr}\approx
  h((t_n+s)^r),\quad s\in[0,\tau],
\end{equation*}
where $\alpha_{0,n},\ldots,\alpha_{\nu-1,n}$ are the coefficients given by the
interpolation conditions on the non-confluent points $c_1,\ldots,c_\nu$
in $[0,1]$.
Substituting the approximation into the variation-of-constants
formula~\eqref{eq:voc} and using Proposition~\ref{prop:inttomitt}, we get
the integrator
\begin{multline}\label{eq:genint}
  y_{n+1}=\rme^{\tau A}y_n+\tau\varphi_1(\tau A)\alpha_{0,n}+\\
  \sum_{j=1}^{\nu-1}\Gamma(1+jr)\left((t_n+\tau)^{1+jr}\varphi_{1+jr}((t_n+\tau) A)-
  t_n^{1+jr}\rme^{\tau A}\varphi_{1+jr}(t_nA)\right)\alpha_{j,n},
\end{multline}
that we label EQRF$\nu$.
The theorems stated above generalize to the following.
\begin{theorem}\label{thm:nupoints}
  Let Assumption~\ref{assum:scs} be valid, and consider the non-confluent
  points $c_1,\ldots,c_\nu$ in $[0,1]$.  Then, the following bounds
  on the error $\epsilon_n=y(t_n)-y_n$ for the time marching
  scheme~\eqref{eq:genint} hold uniformly on $t_n\in[0,T]$:
  \begin{enumerate}
  \item if $h$ is $\nu$ times differentiable with bounded derivatives, then
     {\color{black}
       \begin{equation*}
         \norm{\epsilon_n}\le
         \begin{cases}
           C\tau^{\min\{1+\nu r,\nu\}} ,& \text{if $r\neq\frac{\nu-1}{\nu}$},\\
           C\tau^\nu(1+\lvert \log \tau \rvert), & \text{if $r=\frac{\nu-1}{\nu}$};
         \end{cases}
       \end{equation*}
     }
  \item if $h$ is $\nu+1$ times differentiable with bounded derivatives,
    $h^{(\nu)}\in\mathcal{D}(A)$, and the points satisfy
    \begin{equation}\label{eq:noderelation}
      \frac{1}{\nu+1}-\frac{c_1+\cdots+c_\nu}{\nu}+
      \sum_{i<j}\frac{c_ic_j}{\nu-1}-
      \sum_{i<j<k}\frac{c_ic_jc_k}{\nu-2}+\cdots+(-1)^\nu c_1\cdots c_\nu
      =0,
    \end{equation}
then $\norm{\epsilon_n}\le C\tau^{1+\nu r}$.
    \end{enumerate}
The constant $C$ may depend on the final time $T$,
but not on $n$.
\end{theorem}
\begin{proof}
  The proof is very similar to that of Theorems~\ref{thm:2p} and
  \ref{thm:2pg}, and employs Lemma~\ref{lem:h} and Taylor
  expansions. In particular, for the second statement, we exploit
the identity
  \begin{multline*}
      \int_0^1(s-c_1)\cdots(s-c_\nu)ds=\\
      \frac{1}{\nu+1}-\frac{c_1+\cdots+c_\nu}{\nu}+
      \sum_{i<j}\frac{c_ic_j}{\nu-1}-
      \sum_{i<j<k}\frac{c_ic_jc_k}{\nu-2}+\cdots+(-1)^\nu c_1\cdots c_\nu
  \end{multline*}
  in the following formula
  \begin{multline*}
    \int_0^\tau \rme^{(\tau-s)A}(s-c_1\tau)\cdots(s-c_\nu\tau)ds
          h^{(\nu)}(\eta^r)=\\
      \int_0^\tau (s-c_1\tau)\cdots(s-c_\nu\tau)ds      h^{(\nu)}(\eta^r)\\+
      \int_0^\tau \varphi_1((\tau-s)A)(\tau-s)(s-c_1\tau)\cdots(s-c_\nu\tau)ds
      Ah^{(\nu)}(\eta^r).
  \end{multline*}
This is used in the estimate of the term analogous to $\omega_{j,1}$ in
  formula~\eqref{eq:omega1}. We omit the remaining details.
\end{proof}
Notice that relation~\eqref{eq:noderelation} is satisfied by
any quadrature rule with $\nu$ nodes
of degree of exactness at least $\nu$, such as Gauss--Lobatto
(for $\nu\ge 3$), Gauss--Radau (for $\nu\ge 2$),
and Gauss.
\begin{remark}\label{rem:as}
  In the second statement of Theorem~\ref{thm:nupoints}, the assumption
  $h^{(\nu)}\in\mathcal{D}(A)$ can be removed if
$\{\rme^{t A}\}_{t\geq 0}$ is an \textit{analytic} semigroup, e.g.,
$A=\partial_{xx}$ with homogeneous Dirichlet boundary conditions.
This follows from the fact that in the analytic case the bounds
\begin{equation*}
  \norm{(tA)^\kappa\rme^{tA}}\leq C \quad \text{and} \quad
  \bigg\lVert\tau A \sum_{j=1}^{n-1}\rme^{j\tau A}\bigg\rVert \leq C
\end{equation*}
hold for $\kappa\geq 0$ and $0\leq t \leq T$ (see~\cite[Lemma~1]{HO05}).
We numerically confirm this in Section~\ref{sec:1p}.
\end{remark}
\section{Numerical experiments}\label{sec:numexp}
In this section we present some numerical examples that demonstrate the
theoretical findings and highlight the effectiveness
of the proposed procedure. {\color{black} We perform the experiments by suitably
semidiscretizing in space one-dimensional evolutionary PDEs, leading to a system
of ODEs in the form~\eqref{eq:ODE} where $A$ is the discretization matrix.}
All the simulations have been realized
in MathWorks Matlab\textsuperscript{\textregistered} R2022a on an
Intel Core i7-10750H CPU (16GB of RAM).
{\color{black}
\subsection{On the computation of the arising matrix functions}\label{sec:matfun}
The EQRF$\nu$ schemes require the computation of the fractional $\varphi$
functions~\eqref{eq:philambda} evaluated at a matrix argument.
Being special instances of the Mittag--Leffler functions, they
could be approximated by employing techniques available in the literature (such as
those explained in~\cite{DN24,GP18}). However, in our first three examples
(see Sections~\ref{sec:1p},~\ref{sec:2pexact}, and~\ref{sec:2p3p}) we employ
a Fourier pseudospectral semidiscretization in space,
and therefore all the needed matrix functions are computed in a scalar fashion in
Fourier space. In particular, we exploit the algorithm for general Mittag--Leffler
functions of scalar arguments described in~\cite{G15} to approximate the fractional $\varphi$ functions.
Notice that in these examples we are just
interested in numerically demonstrating the bounds derived in the previous
sections, and hence we do not discuss about the computational cost of the
procedures.

On the other hand, in the last example (see Section~\ref{sec:wpd})
we focus on the computational aspect. In fact, for that experiment we perform the semidiscretization
in space with standard second-order finite differences and we diagonalize once and for all the arising discretization matrix.
Then, we can compute
the matrix functions needed by EQRF2~\eqref{eq:ccint_ml} on scalar arguments.
This approach, however, has a non-negligible computational burden
(see Figure~\ref{fig:homdirgauss}), since 
the scalar fractional $\varphi$ functions 
depend on the current time and must be recomputed at each time step. This contrasts what happens for
the classical exponential quadrature rules, e.g., the scheme
\begin{multline}\label{eq:CEQR2}
  y_{n+1}=\rme^{\tau A}y_n+\tau\left(\frac{c_2}{c_2-c_1}\varphi_1(\tau A)
  -\frac{1}{c_2-c_1}\varphi_2(\tau A)\right)g(t_n+c_1\tau)\\+
  \tau\left(-\frac{c_1}{c_2-c_1}\varphi_1(\tau A)
  +\frac{1}{c_2-c_1}\varphi_2(\tau A)\right)g(t_n+c_2\tau).
\end{multline}
An alternative approach for EQRF2 is briefly described in the following.
We notice that the needed linear combination of fractional
$\varphi$ functions can be computed \textit{at once}
by using its integral representation, see formulation~\eqref{eq:ccint_int}
and Proposition~\ref{prop:inttomitt}.
In this sense, at each time step EQRF2 requires to compute
\begin{equation*}
  \int_0^\tau \rme^{(\tau-s)\Lambda}(t_n+s)^rds,
\end{equation*}
being $\Lambda$ the diagonal matrix containing the eigenvalues.
To this aim, if $t_n=0$ we approximate the integral
by using the Gauss--Jacobi quadrature formula with quadrature weight $s^r$,
while for $t_n> 0$ we employ the Gauss--Legendre
rule. In any case, we fix the number of quadrature nodes $\sigma_i$
to 16.
Note that, since the time step size $\tau$ is constant, this allows to
compute
the quantities $\rme^{(\tau-\sigma_i)\Lambda}$
once and for all before the actual time
integration starts. In fact, in our numerical example, this procedure results
in consistent computational savings without sacrificing accuracy (see the outcome
in Section~\ref{sec:wpd}).
The study of a general, detailed, and more sophisticated technique to compute
(linear combinations of) fractional $\varphi$ functions is a subject of
ongoing work and is far from the scope of this manuscript.

Finally, we mention that the $\varphi$ functions required by the classical exponential
quadrature rule~\eqref{eq:CEQR2}
could be computed by
standard algorithms (see, e.g.,~\cite{SW09,CC24,CC24bis,LYL22}). However, as
for the fractional $\varphi$ functions, we in fact compute them on scalar
arguments by exploiting the diagonalization procedure.
}
\subsection{Results with one collocation point (EQRF1)}\label{sec:1p}
We {\color{black}consider here} the integrator discussed in Section~\ref{sec:EQRF1},
i.e., the classical exponential quadrature rule with the single node
$c_1=\frac{1}{2}$. As test problem, we take
\begin{equation}\label{eq:perbc}
  \left\{
  \begin{aligned}
    \partial_t y(t,x) &= \zeta \partial_{xx} y(t,x) + t^{\frac{3}{4}}v(x), \\
    y(0{\color{black},x}) &= \sin(2\pi x)
  \end{aligned}
  \right.
\end{equation}
in the spatial domain $\Omega=(0,1)$ with periodic boundary conditions. The final
simulation time is $T=3$. We semidiscretize the problem with a Fourier pseudospectral
decomposition using $500$ modes.
We perform the first test with $\zeta=\rmi$ and $v(x)=1/(2+\cos(2\pi x))$.
This choice satisfies all the requirements of the second statement in
Theorem~\ref{thm:expquadsp}.
Therefore, since $r=\frac{3}{4}$, we expect order of convergence
$1+r=\frac{7}{4}$, which is indeed demonstrated
in Figure~\ref{fig:perbc} (top left). On the other hand, by taking
$\zeta=\rmi$ and $v(x)=x$, the latter does not satisfy the boundary
conditions, and therefore the requirement on the domain of the operator in
the second statement
of Theorem~\ref{thm:expquadsp} is violated. Indeed, the results
presented in Figure~\ref{fig:perbc} (top right) show an oscillatory behavior of
the error as the number of time steps increases. Note that this requirement is
not needed if the semigroup generated by the operator is analytic
(see Remark~\ref{rem:as}). In fact, we numerically demonstrate this by setting
$\zeta=1$ and $v(x)=1/(2+\cos(2\pi x))$ (Figure~\ref{fig:perbc}, bottom left),
and $\zeta=1$ and $v(x)=x$ (Figure~\ref{fig:perbc}, bottom right).
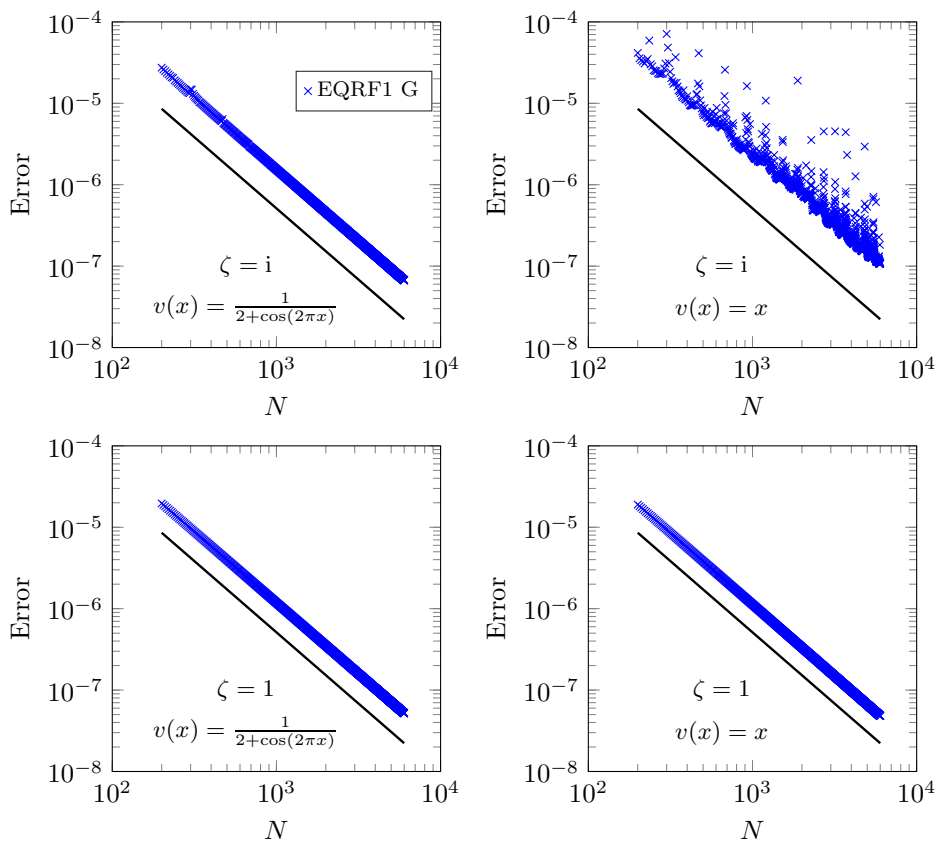
\begin{figure}[!htb]
  \centering
  $r=3/4$\\[2ex]
  \input{plot_sc_okb.tex}
  \input{plot_sc_nob.tex}\\
  \input{plot_an_okb.tex}
  \input{plot_an_nob.tex}
  \caption{Observed rate of convergence of the
    error $\lVert y_N-y(T)\rVert_\infty$, with
    different choices of $\zeta$ and $v(x)$, for problem~\eqref{eq:perbc}.
    The time marching is performed with scheme~\eqref{eq:cc1q} setting
    $c_1=\frac{1}{2}$.
    The slope of the solid line is $-1.75$.}
  \label{fig:perbc}
\end{figure}

\subsection{Exactness of EQRF2}\label{sec:2pexact}
We consider now the problem
\begin{equation}\label{eq:per}
  \left\{
  \begin{aligned}
    \partial_t y(t,x) &= \rmi\partial_{xx} y(t,x) + \frac{t^r}{2+\cos(2\pi x)}, \\
    y(0{\color{black},x}) &= \sin(2\pi x),
  \end{aligned}
  \right.
\end{equation}
with periodic boundary conditions in the spatial domain $\Omega=(0,1)$
and different choices of $r$. The time
interval is $[0,T]$, with $T=1$.
The spatial discretization is again performed with a Fourier pseudospectral method
using $500$ modes.
As time integrator, we consider
the proposed method EQRF2~\eqref{eq:ccint_ml} with $c_1=0$ and $c_2=1$
{\color{black}that we label EQRF2~T}.
In addition, we test the classical exponential quadrature rule~\eqref{eq:CEQR2}
again with $c_1=0$ and $c_2=1$ {\color{black} (labeled CEQR2~T)}.
Notice that the selected collocation points correspond to those of
the trapezoidal quadrature rule.
The results for
different values of $r$ are presented in Figure~\ref{fig:per}.
As mentioned at the beginning of Section~\ref{sec:EQRF2}, the CEQR2 T
method suffers from order reduction (to $1+r$, in particular), while the error
of the proposed integrator EQRF2 T is around $10^{-14}$ (which is
in fact the error of the routine employed for computing the fractional
$\varphi$ functions).
This is expected, since scheme~\eqref{eq:ccint_ml} is exact for
problem~\eqref{eq:per}.
\begin{figure}[!htb]
  \centering
  \input{plot_per1.tex}
  \input{plot_per2.tex}
  \input{plot_per3.tex}
  \caption{Observed convergence rate of the error
    $\lVert y_N-y(T)\rVert_\infty$, with
    different values of $r$, for problem~\eqref{eq:per}.
    For both the classical exponential quadrature rule
    CEQR2 T \eqref{eq:CEQR2} and our proposed
    method EQRF2 T \eqref{eq:ccint_ml} we employ the points $c_1=0$ and $c_2=1$
    (i.e., the trapezoidal quadrature nodes).
    The slope of the solid line is $-1.75$,
    of the dashed lines is $-1.5$, of the dotted line
    is $-1.25$.}
  \label{fig:per}
\end{figure}
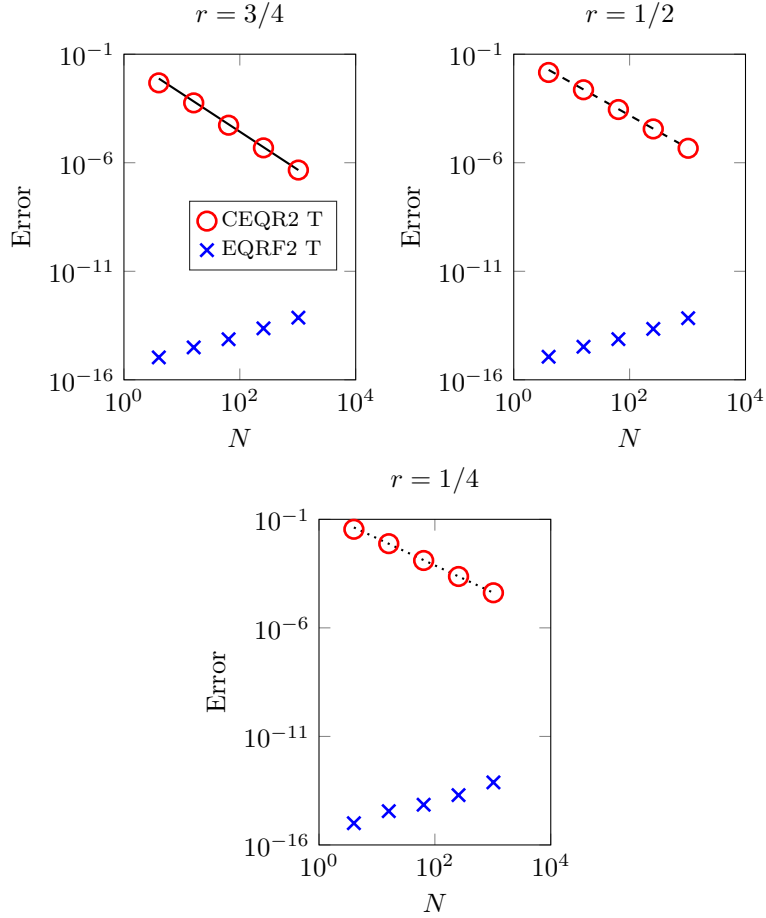

\subsection{Results with two and three collocation points (EQRF2 and EQRF3)}\label{sec:2p3p}
We focus now our attention on the problem
\begin{equation}\label{eq:perrad}
  \left\{
  \begin{aligned}
    \partial_t y(t,x) &= \rmi \partial_{xx} y(t,x) + \frac{\rme^{t^r}}{2+\cos(2\pi x)}, \\
    y(0{\color{black},x}) &= \sin(2\pi x),
  \end{aligned}
  \right.
\end{equation}
with periodic boundary conditions in the spatial domain $\Omega=(0,1)$
and different choices of $r$. The time
interval is $[0,T]$, with final time set to $T=2$.
The spatial discretization is performed once again with a Fourier pseudospectral method
using $500$ modes.
The results of the integrator EQRF2 \eqref{eq:ccint_ml},
 choosing as collocation
points the
trapezoidal nodes (labeled EQRF2 T) and the Gauss--Radau quadrature nodes
(labeled EQRF2 GR), are summarized in
Figure~\ref{fig:periodicrad}. Remark that the trapezoidal nodes do not
satisfy relation~\eqref{eq:c1c2relation} in Theorem~\ref{thm:2pg},
while it is the case for the Gauss--Radau ones.
{\color{black} The obtained outcome matches with the analysis carried out in
Section~\ref{sec:EQRF2}}.
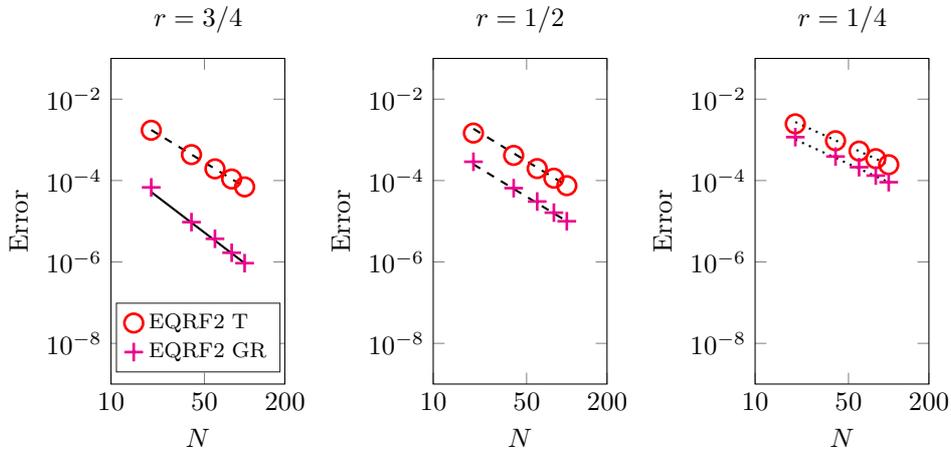
\begin{figure}[!htb]
  \centering
  \input{plot_periodicrad1.tex}
  \input{plot_periodicrad2.tex}
  \input{plot_periodicrad3.tex}
  \caption{Observed convergence rate of the error
    $\lVert y_N-y(T)\rVert_\infty$,
    with different values of $r$, for problem~\eqref{eq:perrad}.
    The time marching is performed with EQRF2 \eqref{eq:ccint_ml}. We employ
    the trapezoidal (EQRF2 T) and the Gauss--Radau (EQRF2 GR) collocation
    points.
    The slope of the solid line is $-2.5$, of the dashed line is $2$
    while of the dotted line is $-1.5$.}
  \label{fig:periodicrad}
\end{figure}

We then repeat the same example using the proposed technique with three
collocation
points, i.e., using scheme~\eqref{eq:genint} with $\nu=3$. More in detail,
we choose a set of points that does not satisfy
condition~\eqref{eq:noderelation}
in Theorem~\ref{thm:nupoints} ($c_1=0$, $c_2=\frac{1}{3}$, and $c_3=1$)
and a set that satisfies it
(the Gauss--Lobatto nodes, corresponding to Simpson's quadrature rule).
The resulting integrators are labeled EQRF3 NC and 
EQRF3 GL, respectively. The results, presented in Figure~\ref{fig:periodiclob},
are in line with the theoretical findings {\color{black}of Section~\ref{sec:nupoints}}.
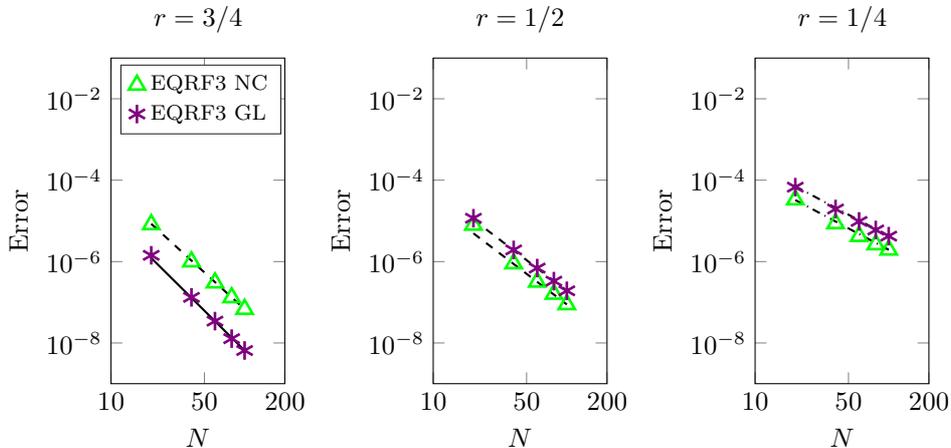
\begin{figure}[!htb]
  \centering
  \input{plot_periodiclob1.tex}
  \input{plot_periodiclob2.tex}
  \input{plot_periodiclob3.tex}
  \caption{Observed convergence rate of the error
    $\lVert y_N-y(T)\rVert_\infty$, 
    with different values of $r$, for problem~\eqref{eq:perrad}.
    The time marching is performed with EQRF3 \eqref{eq:genint}.
    We employ 
    $c_1=0$, $c_2=\frac{1}{3}$, and $c_3=1$ (EQRF3 NC) and the
    Gauss--Lobatto (EQRF3 GL) collocation points.
    The slope of the solid line is $-3.25$, of the dashed line is $-3$,
    of the dotted line is $-2.5$, while of the dashdotted line is $-1.75$.}
  \label{fig:periodiclob}
\end{figure}

\subsection{Work precision diagram}\label{sec:wpd}
Finally, we consider the variable-coefficient heat equation
\begin{equation}\label{eq:heat}
  \left\{
  \begin{aligned}
    \partial_t y(t,x) &= \frac{1+x^2}{10}\partial_{xx} y(t,x) +
    \rme^{t^{\frac{3}{4}}}x(1-x), \\
    y(0{\color{black},x}) &= 4x(1-x),
  \end{aligned}
  \right.
\end{equation}
coupled with homogeneous Dirichlet boundary conditions, in the spatial domain
$\Omega=(0,1)$. The time interval is $[0,T]$, with $T=2$.
We perform the spatial discretization with standard second-order finite
differences using $1000$ inner points, which leads to a system of ODEs in the
form~\eqref{eq:ODE}.
{\color{black} As time integrator we employ the EQRF2 method~\eqref{eq:ccint_ml}
with two Gauss quadrature nodes.
We label the scheme EQRF2~G~(F) when using the formulation based on fractional
$\varphi$ functions, while we label it EQRF2~G~(I) when using the 
equivalent integral formulation (see the discussion in Section~\ref{sec:matfun}).
As term of comparison we consider the classical exponential quadrature
rule~\eqref{eq:CEQR2} with two Gauss quadrature nodes and denote it
CEQR2~G.
We remark that the matrix $A$ of the ODEs system can be symmetrized
and subsequently diagonalized as $A = (DV)\Lambda(V^{\sf T}D^{-1})$,
being $D$ the diagonal symmetrization matrix,
$V$ an orthogonal matrix, and $\Lambda$ the
diagonal matrix collecting the eigenvalues. Therefore, all the
relevant matrix functions can be computed in a scalar fashion on $\Lambda$.}
The outcome of the simulations is summarized in Figure~\ref{fig:homdirgauss}.
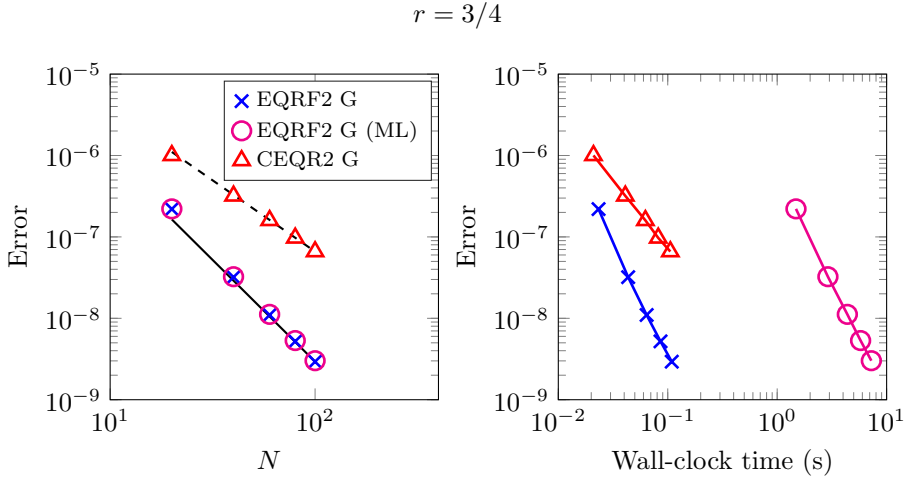
\begin{figure}[!htb]
  \centering
  $r=3/4$\\[2ex]
  \input{plot_homdirg_err.tex}
  \input{plot_homdirg_cpu.tex}
  \caption{Observed convergence rate of the error
    $\lVert y_N-y(T)\rVert_\infty$ (left) and
    work-precision diagram (right)
    for problem~\eqref{eq:heat}. The time marching schemes
    employ two Gauss quadrature nodes.
    We label CEQR2 G the classical exponential quadrature rule
    \eqref{eq:CEQR2}, EQRF2 G (I)
    the integral formulation of the proposed method \eqref{eq:ccint_int},
    and EQRF2 G (F) the fractional $\varphi$ functions formulation of the proposed
    method \eqref{eq:ccint_ml}.
    The slope of the solid line is $-2.5$, while
    of the dashed lines is $-1.75$.}
  \label{fig:homdirgauss}
\end{figure}

First of all, from the left plot we observe that the classical
exponential quadrature
rule CEQR2 G shows the expected order of convergence
$1+r$.
The proposed method EQRF2 clearly exhibits order of convergence $1+2r$,
both when using the integral formulation EQRF2~G~(I) and when
employing the fractional $\varphi$ functions formulation EQRF2~G~(F).
In particular, notice that the overall accuracy of the time marching
method is not influenced by the underlying technique employed to compute the
matrix functions.
On the other hand, in terms of computational cost, the advantage of EQRF2~G~(I)
compared to EQRF2~G~(F)
is clear (see the right plot in Figure~\ref{fig:homdirgauss}). Therefore, we
conclude that for this example the best performing approach is EQRF2~G~(I).

\section{Conclusions and future work}\label{sec:conc}
In this paper we presented a class of collocation-type exponential
integrators tailored 
for linear problems with time-dependent fractional sources. The proposed
schemes are analyzed in the context of strongly continuous
semigroups, and 
may reach higher order compared to classical exponential quadrature rules. The
performed numerical experiments match satisfactorily with the theoretical
investigations and clearly show the superiority of the approach.
As future work, we plan to study a more sophisticated technique to efficiently
compute the special instances of matrix Mittag--Leffler functions which arise
in
the proposed integrators EQRF$\nu$. Also, we plan to generalize the approach
for the time
integration of stiff semilinear PDEs. This would be useful, for instance, in
the context of diffusion equations with sublinear growth~\cite{WY12}.

\section*{Acknowledgements}
The authors are members of the Gruppo Nazionale
Calcolo Scientifico-Istituto Na\-zio\-na\-le di Alta Matematica (GNCS-INdAM).
Fabio Cassini holds a post-doc fellowship funded by INdAM.

\bibliographystyle{plain}
\bibliography{CC25}

\end{document}

%% file: plot_motivating1.tex
%
%
\begin{tikzpicture}

\begin{axis}[%
width=1.5in,
height=1.7in,
at={(0.758in,0.481in)},
scale only axis,
xmode=log,
xmin=1,
xmax=10^4,
xminorticks=true,
xlabel={$N$},
ymode=log,
ymin=1e-9,
ymax=1e-1,
ylabel={Error},
yminorticks=true,
title = {$r=3/4$},
axis background/.style={fill=white},
legend style={legend cell align=left, align=left, draw=white!15!black, font=\footnotesize}
]
\addplot [color=red, only marks,line width=1pt, mark size = 3.5pt, mark=o, mark options={solid, red}]
  table[row sep=crcr]{%
4	0.00224043707277599\\
16	0.000538682773112642\\
64	0.000132553561268622\\
256	3.29362083076878e-05\\
1024	8.21521328153718e-06\\
};
\addlegendentry{$c_1=0$}
\addplot [color=blue, only marks, line width=1pt, mark size=3.5pt, mark=+, mark options={solid, blue}]
  table[row sep=crcr]{%
4	3.82072798039701e-05\\
16	4.37189925528435e-06\\
64	4.49335075702882e-07\\
256	4.36627700661774e-08\\
1024	4.106270656834e-09\\
};
\addlegendentry{$c_1=1/2$}

\addplot [color=black, line width=0.8pt]
  table[row sep=crcr]{%
4	0.00210309460007352\\
1024	8.21521328153718e-06\\
};
\addplot [color=black, dashed, line width=0.8pt]
  table[row sep=crcr]{%
4	6.72771384415681e-05\\
1024	4.106270656834e-09\\
};

\end{axis}

\end{tikzpicture}%

%% file: plot_motivating2.tex
%
%
\begin{tikzpicture}

\begin{axis}[%
width=1.5in,
height=1.7in,
at={(0.758in,0.481in)},
scale only axis,
xmode=log,
xmin=1,
xmax=10^4,
xminorticks=true,
xlabel={$N$},
ymode=log,
ymin=1e-9,
ymax=1e-1,
ylabel={Error},
yminorticks=true,
title = {$r=1/2$},
axis background/.style={fill=white},
legend style={legend cell align=left, align=left, draw=white!15!black, font=\footnotesize}
]
\addplot [color=red, only marks,line width=1pt, mark size = 3.5pt, mark=o, mark options={solid, red}]
  table[row sep=crcr]{%
4	0.00440247209479816\\
16	0.0010149649272212\\
64.0000000000001	0.00024265085541908\\
256	5.92457900824295e-05\\
1024	1.4632200326381e-05\\
};
\addplot [color=blue, only marks, line width=1pt, mark size=3.5pt, mark=+, mark options={solid, blue}]
  table[row sep=crcr]{%
4	0.000170412312989754\\
16	2.42133651500165e-05\\
64.0000000000001	3.21328394037313e-06\\
256	4.13459853465525e-07\\
1024	5.24241506871803e-08\\
};

\addplot [color=black, line width=0.8pt]
  table[row sep=crcr]{%
4	0.00374584328355353\\
1024	1.4632200326381e-05\\
};
\addplot [color=black, dotted, line width=0.8pt]
  table[row sep=crcr]{%
4	0.000214729321214691\\
1024	5.24241506871803e-08\\
};

\end{axis}

\end{tikzpicture}%

%% file: plot_motivating3.tex
%
%
\begin{tikzpicture}

\begin{axis}[%
width=1.5in,
height=1.7in,
at={(0.758in,0.481in)},
scale only axis,
xmode=log,
xmin=1,
xmax=10^4,
xminorticks=true,
xlabel={$N$},
ymode=log,
ymin=1e-9,
ymax=1e-1,
ylabel={Error},
yminorticks=true,
title = {$r=1/4$},
axis background/.style={fill=white},
legend style={legend cell align=left, align=left, draw=white!15!black, font=\footnotesize}
]
\addplot [color=red, only marks,line width=1pt, mark size = 3.5pt, mark=o, mark options={solid, red}]
  table[row sep=crcr]{%
4	0.00940141495069434\\
16	0.00213321556751167\\
64.0000000000001	0.000495790661841244\\
256	0.000117347152885916\\
1024	2.81708734776176e-05\\
};
\addplot [color=blue, only marks, line width=1pt, mark size=3.5pt, mark=+, mark options={solid, blue}]
  table[row sep=crcr]{%
4	0.000414915140317329\\
16	7.8219036289906e-05\\
64.0000000000001	1.41479786729404e-05\\
256	2.52176166060103e-06\\
1024	4.47114457124975e-07\\
};

\addplot [color=black, line width=0.8pt]
  table[row sep=crcr]{%
4	0.0072117436102701\\
1024	2.81708734776176e-05\\
};
\addplot [color=black, dashdotted, line width=0.8pt]
  table[row sep=crcr]{%
4	0.000457845204095974\\
1024	4.47114457124975e-07\\
};

\end{axis}

\end{tikzpicture}%

%% file: plot_sc_okb.tex
%
%
\begin{tikzpicture}

\begin{axis}[%
width=1.7in,
height=1.7in,
at={(0.769in,0.477in)},
scale only axis,
xmode=log,
xmin=100,
xmax=10000,
xminorticks=true,
ymode=log,
ymin=1e-08,
ymax=0.0001,
yminorticks=true,
axis background/.style={fill=white},
xlabel = {$N$},
ylabel = {Error},
legend style={legend cell align=left, at={(0.975,0.85)}, align=left, draw=white!15!black, font=\footnotesize}
]
\addplot [color=blue, only marks, mark=x, mark options={solid, blue}]
  table[row sep=crcr]{%
200	2.74285884091437e-05\\
205	2.60686616616452e-05\\
210	2.48437866531435e-05\\
215	2.37301210105759e-05\\
220	2.27119396013883e-05\\
225	2.18068464088229e-05\\
230	2.08802637431627e-05\\
235	2.05646081236104e-05\\
240	1.92210261154628e-05\\
245	1.85085686389104e-05\\
250	1.78313385868335e-05\\
255	1.71926265777304e-05\\
260	1.65887581418105e-05\\
265	1.60169027266205e-05\\
270	1.54750761975213e-05\\
275	1.49600991670774e-05\\
280	1.44702584316596e-05\\
285	1.4005476581862e-05\\
290	1.35745806803333e-05\\
295	1.32600087897135e-05\\
300	1.48747541585692e-05\\
305	1.46887456699266e-05\\
310	1.30657130224756e-05\\
315	1.23931589960698e-05\\
320	1.19121046421526e-05\\
325	1.15110603541809e-05\\
330	1.11560467774826e-05\\
335	1.08370593179045e-05\\
340	1.03805579496143e-05\\
345	1.01983345458767e-05\\
350	9.93670452628783e-06\\
355	9.68217828211122e-06\\
360	9.43787976275077e-06\\
365	9.2038487607327e-06\\
370	8.97933206868479e-06\\
375	8.76381527837131e-06\\
380	8.55674776717394e-06\\
385	8.35766198131389e-06\\
390	8.16624649039635e-06\\
395	7.9819126744939e-06\\
400	7.80420435254142e-06\\
405.000000000001	7.63326631187102e-06\\
410	7.4681066957919e-06\\
415	7.30892662573125e-06\\
420	7.15529311532794e-06\\
425	7.00697294689084e-06\\
430.000000000001	6.86377314861825e-06\\
435.000000000001	6.72574484130719e-06\\
440	6.59280397998053e-06\\
445	6.46487565025088e-06\\
450	6.34295860110941e-06\\
455	6.22820602662182e-06\\
460	6.12653739353585e-06\\
465.000000000001	6.05302706091689e-06\\
470.000000000001	6.42543905091091e-06\\
475	5.73425319013145e-06\\
480	5.63248727557292e-06\\
485	5.53155658276995e-06\\
490	5.43273400610978e-06\\
495.000000000001	5.33640610352391e-06\\
500	5.24261083152811e-06\\
505.000000000001	5.15134983388055e-06\\
510	5.06244736719079e-06\\
515.000000000001	4.97589411057316e-06\\
520	4.89170298822196e-06\\
525	4.80973494925799e-06\\
530	4.72991263086786e-06\\
535	4.65216045448181e-06\\
540.000000000001	4.57641053330979e-06\\
545	4.50258839033852e-06\\
550	4.43063674783988e-06\\
555	4.36048823435981e-06\\
560	4.29208325992025e-06\\
565	4.22536436609365e-06\\
570	4.16027754724593e-06\\
575	4.0967660047273e-06\\
580.000000000001	4.03478337794871e-06\\
585.000000000001	3.97427931731581e-06\\
590.000000000001	3.91520567326096e-06\\
595	3.85751061714166e-06\\
600	3.801106298084e-06\\
605.000000000001	3.74647564578888e-06\\
610	3.69261617198174e-06\\
615	3.64011117837524e-06\\
620.000000000001	3.58883574364128e-06\\
625.000000000001	3.53876237215781e-06\\
630.000000000001	3.48977094781672e-06\\
635	3.44203461819203e-06\\
640	3.39543014098753e-06\\
645	3.34999215235275e-06\\
650	3.30576887380653e-06\\
655	3.26287255068082e-06\\
660.000000000001	3.22155982739276e-06\\
665	3.18248589305477e-06\\
670.000000000001	3.14799016402673e-06\\
675.000000000001	3.13196663055493e-06\\
680	3.0436343048913e-06\\
685	2.96701302581062e-06\\
690.000000000001	2.94162509091652e-06\\
695	2.91063929569018e-06\\
700	2.87771183067803e-06\\
705.000000000001	2.84422972739172e-06\\
710	2.81070396091149e-06\\
715.000000000001	2.77740816153675e-06\\
720.000000000001	2.74448028079027e-06\\
725	2.71200701783839e-06\\
730	2.68002071010849e-06\\
735	2.64853417114197e-06\\
740	2.61755868904162e-06\\
745	2.58709940653064e-06\\
750	2.55715731233446e-06\\
755	2.52773468250073e-06\\
760	2.49885985231078e-06\\
765.000000000001	2.47032971242481e-06\\
770.000000000001	2.44231279600756e-06\\
775	2.41483083059873e-06\\
780.000000000001	2.38782327572796e-06\\
785	2.36127828112889e-06\\
790.000000000001	2.33518605405944e-06\\
795	2.30953763852658e-06\\
800	2.28432156167138e-06\\
805	2.25953203116675e-06\\
810	2.23515837129851e-06\\
815	2.21119198464849e-06\\
820	2.18762424941801e-06\\
825.000000000001	2.16444680325036e-06\\
830	2.1416513426323e-06\\
835	2.11923000984211e-06\\
840.000000000001	2.09717485147591e-06\\
845.000000000001	2.07547848450116e-06\\
850	2.05413355130662e-06\\
855	2.03313309926293e-06\\
860.000000000001	2.01247040246232e-06\\
865.000000000001	1.99213920507483e-06\\
870.000000000001	1.97213357600695e-06\\
875	1.95244841134796e-06\\
880.000000000001	1.93307917993219e-06\\
885.000000000001	1.91402275159739e-06\\
890.000000000001	1.89527810389275e-06\\
895	1.87684777554645e-06\\
900.000000000001	1.85874198214199e-06\\
905	1.84097277078394e-06\\
910.000000000001	1.82366943577881e-06\\
915	1.80704823616501e-06\\
920.000000000001	1.79279670368202e-06\\
925	1.75639916056759e-06\\
930.000000000001	1.74706210806375e-06\\
935.000000000001	1.73310937412553e-06\\
940.000000000001	1.71803844715111e-06\\
945.000000000001	1.70261130533951e-06\\
950	1.68733748897748e-06\\
955.000000000001	1.67214904412358e-06\\
960.000000000001	1.65710348740725e-06\\
965	1.64222458062531e-06\\
970.000000000001	1.62752430219679e-06\\
975	1.61300800921877e-06\\
980.000000000001	1.59867822863335e-06\\
985.000000000001	1.5845354278906e-06\\
989.999999999999	1.57057883113741e-06\\
995.000000000001	1.55680707163838e-06\\
1000	1.5432181584115e-06\\
1005	1.52980987111968e-06\\
1010	1.5165797173823e-06\\
1015	1.50352508801402e-06\\
1020	1.49064338358299e-06\\
1025	1.47793185154826e-06\\
1030	1.46538778121558e-06\\
1035	1.45300841133401e-06\\
1040	1.44079105288795e-06\\
1045	1.42873299728661e-06\\
1050	1.41683161713151e-06\\
1055	1.40508425055188e-06\\
1060	1.39348867511136e-06\\
1065	1.38204005573516e-06\\
1070	1.37073960520713e-06\\
1075	1.35958275220716e-06\\
1080	1.34856731822251e-06\\
1085	1.3376908841297e-06\\
1090	1.32695118165145e-06\\
1095	1.31634594270962e-06\\
1100	1.30587282604868e-06\\
1105	1.29552975905508e-06\\
1110	1.28531449136538e-06\\
1115	1.27522726843745e-06\\
1120	1.26526411400802e-06\\
1125	1.25542264146849e-06\\
1130	1.24570104874722e-06\\
1135	1.23609759523803e-06\\
1140	1.22661469148183e-06\\
1145	1.21723404966342e-06\\
1150	1.20797185727935e-06\\
1155	1.19882058115075e-06\\
1160	1.1897783880119e-06\\
1165	1.18084368029605e-06\\
1170	1.17201498967067e-06\\
1175	1.16329125903703e-06\\
1180	1.15467157507027e-06\\
1185	1.14615601353798e-06\\
1190	1.13774626243808e-06\\
1195	1.12945089468435e-06\\
1200	1.12137886121973e-06\\
1205	1.11531582149454e-06\\
1210	1.10494097387709e-06\\
1215	1.09689216409115e-06\\
1220	1.08902442906722e-06\\
1225	1.08124292103095e-06\\
1230	1.07354758468202e-06\\
1235	1.06593726292939e-06\\
1240	1.05841136373761e-06\\
1245	1.05096837819354e-06\\
1250	1.04360712464265e-06\\
1255	1.03632649250442e-06\\
1260	1.02912528438272e-06\\
1265	1.02200233762641e-06\\
1270	1.01495649035e-06\\
1275	1.00798663823244e-06\\
1280	1.0010916249372e-06\\
1285	9.94270376888602e-07\\
1290	9.87521897055176e-07\\
1295	9.80845082533091e-07\\
1300	9.74238876603498e-07\\
1305	9.67703326892297e-07\\
1310	9.61236682030714e-07\\
1315	9.54837511454216e-07\\
1320	9.48505045218018e-07\\
1325	9.42238308684613e-07\\
1330	9.36036349844253e-07\\
1335	9.29898284347251e-07\\
1340	9.23823284232758e-07\\
1345	9.17810435228222e-07\\
1350	9.11859063290866e-07\\
1355	9.05969917707817e-07\\
1360	9.00133003012338e-07\\
1365	8.94360891510358e-07\\
1370	8.88646069445198e-07\\
1375	8.82988285952849e-07\\
1380	8.77386872876896e-07\\
1385	8.71841064754625e-07\\
1390	8.66350162552932e-07\\
1395	8.60913383474986e-07\\
1400	8.55530128502334e-07\\
1405	8.50199499332808e-07\\
1410	8.44920962136344e-07\\
1415	8.39693795290093e-07\\
1420	8.3451728878145e-07\\
1425	8.29390779183922e-07\\
1430	8.2431364381803e-07\\
1435	8.19285273847606e-07\\
1440	8.14305038691988e-07\\
1445	8.09372232940789e-07\\
1450	8.04486386644482e-07\\
1455	7.99646906856975e-07\\
1460	7.94853139459512e-07\\
1465	7.90104585438242e-07\\
1470	7.85400735345975e-07\\
1475	7.80741012919101e-07\\
1480	7.76125070301609e-07\\
1485	7.7155229366102e-07\\
1490	7.67022426245434e-07\\
1495	7.62535194893528e-07\\
1500	7.58090559539603e-07\\
1505	7.53688619496946e-07\\
1510	7.49330540188311e-07\\
1515	7.45019434424714e-07\\
1520	7.40766685478854e-07\\
1525	7.36660069324202e-07\\
1530	7.32230392411721e-07\\
1535	7.28006119091936e-07\\
1540	7.23858264624452e-07\\
1545	7.19752888867755e-07\\
1550	7.1568570806696e-07\\
1555	7.11655162014328e-07\\
1560	7.07660540636528e-07\\
1565	7.03701254803311e-07\\
1570	6.99776734610803e-07\\
1575	6.95886551503803e-07\\
1580	6.92030338537863e-07\\
1585	6.8820766493895e-07\\
1590	6.84418487151959e-07\\
1595	6.80659291546345e-07\\
1600	6.76935507780909e-07\\
1605	6.73242959207016e-07\\
1610	6.69581877435002e-07\\
1615	6.65951961188932e-07\\
1620	6.62352797803729e-07\\
1625	6.58784153649337e-07\\
1630	6.55245586684224e-07\\
1635	6.51736776725321e-07\\
1640	6.48257431160107e-07\\
1645	6.44807233939478e-07\\
1650	6.41385723225054e-07\\
1655	6.37992631458158e-07\\
1660	6.34627707273431e-07\\
1665	6.31290611962366e-07\\
1670	6.27980969584104e-07\\
1675	6.24698508532956e-07\\
1680	6.21442962457671e-07\\
1685	6.18213983876773e-07\\
1690	6.15011301033526e-07\\
1695	6.11834546983569e-07\\
1700	6.08683577932047e-07\\
1705	6.05557974620518e-07\\
1710	6.02457538334478e-07\\
1715	5.99381984270506e-07\\
1720	5.96331047544929e-07\\
1725	5.93304481452449e-07\\
1730	5.90301916191692e-07\\
1735	5.87323137298165e-07\\
1740	5.84367957189052e-07\\
1745	5.81436165806796e-07\\
1750	5.78527323462984e-07\\
1755	5.75641308672277e-07\\
1760	5.72777858344705e-07\\
1765	5.69936801316675e-07\\
1770	5.67117796241941e-07\\
1775	5.64320608998497e-07\\
1780	5.61545078545195e-07\\
1785	5.58790986809038e-07\\
1790	5.56058080401778e-07\\
1795	5.53346112136059e-07\\
1800	5.50654953748299e-07\\
1805	5.47984353406568e-07\\
1810	5.45333955092358e-07\\
1815	5.42703755624086e-07\\
1820	5.40093543026379e-07\\
1825	5.37503069494358e-07\\
1830	5.34932148175155e-07\\
1835	5.32380588501431e-07\\
1840	5.29848228172848e-07\\
1845	5.27335095684008e-07\\
1850	5.24840411267018e-07\\
1855	5.22365055662868e-07\\
1860	5.19908355824434e-07\\
1865	5.1747073459793e-07\\
1870	5.15052491242788e-07\\
1875	5.12655367444384e-07\\
1880	5.10286564343961e-07\\
1885	5.07786702650357e-07\\
1890	5.0546384798866e-07\\
1895	5.03144932337409e-07\\
1900	5.00833411962381e-07\\
1905	4.9853565218593e-07\\
1910	4.96253017804036e-07\\
1915	4.93986071026077e-07\\
1920	4.91734849000726e-07\\
1925	4.89499486072847e-07\\
1930	4.87279810059054e-07\\
1935	4.85075746641871e-07\\
1940	4.82887246020587e-07\\
1945	4.80714106814788e-07\\
1950	4.78556191502777e-07\\
1955	4.76413422722475e-07\\
1960	4.74285622174041e-07\\
1965	4.72172716877595e-07\\
1970	4.70074561072324e-07\\
1975	4.67990923952073e-07\\
1980	4.6592171984708e-07\\
1985	4.63866877524208e-07\\
1990	4.61826272450398e-07\\
1995	4.59799687495526e-07\\
2000	4.57787005301205e-07\\
2005	4.55788144768782e-07\\
2010	4.53802920019301e-07\\
2015	4.51831225448189e-07\\
2020	4.49872956861081e-07\\
2025	4.47927975481645e-07\\
2030	4.45996264386512e-07\\
2035	4.44077510771236e-07\\
2040	4.42171700537213e-07\\
2045	4.4027871453913e-07\\
2050	4.38398441545513e-07\\
2055	4.36530781914275e-07\\
2060	4.34675545622112e-07\\
2065	4.32832656443222e-07\\
2070	4.31001992378373e-07\\
2075	4.29183539774374e-07\\
2080	4.27377195938338e-07\\
2085	4.25582768731216e-07\\
2090	4.23800232686412e-07\\
2095	4.22029359942684e-07\\
2100	4.20270112360239e-07\\
2105	4.18522445811116e-07\\
2110	4.16786009330161e-07\\
2115	4.15060948305496e-07\\
2120	4.13347099337446e-07\\
2125	4.11644369531415e-07\\
2130	4.09952544342298e-07\\
2135	4.08271679603926e-07\\
2140	4.06601499983743e-07\\
2145	4.04942149964905e-07\\
2150	4.03293429836552e-07\\
2155	4.01655135076096e-07\\
2160	4.00027326785634e-07\\
2165	3.98409867462116e-07\\
2170	3.96802645070747e-07\\
2175	3.95205644596702e-07\\
2180	3.93618647371539e-07\\
2185	3.92041702567943e-07\\
2190	3.90474618866042e-07\\
2195	3.88917291539591e-07\\
2200	3.87369853636528e-07\\
2205	3.85831921440317e-07\\
2210	3.84303658106779e-07\\
2215	3.82784930568875e-07\\
2220	3.81275476820846e-07\\
2225	3.79775500935127e-07\\
2230	3.78284661529961e-07\\
2235	3.76803067662009e-07\\
2240	3.75330589418293e-07\\
2245	3.73867157037821e-07\\
2250	3.72412633013698e-07\\
2255	3.70967038597445e-07\\
2260	3.69530419771267e-07\\
2265	3.68102550225363e-07\\
2270	3.66683587269492e-07\\
2275	3.65273826114583e-07\\
2280	3.63877163795496e-07\\
2285	3.62475784524323e-07\\
2290	3.61088802004646e-07\\
2295	3.59711046632951e-07\\
2300	3.58341753436645e-07\\
2305	3.56980732963198e-07\\
2310	3.55627860888672e-07\\
2315	3.5428301020363e-07\\
2320	3.52946160312127e-07\\
2325	3.51617240967171e-07\\
2330	3.50296227507498e-07\\
2335	3.4898283654962e-07\\
2340	3.47677227158476e-07\\
2345	3.46379361120647e-07\\
2350	3.45089003878366e-07\\
2355	3.43806251261933e-07\\
2360	3.42530959090303e-07\\
2365	3.41263078256986e-07\\
2370	3.4000251420094e-07\\
2375	3.38749305843251e-07\\
2380	3.37503331942974e-07\\
2385	3.36264565055812e-07\\
2390	3.35032947262145e-07\\
2395	3.33808279140647e-07\\
2400	3.32590703779888e-07\\
2405	3.31380115446377e-07\\
2410	3.30176336853161e-07\\
2415	3.28979540305134e-07\\
2420	3.27789400152462e-07\\
2425	3.26606213400166e-07\\
2430	3.25429655130578e-07\\
2435	3.24259665053977e-07\\
2440	3.23096289079733e-07\\
2445	3.21939438318703e-07\\
2450	3.20789151717145e-07\\
2455	3.19645247185213e-07\\
2460	3.18507753207912e-07\\
2465	3.17376667092659e-07\\
2470	3.16251881247834e-07\\
2475	3.1513319048487e-07\\
2480	3.1402080175602e-07\\
2485	3.12914624665972e-07\\
2490	3.11814488240478e-07\\
2495	3.10720538149098e-07\\
2500	3.0963246025256e-07\\
2505	3.08550418399035e-07\\
2510	3.0747437133755e-07\\
2515	3.06404105841123e-07\\
2520	3.05339684118721e-07\\
2525	3.04281114675418e-07\\
2530	3.03228274088928e-07\\
2535	3.02181102933491e-07\\
2540	3.01139644544645e-07\\
2545	3.00103780294574e-07\\
2550	2.9907355386297e-07\\
2555	2.98048893097423e-07\\
2560	2.97029649658802e-07\\
2565	2.96015984206919e-07\\
2570	2.95007660030842e-07\\
2575	2.9400475975912e-07\\
2580	2.93007173371959e-07\\
2585	2.92014919621528e-07\\
2590	2.91027942201761e-07\\
2595	2.90046236014315e-07\\
2600	2.8906956293323e-07\\
2605	2.88098128766252e-07\\
2610	2.87131914273438e-07\\
2615	2.86170698316185e-07\\
2620	2.85214560051263e-07\\
2625	2.84263397823725e-07\\
2630	2.83317223801125e-07\\
2635	2.82376044447325e-07\\
2640	2.81439723291802e-07\\
2645	2.80508255919164e-07\\
2650	2.79581672423491e-07\\
2655	2.78659830834449e-07\\
2660	2.77742832175219e-07\\
2665	2.76830425716839e-07\\
2670	2.75922844898967e-07\\
2675	2.75019954155415e-07\\
2680	2.74121736031853e-07\\
2685	2.73228050363162e-07\\
2690	2.72338927110393e-07\\
2695	2.7145447941759e-07\\
2700	2.70574569744795e-07\\
2705	2.69699381281531e-07\\
2710	2.68829544022488e-07\\
2715	2.67951666907655e-07\\
2720	2.67096358497731e-07\\
2725	2.66238762600451e-07\\
2730	2.65385202646452e-07\\
2735	2.64535901392898e-07\\
2740	2.63690729879018e-07\\
2745	2.62849903009522e-07\\
2750	2.62013213604466e-07\\
2755	2.61180701163827e-07\\
2760	2.60352275037625e-07\\
2765	2.59528043672301e-07\\
2770	2.5870790741858e-07\\
2775	2.57891779878519e-07\\
2780	2.57079743230332e-07\\
2785	2.56271670418159e-07\\
2790	2.55467537444569e-07\\
2795	2.54667453782004e-07\\
2800	2.53871242407392e-07\\
2805	2.53078906120224e-07\\
2810	2.52290472338377e-07\\
2815	2.51505958575273e-07\\
2820	2.50725239728962e-07\\
2825	2.49948325831686e-07\\
2830	2.4917514755974e-07\\
2835	2.48405693773868e-07\\
2840	2.47640028301117e-07\\
2845	2.46878009989029e-07\\
2850	2.46119842433805e-07\\
2855	2.45365117594266e-07\\
2860	2.44614080648552e-07\\
2865	2.43866631400123e-07\\
2870	2.43122792030887e-07\\
2875	2.4238257247371e-07\\
2880	2.41645782036836e-07\\
2885	2.40912500678394e-07\\
2890	2.40182845325429e-07\\
2895	2.39456549142303e-07\\
2900	2.38733736463218e-07\\
2905	2.38014318539294e-07\\
2910	2.37298297274569e-07\\
2915	2.36585697968949e-07\\
2920	2.35876447796591e-07\\
2925	2.35170512022237e-07\\
2930	2.3446787021248e-07\\
2935	2.33768587333713e-07\\
2940	2.33072438055568e-07\\
2945	2.32379722665086e-07\\
2950	2.31690204724019e-07\\
2955	2.3100370091443e-07\\
2960	2.30320500695675e-07\\
2965	2.29640506347886e-07\\
2970	2.28963655341649e-07\\
2975	2.28289922690748e-07\\
2980	2.27619351064035e-07\\
2985	2.2695177866159e-07\\
2990	2.26287376757901e-07\\
2995	2.25625924839601e-07\\
3000	2.24967523075178e-07\\
3005	2.24312221826405e-07\\
3010	2.23659867210647e-07\\
3015	2.23010528915711e-07\\
3020	2.2236398962696e-07\\
3025	2.21720509142558e-07\\
3030	2.21079999300726e-07\\
3035	2.20442227070483e-07\\
3040	2.19807466786546e-07\\
3045	2.1917555712616e-07\\
3050	2.18546436932637e-07\\
3055	2.17920228841363e-07\\
3060	2.17296832892273e-07\\
3065	2.16676267288022e-07\\
3070	2.16058378304369e-07\\
3075	2.1544336570095e-07\\
3080	2.14831025661638e-07\\
3085	2.14221464789683e-07\\
3090	2.13614584603346e-07\\
3095	2.13010253459242e-07\\
3100	2.12408853301547e-07\\
3105	2.1181000403056e-07\\
3110	2.11213682807809e-07\\
3115	2.10620223135431e-07\\
3120	2.1002926834691e-07\\
3125	2.09440973146901e-07\\
3130	2.08855174485565e-07\\
3135	2.08271967287532e-07\\
3140	2.07691321740942e-07\\
3145	2.07113284899816e-07\\
3150	2.0653776149318e-07\\
3155	2.0596464689111e-07\\
3160	2.05394089097632e-07\\
3165	2.04826031774458e-07\\
3170	2.04260530754691e-07\\
3175	2.03697456098282e-07\\
3180	2.03137029256414e-07\\
3185	2.02583057526069e-07\\
3190	2.02020603750074e-07\\
3195	2.01467709151394e-07\\
3200	2.0091682376843e-07\\
3205	2.00368195881507e-07\\
3210	1.99821825354574e-07\\
3215	1.9927771032444e-07\\
3220	1.98736076853426e-07\\
3225	1.98196642200548e-07\\
3230	1.97659475836e-07\\
3235	1.97124615351147e-07\\
3240	1.96591971719517e-07\\
3245	1.9606176740039e-07\\
3250	1.95533745839815e-07\\
3255	1.95007780384108e-07\\
3260	1.94484258566693e-07\\
3265	1.93962671358123e-07\\
3270	1.93443407542889e-07\\
3275	1.92926408312683e-07\\
3280	1.92411482384335e-07\\
3285	1.91898783165789e-07\\
3290	1.91388223157098e-07\\
3295	1.90879683693831e-07\\
3300	1.90373400384623e-07\\
3305	1.89869074069836e-07\\
3310	1.89366927278559e-07\\
3315	1.88866867962442e-07\\
3320	1.88368908222956e-07\\
3325	1.87873038498292e-07\\
3330	1.87379083665292e-07\\
3335	1.86887266144291e-07\\
3340	1.86397482328639e-07\\
3345	1.85909584478312e-07\\
3350	1.85423917641796e-07\\
3355	1.84940036455954e-07\\
3360	1.84458203877226e-07\\
3365	1.83978321559056e-07\\
3370	1.83500450669109e-07\\
3375	1.83024510256507e-07\\
3380	1.82550448652117e-07\\
3385	1.82078432234121e-07\\
3390	1.81608147784358e-07\\
3395	1.81139909685856e-07\\
3400	1.80673511829798e-07\\
3405	1.8020904770631e-07\\
3410	1.79746519001007e-07\\
3415	1.79285594501562e-07\\
3420	1.78826669062614e-07\\
3425	1.78369697550672e-07\\
3430	1.77914501120886e-07\\
3435	1.77461121003049e-07\\
3440	1.77009424157047e-07\\
3445	1.76559636428897e-07\\
3450	1.76111710789868e-07\\
3455	1.75665389100877e-07\\
3460	1.75221021622607e-07\\
3465	1.74778350306585e-07\\
3470	1.74337328599054e-07\\
3475.00000000001	1.73898163952457e-07\\
3480	1.73460816691852e-07\\
3485	1.73025161732629e-07\\
3490	1.72591093535455e-07\\
3495	1.72158882167623e-07\\
3500	1.71728424956733e-07\\
3505	1.71299556074981e-07\\
3510	1.70872309285867e-07\\
3515	1.70446794430088e-07\\
3520.00000000001	1.70022887056425e-07\\
3525	1.6960074301484e-07\\
3530	1.6918018619422e-07\\
3535	1.68761369074881e-07\\
3540	1.68344071782178e-07\\
3545	1.67928395293425e-07\\
3550	1.67514339118031e-07\\
3555	1.67101849616628e-07\\
3560	1.66690942391976e-07\\
3565	1.66281715775162e-07\\
3570	1.6587410411749e-07\\
3575	1.65467899799704e-07\\
3580	1.65063487067375e-07\\
3585	1.64660336801403e-07\\
3590	1.642589750404e-07\\
3595	1.63859098412795e-07\\
3600	1.63460635363021e-07\\
3605	1.63063620700289e-07\\
3610	1.62668287867446e-07\\
3615	1.6227445738085e-07\\
3620	1.61882084267246e-07\\
3625	1.61491206334827e-07\\
3630	1.61101705985678e-07\\
3635	1.60713928943149e-07\\
3640	1.60327437531741e-07\\
3645	1.59942421589476e-07\\
3650	1.59558869932661e-07\\
3655	1.59176760449814e-07\\
3660	1.58796050371958e-07\\
3665	1.58416834653283e-07\\
3670	1.58039044371826e-07\\
3675	1.57662640888783e-07\\
3680	1.5728762502696e-07\\
3685	1.56914056813875e-07\\
3690	1.565419030527e-07\\
3695	1.56169099837424e-07\\
3700	1.55801248927323e-07\\
3705	1.55433247766071e-07\\
3710	1.55066671016713e-07\\
3715	1.54701138783369e-07\\
3720	1.54337172866282e-07\\
3725	1.53974585091645e-07\\
3730	1.53613202011437e-07\\
3735	1.53253386120442e-07\\
3740.00000000001	1.52894696180686e-07\\
3745	1.52537289608579e-07\\
3750	1.52181154248601e-07\\
3755	1.51826513131524e-07\\
3760	1.51473053618472e-07\\
3765	1.51120876172683e-07\\
3770.00000000001	1.50770012612718e-07\\
3775	1.50420369182621e-07\\
3780	1.50072060384074e-07\\
3785	1.49725133678385e-07\\
3790	1.49379254393581e-07\\
3795.00000000001	1.49034736488757e-07\\
3800	1.48691409706473e-07\\
3805	1.48349467516944e-07\\
3810	1.48008488739374e-07\\
3815	1.47669078927503e-07\\
3820.00000000001	1.47330486095034e-07\\
3825	1.46993515915509e-07\\
3830	1.46657595908596e-07\\
3835	1.4632283785641e-07\\
3840.00000000001	1.45989267901489e-07\\
3845	1.45656953274699e-07\\
3850	1.45325779586911e-07\\
3855	1.44995877565411e-07\\
3860	1.44667084538052e-07\\
3865	1.44339466645872e-07\\
3870	1.44013033537922e-07\\
3875	1.43687694904364e-07\\
3880	1.43363525369044e-07\\
3885	1.43040477886085e-07\\
3890	1.4271864402598e-07\\
3895	1.42397981283936e-07\\
3900	1.42078392781313e-07\\
3905	1.41759951330098e-07\\
3910	1.41442486039786e-07\\
3915	1.41126324318174e-07\\
3920	1.40811183044437e-07\\
3925	1.40497067861782e-07\\
3930	1.40184268995295e-07\\
3935	1.39872436260673e-07\\
3940	1.39561736966344e-07\\
3945	1.39251933060528e-07\\
3950	1.38943446154588e-07\\
3955	1.38635820040841e-07\\
3960	1.38329464522296e-07\\
3965	1.38024129567149e-07\\
3970	1.37719681387392e-07\\
3975	1.37416627007113e-07\\
3980	1.37114348653398e-07\\
3985	1.36813278794141e-07\\
3990	1.36513005496787e-07\\
3995	1.36213967856764e-07\\
4000	1.35915848978978e-07\\
4005	1.35618740217311e-07\\
4010	1.3532271145814e-07\\
4015	1.35027855169859e-07\\
4020	1.34733773622592e-07\\
4025	1.34440749365388e-07\\
4030.00000000001	1.34148850285323e-07\\
4035	1.33857815153378e-07\\
4040	1.33567638937902e-07\\
4045	1.3327877144165e-07\\
4050	1.32990590312414e-07\\
4055	1.32703611551367e-07\\
4060.00000000001	1.32417690375046e-07\\
4065	1.32132383747944e-07\\
4070	1.31848186991886e-07\\
4075	1.31564962668911e-07\\
4080.00000000001	1.31282777526045e-07\\
4085	1.31001403049891e-07\\
4090	1.30721144371561e-07\\
4095.00000000001	1.3044162729508e-07\\
4100	1.30163112496655e-07\\
4105	1.29885637603019e-07\\
4110	1.29609045153211e-07\\
4115	1.29333324906999e-07\\
4120	1.29058581928976e-07\\
4125	1.28784660337197e-07\\
4130.00000000001	1.28511816556125e-07\\
4135	1.28239806825639e-07\\
4140.00000000001	1.27968542682575e-07\\
4145.00000000001	1.27698411018465e-07\\
4150	1.27429045732264e-07\\
4155.00000000001	1.27160673176149e-07\\
4160	1.26893090275041e-07\\
4165	1.26626434463193e-07\\
4170	1.26360558579117e-07\\
4175	1.26095737190378e-07\\
4180	1.25831544903022e-07\\
4185	1.25568446170069e-07\\
4190	1.25306122690562e-07\\
4195.00000000001	1.25044705142992e-07\\
4200	1.24784058268783e-07\\
4205	1.24524301901276e-07\\
4210	1.24265377327606e-07\\
4215	1.24007380907786e-07\\
4220	1.23750128964186e-07\\
4225	1.23493685391991e-07\\
4230	1.23238235885699e-07\\
4235	1.22983427919789e-07\\
4240	1.22729583197678e-07\\
4245.00000000001	1.22476613316819e-07\\
4250	1.22224364179472e-07\\
4255	1.21972873882029e-07\\
4260	1.21722471627232e-07\\
4265.00000000001	1.21472518665961e-07\\
4270	1.21223525849394e-07\\
4275	1.20975330101333e-07\\
4280	1.20727851804835e-07\\
4285	1.20481073511706e-07\\
4290.00000000001	1.20235501294276e-07\\
4295.00000000001	1.19990427001825e-07\\
4300.00000000001	1.19746186197739e-07\\
4305	1.19502645004673e-07\\
4310	1.19260172664488e-07\\
4315	1.19018086892981e-07\\
4320	1.18777107812535e-07\\
4325	1.18536577317655e-07\\
4330	1.18296980658589e-07\\
4335	1.18058041985195e-07\\
4340.00000000001	1.17820089150356e-07\\
4345.00000000001	1.17582491523217e-07\\
4350.00000000001	1.17345932144166e-07\\
4355.00000000001	1.17110048898137e-07\\
4360	1.16875072935769e-07\\
4365.00000000001	1.16640702999211e-07\\
4370	1.16407032918675e-07\\
4375	1.16174196732522e-07\\
4380	1.15941958841499e-07\\
4385.00000000001	1.1571062591705e-07\\
4390	1.15479836157711e-07\\
4395.00000000001	1.1524988183886e-07\\
4400	1.1502068478849e-07\\
4405	1.14792142031932e-07\\
4410.00000000001	1.14564077527283e-07\\
4415	1.1433697702164e-07\\
4420	1.14110626950785e-07\\
4425.00000000001	1.13884965595513e-07\\
4430	1.1365996669094e-07\\
4435	1.13435553596921e-07\\
4440	1.13211975999423e-07\\
4445	1.12989050524335e-07\\
4450	1.1276686239882e-07\\
4455	1.125452106224e-07\\
4460	1.12324286093097e-07\\
4465	1.12104106080555e-07\\
4470	1.11884706106242e-07\\
4475	1.11665761123563e-07\\
4480	1.11447614835147e-07\\
4485	1.11230126350767e-07\\
4490	1.11013287127547e-07\\
4495	1.10797159853429e-07\\
4500	1.10581474988784e-07\\
4505	1.10366734829519e-07\\
4510	1.10152573657172e-07\\
4515	1.09938942476071e-07\\
4520	1.09726082132416e-07\\
4525	1.09513843289481e-07\\
4530	1.09302322771865e-07\\
4535	1.0909125472103e-07\\
4540	1.08880970743753e-07\\
4545	1.08671362804571e-07\\
4550	1.0846220206784e-07\\
4555	1.08253866065902e-07\\
4560	1.08045932154073e-07\\
4565	1.07838854135491e-07\\
4570	1.07632188925202e-07\\
4575	1.07426395638821e-07\\
4580	1.07221036727926e-07\\
4585	1.07016269844168e-07\\
4590	1.06812322908809e-07\\
4595.00000000001	1.06608877056677e-07\\
4600	1.06405913048003e-07\\
4605.00000000001	1.06203648241824e-07\\
4610	1.06001992417369e-07\\
4615	1.05801059248174e-07\\
4620.00000000001	1.05600401209628e-07\\
4625	1.05400529585904e-07\\
4630	1.0520136069316e-07\\
4635	1.05002796499045e-07\\
4640	1.04804707493951e-07\\
4645	1.04607129059467e-07\\
4650.00000000001	1.04410382561346e-07\\
4655	1.0421387454703e-07\\
4660	1.04018143641876e-07\\
4665.00000000001	1.03822965455986e-07\\
4670.00000000001	1.0362835097476e-07\\
4675	1.03434476943598e-07\\
4680.00000000001	1.03240941525357e-07\\
4685	1.03048087753066e-07\\
4690.00000000001	1.02855764106482e-07\\
4695	1.0266394553355e-07\\
4700.00000000001	1.02472876321775e-07\\
4705	1.02282210106869e-07\\
4710	1.02092184049833e-07\\
4715	1.01902644398997e-07\\
4720.00000000001	1.0171356158091e-07\\
4725.00000000001	1.01525117927901e-07\\
4730	1.01337327037875e-07\\
4735	1.01149953321852e-07\\
4740.00000000001	1.00963186349163e-07\\
4745	1.00776864157888e-07\\
4750	1.00591204848424e-07\\
4755	1.00406029263657e-07\\
4760	1.00221435924281e-07\\
4765.00000000001	1.00037387393227e-07\\
4770	9.98537678367718e-08\\
4775	9.96707033429209e-08\\
4780	9.94882287914447e-08\\
4785	9.93060920776417e-08\\
4790	9.91248177432457e-08\\
4795	9.89438585466491e-08\\
4800	9.87634542090152e-08\\
4805	9.85834323253961e-08\\
4810	9.84039902897199e-08\\
4815	9.82252165869374e-08\\
4820	9.8046775877197e-08\\
4825	9.786884459693e-08\\
4830	9.76914931630132e-08\\
4835	9.75148033689213e-08\\
4840	9.73383194157975e-08\\
4845	9.71624226557878e-08\\
4850	9.69870863868601e-08\\
4855	9.68123157141768e-08\\
4860	9.66380157096227e-08\\
4865	9.64641112341575e-08\\
4870	9.62906075828573e-08\\
4875	9.61178034999998e-08\\
4880	9.59455153206315e-08\\
4885	9.57734472006157e-08\\
4890.00000000001	9.56020411047443e-08\\
4895	9.54312111246057e-08\\
4900	9.52604240831734e-08\\
4905.00000000001	9.50905167796426e-08\\
4910	9.49211418793021e-08\\
4915.00000000001	9.4752153942969e-08\\
4920	9.45834257215509e-08\\
4925	9.44154170385941e-08\\
4930.00000000001	9.42477543833754e-08\\
4935	9.40806680578061e-08\\
4940	9.39139276505719e-08\\
4945	9.37477221678397e-08\\
4950.00000000001	9.35817160984777e-08\\
4955	9.34165347158311e-08\\
4960.00000000001	9.32516248579619e-08\\
4965	9.30872648906772e-08\\
4970	9.29233127584494e-08\\
4975.00000000001	9.27599143197556e-08\\
4980	9.25969549068208e-08\\
4985	9.24342110842757e-08\\
4990	9.22722203526697e-08\\
4995	9.21103668830511e-08\\
5000	9.19491955522524e-08\\
5005.00000000001	9.17882810564585e-08\\
5010	9.16280259110742e-08\\
5015.00000000001	9.14681219688977e-08\\
5020	9.13086394594619e-08\\
5025.00000000001	9.114961588561e-08\\
5030	9.09908149774553e-08\\
5035	9.0832731281247e-08\\
5040.00000000001	9.067502600209e-08\\
5045	9.05176585824524e-08\\
5050.00000000001	9.03608397452976e-08\\
5055	9.02044627337078e-08\\
5060.00000000001	9.00483660567879e-08\\
5065	8.98926766008405e-08\\
5070.00000000001	8.97375712496111e-08\\
5075	8.95829212609505e-08\\
5080.00000000001	8.94285257739327e-08\\
5085.00000000001	8.92743308005407e-08\\
5090.00000000001	8.91208925905768e-08\\
5095	8.89678279752384e-08\\
5100	8.88151055213117e-08\\
5105	8.86626817155693e-08\\
5110	8.85108177871373e-08\\
5115	8.83594607903098e-08\\
5120.00000000001	8.8208324101562e-08\\
5125	8.80578572960893e-08\\
5130	8.79074968570338e-08\\
5135.00000000001	8.77577341838016e-08\\
5140.00000000001	8.7608179187673e-08\\
5145	8.74592408331338e-08\\
5150.00000000001	8.7310674903363e-08\\
5155	8.71624391913226e-08\\
5160	8.70143985761983e-08\\
5165	8.68670223045092e-08\\
5170	8.67198839316796e-08\\
5175	8.65732786505968e-08\\
5180	8.64269742443307e-08\\
5185	8.62809900640419e-08\\
5190	8.61355477393663e-08\\
5195.00000000001	8.59902795545955e-08\\
5200	8.58455920494154e-08\\
5205	8.57011871905262e-08\\
5210.00000000001	8.55572579134369e-08\\
5215	8.54138579154564e-08\\
5220	8.527040319435e-08\\
5225	8.5127535604016e-08\\
5230	8.49852317109799e-08\\
5235	8.48430967892264e-08\\
5240	8.47014251729422e-08\\
5245	8.45600463643221e-08\\
5250	8.44190673153146e-08\\
5255	8.42783237842538e-08\\
5260	8.41382312577237e-08\\
5265	8.39982405832826e-08\\
5270.00000000001	8.38586881971131e-08\\
5275	8.37196667971576e-08\\
5280	8.35807232182955e-08\\
5285	8.3442293138375e-08\\
5290	8.3304383720295e-08\\
5295	8.31666463332374e-08\\
5300	8.30292216580914e-08\\
5305	8.28922715120444e-08\\
5310	8.27558080480891e-08\\
5315	8.2619393495089e-08\\
5320	8.24833631434167e-08\\
5325	8.23478427688175e-08\\
5330	8.22125008897679e-08\\
5335.00000000001	8.20778041175656e-08\\
5340	8.19431325111573e-08\\
5345	8.18090278243526e-08\\
5350	8.16750481988086e-08\\
5355.00000000001	8.15417309987908e-08\\
5360.00000000001	8.14086362932272e-08\\
5365	8.12757982233685e-08\\
5370	8.11431115619343e-08\\
5375	8.10111462556076e-08\\
5380.00000000001	8.08793383765662e-08\\
5385	8.07477148583714e-08\\
5390.00000000001	8.06166434361013e-08\\
5395	8.04859213663382e-08\\
5400.00000000001	8.03554433275516e-08\\
5405.00000000001	8.02253981276358e-08\\
5410	8.0095457156669e-08\\
5415	7.99659512045939e-08\\
5420.00000000001	7.98369165614822e-08\\
5425	7.97081603646918e-08\\
5430	7.95796027575046e-08\\
5435	7.9451306688172e-08\\
5440.00000000001	7.93235137924294e-08\\
5445	7.91959887306378e-08\\
5450	7.90687920056256e-08\\
5455	7.89420524873554e-08\\
5460.00000000001	7.88154779574768e-08\\
5465.00000000001	7.86892929311163e-08\\
5470	7.85631910578301e-08\\
5475	7.84376373842894e-08\\
5480	7.83122948918406e-08\\
5485.00000000001	7.8187438819087e-08\\
5490.00000000001	7.80627730172038e-08\\
5495	7.79382553171762e-08\\
5500	7.78144477182458e-08\\
5505.00000000001	7.76906691544205e-08\\
5510	7.75671640643301e-08\\
5515	7.74441985513338e-08\\
5520.00000000001	7.73213531439321e-08\\
5525	7.71988492370236e-08\\
5530.00000000001	7.70766992390784e-08\\
5535.00000000001	7.69548208877667e-08\\
5540.00000000001	7.68331853412903e-08\\
5545.00000000001	7.67118599689804e-08\\
5550	7.65910192283071e-08\\
5555	7.64702209982615e-08\\
5560	7.6349880522687e-08\\
5565	7.62297160260675e-08\\
5570.00000000001	7.61097959365235e-08\\
5575	7.5990490875426e-08\\
5580	7.58713025242816e-08\\
5585	7.57523205235815e-08\\
5590	7.56337441719151e-08\\
5595	7.55155488627423e-08\\
5600	7.53973918894705e-08\\
5605	7.52795307826987e-08\\
5610	7.51621840986351e-08\\
5615.00000000001	7.50449908654964e-08\\
5620	7.49282688804314e-08\\
5625	7.48114434477249e-08\\
5630	7.46953202473659e-08\\
5635	7.45791296915439e-08\\
5640	7.44633550291564e-08\\
5645	7.43481355515108e-08\\
5650	7.42329462461811e-08\\
5655	7.41178636773e-08\\
5660.00000000001	7.40032090045984e-08\\
5665	7.38890628500701e-08\\
5670	7.3774947779141e-08\\
5675.00000000001	7.36610716566358e-08\\
5680	7.35476154828045e-08\\
5685	7.3434219203189e-08\\
5690.00000000001	7.33214542429322e-08\\
5695.00000000001	7.32089112161174e-08\\
5700	7.30963105323511e-08\\
5705	7.29841785838009e-08\\
5710	7.28722167048373e-08\\
5715.00000000001	7.27605971450686e-08\\
5720.00000000001	7.26492630863173e-08\\
5725	7.25382911410409e-08\\
5730	7.24272586615065e-08\\
5735	7.23169210068289e-08\\
5740	7.22067505235503e-08\\
5745	7.20966092911048e-08\\
5750	7.19868296745444e-08\\
5755	7.18773530978193e-08\\
5760	7.17681470086967e-08\\
5765	7.16590798628794e-08\\
5770.00000000001	7.15504042069188e-08\\
5775.00000000001	7.14419515258645e-08\\
5780.00000000001	7.13338391200616e-08\\
5785.00000000001	7.12260636893574e-08\\
5790	7.11180766150494e-08\\
5795	7.101073127559e-08\\
5800.00000000001	7.09035759683225e-08\\
5805	7.07967558853593e-08\\
5810	7.06900589439912e-08\\
5815	7.05836141113183e-08\\
5820.00000000001	7.04775197900762e-08\\
5825	7.0371716832099e-08\\
5830.00000000001	7.02657584344793e-08\\
5835.00000000001	7.01605847297515e-08\\
5840	7.00554108244766e-08\\
5845	6.99502650208755e-08\\
5850.00000000001	6.98456813097549e-08\\
5855	6.97414092059783e-08\\
5860.00000000001	6.96372707526424e-08\\
5865.00000000001	6.95334228125004e-08\\
5870.00000000001	6.9429668802688e-08\\
5875	6.93262330279639e-08\\
5880.00000000001	6.92231286013029e-08\\
5885	6.9120014855284e-08\\
5890	6.90171420287142e-08\\
5895.00000000001	6.89148188795754e-08\\
5900.00000000001	6.88124473137555e-08\\
5905.00000000001	6.87105821149221e-08\\
5910.00000000001	6.86086595775884e-08\\
5915.00000000001	6.85072848613828e-08\\
5920	6.84060068830586e-08\\
5925.00000000001	6.83049810761741e-08\\
5930.00000000001	6.82040432098721e-08\\
5935	6.81034250836218e-08\\
5940	6.80032074147329e-08\\
5945	6.79030120938629e-08\\
5950	6.78031019891788e-08\\
5955	6.77034339334524e-08\\
5960	6.76040285078729e-08\\
5965.00000000001	6.75049707515822e-08\\
5970	6.7405850239763e-08\\
5975	6.73073631276191e-08\\
5980.00000000001	6.72087710678927e-08\\
5985	6.71103062029491e-08\\
5990	6.70124183360464e-08\\
5995.00000000001	6.69144651172373e-08\\
6000	6.68169799681047e-08\\
};
\addlegendentry{EQRF1 G}
  \addplot [color=black, line width=0.9pt, forget plot]
  table[row sep=crcr]{%
200	8.565008204102601e-06\\
6000	2.227232665603490e-08\\
};
\node[] at (axis cs: 650, 1e-7) {\small $\zeta=\rmi$};
\node[] at (axis cs: 650, 3e-8) {\small $v(x)=\frac{1}{2+\cos(2\pi x)}$};

\end{axis}

\end{tikzpicture}%

%% file: plot_sc_nob.tex
%
%
\begin{tikzpicture}

\begin{axis}[%
width=1.7in,
height=1.7in,
at={(0.769in,0.477in)},
scale only axis,
xmode=log,
xmin=100,
xmax=10000,
xminorticks=true,
ymode=log,
ymin=1e-08,
ymax=0.0001,
yminorticks=true,
axis background/.style={fill=white},
xlabel = {$N$},
ylabel = {Error}
]
\addplot [color=blue, only marks, mark=x, mark options={solid, blue}, forget plot]
  table[row sep=crcr]{%
200	4.15543178673231e-05\\
205	3.57006653779513e-05\\
210	3.38536733800506e-05\\
215	3.25527302873696e-05\\
220	3.08644348600216e-05\\
225	3.30940113983169e-05\\
230	3.50441221059137e-05\\
235	5.8966656999998e-05\\
240	3.01225383212626e-05\\
245	2.84223573212039e-05\\
250	2.54505434389596e-05\\
255	2.50242556678084e-05\\
260	2.49034947323558e-05\\
265	2.33750620495502e-05\\
270	2.29247276583811e-05\\
275	2.28978203347881e-05\\
280	2.29773853154729e-05\\
285	2.32641132336557e-05\\
290	2.68177517378636e-05\\
295	3.07867247272195e-05\\
300	7.12177547304067e-05\\
305	4.83896459008731e-05\\
310	3.12675600364813e-05\\
315	2.44000261066529e-05\\
320	2.04284062387691e-05\\
325	1.86040658817953e-05\\
330	1.7324241985623e-05\\
335	1.79247566944538e-05\\
340	3.68786637191443e-05\\
345	1.78772101790785e-05\\
350	1.61029917890292e-05\\
355	1.44841310985676e-05\\
360	1.4005791689261e-05\\
365	1.32912899593622e-05\\
370	1.23810912997599e-05\\
375	1.18529040303432e-05\\
380	1.44590381141797e-05\\
385	1.15941180876363e-05\\
390	1.11492950491276e-05\\
395	1.11152579194781e-05\\
400	1.21934640886339e-05\\
405.000000000001	1.07285154162852e-05\\
410	1.02464342432922e-05\\
415	9.54425244487002e-06\\
420	9.39261223031946e-06\\
425	9.30719306581133e-06\\
430.000000000001	9.4028336274495e-06\\
435.000000000001	9.5228490079041e-06\\
440	9.54534050032272e-06\\
445	9.75324839848045e-06\\
450	1.03759472422703e-05\\
455	1.28178822822737e-05\\
460	1.29185290131276e-05\\
465.000000000001	1.75503184516869e-05\\
470.000000000001	4.12425211329679e-05\\
475	1.61496874058784e-05\\
480	1.08782746262369e-05\\
485	9.44137676795144e-06\\
490	8.6443786141367e-06\\
495.000000000001	7.95245378126456e-06\\
500	7.53586755082338e-06\\
505.000000000001	7.12470669756677e-06\\
510	8.74937031611601e-06\\
515.000000000001	7.4457632237708e-06\\
520	7.17846867087199e-06\\
525	6.80241331090397e-06\\
530	7.77547244665737e-06\\
535	6.56532518253373e-06\\
540.000000000001	6.7971985063736e-06\\
545	6.67660770628854e-06\\
550	5.9537443660553e-06\\
555	6.190950607782e-06\\
560	5.70538272325422e-06\\
565	5.63729646870987e-06\\
570	7.86920250157554e-06\\
575	5.72050759145249e-06\\
580.000000000001	6.94735216584257e-06\\
585.000000000001	6.08032889287576e-06\\
590.000000000001	5.71087271855791e-06\\
595	5.97250227084724e-06\\
600	7.33392536094646e-06\\
605.000000000001	8.81704928128029e-06\\
610	6.48451429924643e-06\\
615	5.72564735251749e-06\\
620.000000000001	5.58238754967494e-06\\
625.000000000001	5.33681605635067e-06\\
630.000000000001	5.66299064711998e-06\\
635	5.17137650453777e-06\\
640	5.19560995666668e-06\\
645	5.03080634200211e-06\\
650	5.07793335915038e-06\\
655	5.24294541273318e-06\\
660.000000000001	5.55615974330697e-06\\
665	6.16211349345943e-06\\
670.000000000001	7.54643982397267e-06\\
675.000000000001	1.22882648962515e-05\\
680	2.57931387944601e-05\\
685	9.57043425550968e-06\\
690.000000000001	7.55426639819719e-06\\
695	6.13609531243483e-06\\
700	5.49691103479077e-06\\
705.000000000001	5.13693612588276e-06\\
710	4.85089173056562e-06\\
715.000000000001	4.66687480347188e-06\\
720.000000000001	4.47015170071653e-06\\
725	4.27660151054565e-06\\
730	4.22633861287276e-06\\
735	4.20401283118233e-06\\
740	4.01622481750904e-06\\
745	3.76879244046219e-06\\
750	3.78937111589434e-06\\
755	3.89783824041367e-06\\
760	5.70159293742361e-06\\
765.000000000001	5.29761057464473e-06\\
770.000000000001	3.49976074399915e-06\\
775	3.39664100929533e-06\\
780.000000000001	3.34916132444543e-06\\
785	3.19792770919232e-06\\
790.000000000001	3.17271642329014e-06\\
795	3.52927099748075e-06\\
800	3.3525153250964e-06\\
805	3.61548394030413e-06\\
810	3.22496724859314e-06\\
815	3.09259060107351e-06\\
820	3.04377511587705e-06\\
825.000000000001	2.99524126639604e-06\\
830	2.92980670269523e-06\\
835	3.09902372478857e-06\\
840.000000000001	2.86361150258207e-06\\
845.000000000001	2.8714294989786e-06\\
850	2.90895853000485e-06\\
855	2.81436956526896e-06\\
860.000000000001	2.79781113153739e-06\\
865.000000000001	2.7733818582755e-06\\
870.000000000001	2.79977943237439e-06\\
875	2.73694307995256e-06\\
880.000000000001	2.67354017343834e-06\\
885.000000000001	2.65693682979743e-06\\
890.000000000001	2.69078183511493e-06\\
895	2.78923184262089e-06\\
900.000000000001	2.92897374422946e-06\\
905	6.70547840517368e-06\\
910.000000000001	3.58408987043577e-06\\
915	4.32148477461891e-06\\
920.000000000001	7.0027482760469e-06\\
925	1.62826535316243e-05\\
930.000000000001	6.23878240740318e-06\\
935.000000000001	4.81196137005005e-06\\
940.000000000001	4.50627277877544e-06\\
945.000000000001	3.3648000383786e-06\\
950	3.24565869044522e-06\\
955.000000000001	3.01727070468317e-06\\
960.000000000001	2.94310196298503e-06\\
965	5.5603816369805e-06\\
970.000000000001	2.59406891681797e-06\\
975	2.55662522785886e-06\\
980.000000000001	2.46687491557392e-06\\
985.000000000001	2.41232171434327e-06\\
989.999999999999	2.37438027855459e-06\\
995.000000000001	2.34545547702144e-06\\
1000	2.30324923059427e-06\\
1005	2.24966287041872e-06\\
1010	2.26428157682999e-06\\
1015	2.21989166178247e-06\\
1020	2.16690052230053e-06\\
1025	2.13640909716997e-06\\
1030	2.11201452215808e-06\\
1035	2.09672915405809e-06\\
1040	2.17884601675262e-06\\
1045	2.06817164121245e-06\\
1050	2.06629661012668e-06\\
1055	2.11036242257658e-06\\
1060	2.78825666699562e-06\\
1065	2.29056711198297e-06\\
1070	2.05367956394101e-06\\
1075	2.00508983148343e-06\\
1080	1.918412004425e-06\\
1085	1.93949262033787e-06\\
1090	2.54841566646894e-06\\
1095	2.00702492049677e-06\\
1100	1.97600152002314e-06\\
1105	1.96093981217585e-06\\
1110	1.89241288177763e-06\\
1115	1.89049037234216e-06\\
1120	1.8924042593634e-06\\
1125	1.89359136994357e-06\\
1130	1.90903080795007e-06\\
1135	2.00595752061337e-06\\
1140	4.20986288729438e-06\\
1145	2.11179352410289e-06\\
1150	1.94790881246803e-06\\
1155	1.99145228493412e-06\\
1160	1.99560528800137e-06\\
1165	2.03879937485234e-06\\
1170	2.10114539024449e-06\\
1175	2.16308879331008e-06\\
1180	2.27168647957468e-06\\
1185	2.40556174769741e-06\\
1190	2.66791675409186e-06\\
1195	3.06385342495458e-06\\
1200	4.05512836777496e-06\\
1205	1.07877734340288e-05\\
1210	5.13328639325576e-06\\
1215	3.25877056877417e-06\\
1220	2.81055932177991e-06\\
1225	2.2482004588323e-06\\
1230	2.26258813320099e-06\\
1235	2.30568703585583e-06\\
1240	2.07155928591693e-06\\
1245	1.95365650362992e-06\\
1250	1.87953018963407e-06\\
1255	1.88344165464813e-06\\
1260	1.7547663701986e-06\\
1265	1.71749951337362e-06\\
1270	1.68180705779163e-06\\
1275	1.68567518793909e-06\\
1280	1.63321142378196e-06\\
1285	1.61735971778194e-06\\
1290	1.58064595351795e-06\\
1295	1.57363577092549e-06\\
1300	1.54139232539085e-06\\
1305	1.90636471393593e-06\\
1310	1.49267575641205e-06\\
1315	1.50096331867737e-06\\
1320	1.51868882480191e-06\\
1325	1.49547668065911e-06\\
1330	1.49194928769704e-06\\
1335	1.50090098305322e-06\\
1340	1.52157965920817e-06\\
1345	1.55700895423774e-06\\
1350	1.68848743143015e-06\\
1355	2.31849775652637e-06\\
1360	1.94867234104811e-06\\
1365	1.4557520173023e-06\\
1370	1.36113557364145e-06\\
1375	1.36813118821569e-06\\
1380	1.31990712744481e-06\\
1385	1.35696733759989e-06\\
1390	1.27632192994046e-06\\
1395	1.23285562877595e-06\\
1400	1.21865120779127e-06\\
1405	1.2006291614593e-06\\
1410	1.24095185063247e-06\\
1415	1.36123753090974e-06\\
1420	1.23608877100365e-06\\
1425	1.33572737667059e-06\\
1430	1.19986878251996e-06\\
1435	1.1983782289031e-06\\
1440	1.18480187207734e-06\\
1445	1.18684095483692e-06\\
1450	1.17038280585944e-06\\
1455	1.16108124169345e-06\\
1460	1.1657125974216e-06\\
1465	1.15942695250912e-06\\
1470	1.15995821330107e-06\\
1475	1.17160421375787e-06\\
1480	1.17567052290387e-06\\
1485	1.20239321384435e-06\\
1490	1.18381052513155e-06\\
1495	1.20159698668214e-06\\
1500	1.22018240893025e-06\\
1505	1.25383030197945e-06\\
1510	1.47107735573443e-06\\
1515	1.53064605580799e-06\\
1520	1.88888662828742e-06\\
1525	3.56171860906115e-06\\
1530	3.34994890671824e-06\\
1535	1.9957665987377e-06\\
1540	1.61879358818103e-06\\
1545	1.46124273772073e-06\\
1550	1.34394514255264e-06\\
1555	1.3441344552087e-06\\
1560	1.28043217743413e-06\\
1565	1.24015625625697e-06\\
1570	1.21554278361123e-06\\
1575	1.18679790183456e-06\\
1580	1.18365794823316e-06\\
1585	1.23317807565916e-06\\
1590	1.51024161549237e-06\\
1595	1.6909175163098e-06\\
1600	1.24341729367486e-06\\
1605	1.1295167029211e-06\\
1610	1.57448930156255e-06\\
1615	1.26076086772927e-06\\
1620	1.1934873794587e-06\\
1625	1.18210074393347e-06\\
1630	1.12968232774188e-06\\
1635	1.11508171221812e-06\\
1640	1.10564977657233e-06\\
1645	1.10548111199917e-06\\
1650	1.07076082943936e-06\\
1655	1.07559035749386e-06\\
1660	1.09289156968459e-06\\
1665	1.0034376024971e-06\\
1670	1.01456845460367e-06\\
1675	1.04749687269703e-06\\
1680	9.9444816129477e-07\\
1685	1.00714904563454e-06\\
1690	9.86458903371768e-07\\
1695	9.92665040104543e-07\\
1700	1.06013549979569e-06\\
1705	9.6729971384052e-07\\
1710	9.45694629403447e-07\\
1715	9.25411700733617e-07\\
1720	9.69236903044905e-07\\
1725	9.49776365846666e-07\\
1730	9.68098072204377e-07\\
1735	1.09191141862864e-06\\
1740	9.42216068026183e-07\\
1745	9.45459096348662e-07\\
1750	9.38502397143555e-07\\
1755	9.31809629200274e-07\\
1760	9.11210925793329e-07\\
1765	9.04524040539049e-07\\
1770	8.99249551940403e-07\\
1775	9.03836212195e-07\\
1780	8.93752983511969e-07\\
1785	8.91104162483206e-07\\
1790	8.93756303279304e-07\\
1795	8.8928376566e-07\\
1800	8.90647468513065e-07\\
1805	8.90938814392193e-07\\
1810	1.15008073539752e-06\\
1815	1.36096274705884e-06\\
1820	1.0279805612718e-06\\
1825	9.12258128109516e-07\\
1830	8.62602966247649e-07\\
1835	8.9906172224619e-07\\
1840	9.77613731079346e-07\\
1845	1.33044030723923e-06\\
1850	1.36804612580663e-06\\
1855	1.08851753590516e-06\\
1860	1.04703643270379e-06\\
1865	1.09300101367783e-06\\
1870	1.11026093626591e-06\\
1875	1.22080520916452e-06\\
1880	1.52591166148965e-06\\
1885	1.90084018655245e-05\\
1890	1.61083639830384e-06\\
1895	1.26115104176151e-06\\
1900	1.11956825830542e-06\\
1905	1.04813779138072e-06\\
1910	9.98555714097003e-07\\
1915	9.62211996410913e-07\\
1920	9.40763646414262e-07\\
1925	9.17145217503958e-07\\
1930	9.26193915553609e-07\\
1935	8.70715015777086e-07\\
1940	8.589005812505e-07\\
1945	8.44795128150836e-07\\
1950	8.41184539859139e-07\\
1955	8.34136488393281e-07\\
1960	8.36350011699524e-07\\
1965	8.31278953092619e-07\\
1970	7.97807800934268e-07\\
1975	7.85703892346281e-07\\
1980	7.71097262120998e-07\\
1985	1.06527609522403e-06\\
1990	7.63128657912073e-07\\
1995	9.62863863953926e-07\\
2000	8.15792517997681e-07\\
2005	8.02090449088438e-07\\
2010	7.92670344462168e-07\\
2015	8.02717860585957e-07\\
2020	8.02401031342699e-07\\
2025	8.05880978664624e-07\\
2030	8.36167291866658e-07\\
2035	1.21259006199835e-06\\
2040	7.25399981640762e-07\\
2045	7.33890183079696e-07\\
2050	7.56022653444716e-07\\
2055	7.25130856484973e-07\\
2060	7.21076371985011e-07\\
2065	7.20921488288186e-07\\
2070	7.1347633968631e-07\\
2075	7.12696009029414e-07\\
2080	6.96869512083363e-07\\
2085	7.12908673816398e-07\\
2090	7.19249129138745e-07\\
2095	7.2413400194659e-07\\
2100	7.29898778402424e-07\\
2105	7.45854273799541e-07\\
2110	7.5636497944427e-07\\
2115	8.21090116119811e-07\\
2120	1.73012214524166e-06\\
2125	7.31319331263908e-07\\
2130	7.10320589965255e-07\\
2135	6.99022726247292e-07\\
2140	7.01186692033537e-07\\
2145	6.97959504001412e-07\\
2150	6.91284549467267e-07\\
2155	7.11622839824443e-07\\
2160	6.97675933435466e-07\\
2165	7.07407147201733e-07\\
2170	7.64035521137649e-07\\
2175	6.58969603542044e-07\\
2180	6.72960319137686e-07\\
2185	6.74753512348718e-07\\
2190	6.74326102877577e-07\\
2195	6.79433679761148e-07\\
2200	6.80841287198919e-07\\
2205	7.02348061840491e-07\\
2210	6.85208889917382e-07\\
2215	6.87663329278517e-07\\
2220	6.93656513833076e-07\\
2225	7.05154683566278e-07\\
2230	7.11321452577421e-07\\
2235	7.17391229640587e-07\\
2240	7.31568971017037e-07\\
2245	7.62727647227608e-07\\
2250	7.61379279938232e-07\\
2255	7.87615177409287e-07\\
2260	8.21570086095926e-07\\
2265	9.68252244841965e-07\\
2270	1.00221193777795e-06\\
2275	1.2467705463718e-06\\
2280	3.54394319365789e-06\\
2285	1.46681995152701e-06\\
2290	9.81697622754319e-07\\
2295	7.37344276999216e-07\\
2300	6.40693838662591e-07\\
2305	5.91230331541923e-07\\
2310	5.57137916350174e-07\\
2315	5.20931515357527e-07\\
2320	5.69725007924393e-07\\
2325	4.87510095833456e-07\\
2330	4.84337872040113e-07\\
2335	4.87581537381448e-07\\
2340	4.96138349752051e-07\\
2345	4.8409767597878e-07\\
2350	4.800265492305e-07\\
2355	5.07172356208031e-07\\
2360	4.99136049008348e-07\\
2365	4.92425132773158e-07\\
2370	4.92455844566836e-07\\
2375	5.667138374925e-07\\
2380	4.90283392805707e-07\\
2385	5.0082351528969e-07\\
2390	5.04479906285722e-07\\
2395	5.08203140866225e-07\\
2400	5.25287749565651e-07\\
2405	5.75461407954481e-07\\
2410	7.9080215616516e-07\\
2415	8.99747345827922e-07\\
2420	6.34218764929652e-07\\
2425	6.00522548769982e-07\\
2430	5.61779308534688e-07\\
2435	5.38215905944321e-07\\
2440	5.25279504555771e-07\\
2445	5.09571775988234e-07\\
2450	5.1601326662827e-07\\
2455	5.0349090231574e-07\\
2460	4.99775404116017e-07\\
2465	5.01815632173905e-07\\
2470	4.95309616242772e-07\\
2475	4.97172678322778e-07\\
2480	5.01039679568596e-07\\
2485	5.00445631274039e-07\\
2490	5.3812364679006e-07\\
2495	5.27461965569921e-07\\
2500	4.82691763833664e-07\\
2505	4.9131692189006e-07\\
2510	5.09098806336077e-07\\
2515	5.38417893200284e-07\\
2520	4.90472392767448e-07\\
2525	4.82838380652663e-07\\
2530	4.82671523240282e-07\\
2535	4.82133027243525e-07\\
2540	4.88364863720575e-07\\
2545	5.08291197204252e-07\\
2550	4.61811507789419e-07\\
2555	4.63823230794863e-07\\
2560	4.57789501499254e-07\\
2565	4.66003356881702e-07\\
2570	4.71527645432184e-07\\
2575	4.65797303265643e-07\\
2580	4.6864585062305e-07\\
2585	4.76526209029368e-07\\
2590	4.80839465741535e-07\\
2595	4.73891987456415e-07\\
2600	4.74075107464315e-07\\
2605	4.72407652711549e-07\\
2610	4.80676344627838e-07\\
2615	4.80389667806032e-07\\
2620	4.78820014070976e-07\\
2625	4.8664031013455e-07\\
2630	4.88702323828336e-07\\
2635	4.94878169738923e-07\\
2640	5.12083799967995e-07\\
2645	4.74132312298296e-07\\
2650	4.91791214093223e-07\\
2655	5.01089492663262e-07\\
2660	5.06416037838154e-07\\
2665	5.22181030446843e-07\\
2670	5.34019030585966e-07\\
2675	5.50051197210731e-07\\
2680	5.67283315551605e-07\\
2685	5.97440721524202e-07\\
2690	6.33449520914692e-07\\
2695	6.85784204356856e-07\\
2700	7.70226393340243e-07\\
2705	9.36818725194486e-07\\
2710	1.42650932508054e-06\\
2715	4.52033250843935e-06\\
2720	1.11621939608368e-06\\
2725	1.21954753912076e-06\\
2730	7.11345601316416e-07\\
2735	5.95376918867966e-07\\
2740	5.4932169933655e-07\\
2745	5.20521973345569e-07\\
2750	5.13771105023126e-07\\
2755	4.6841812384979e-07\\
2760	4.57640432005409e-07\\
2765	4.48877806860705e-07\\
2770	5.23777175255929e-07\\
2775	4.68075772315501e-07\\
2780	4.32025973048509e-07\\
2785	4.16710064973985e-07\\
2790	4.02602907252544e-07\\
2795	3.96378084392068e-07\\
2800	3.9006070843291e-07\\
2805	3.86064633023572e-07\\
2810	3.81325797060797e-07\\
2815	3.78843727766306e-07\\
2820	3.74598344551663e-07\\
2825	3.70959970699964e-07\\
2830	3.82610489613106e-07\\
2835	3.70673117945133e-07\\
2840	3.67126546224841e-07\\
2845	3.62233157961484e-07\\
2850	3.64534823133069e-07\\
2855	3.56782897126058e-07\\
2860	3.54342993747138e-07\\
2865	3.53354352388997e-07\\
2870	3.61573853923671e-07\\
2875	3.46881292232213e-07\\
2880	3.50796035190956e-07\\
2885	3.48173345624265e-07\\
2890	3.4442976199914e-07\\
2895	3.52279753845612e-07\\
2900	3.40674079676226e-07\\
2905	3.39081881473048e-07\\
2910	3.37884724045488e-07\\
2915	3.37395386456717e-07\\
2920	3.36524357128724e-07\\
2925	3.36769968321709e-07\\
2930	3.34925154666447e-07\\
2935	3.39508581280485e-07\\
2940	3.44809540133406e-07\\
2945	4.44129319964055e-07\\
2950	3.44954885605149e-07\\
2955	3.51232051261662e-07\\
2960	3.48892882426756e-07\\
2965	3.40702510413569e-07\\
2970	3.36518055330623e-07\\
2975	3.35253351274171e-07\\
2980	3.32953790413997e-07\\
2985	3.3051060621533e-07\\
2990	3.29366167027154e-07\\
2995	3.29332143742865e-07\\
3000	3.25062617641023e-07\\
3005	3.28070769023442e-07\\
3010	3.30588235414594e-07\\
3015	3.49341740745038e-07\\
3020	3.46488580650571e-07\\
3025	3.2761731366767e-07\\
3030	3.34445382145108e-07\\
3035	3.49879715817503e-07\\
3040	4.61051722990342e-07\\
3045	3.80638568872858e-07\\
3050	4.65880150237594e-07\\
3055	5.5410346455719e-07\\
3060	3.65902844989225e-07\\
3065	3.41866328376557e-07\\
3070	3.3570591401826e-07\\
3075	3.29862287615898e-07\\
3080	3.27211873497261e-07\\
3085	3.21947891092825e-07\\
3090	3.22658723353855e-07\\
3095	3.21043741683864e-07\\
3100	3.22770088995196e-07\\
3105	3.22409575915031e-07\\
3110	3.23839818465841e-07\\
3115	3.32434109750206e-07\\
3120	3.34231326034628e-07\\
3125	3.39616852222511e-07\\
3130	3.46224971672004e-07\\
3135	3.52361332140883e-07\\
3140	3.61721001965768e-07\\
3145	3.74764939109386e-07\\
3150	3.90438998490995e-07\\
3155	4.10079618277724e-07\\
3160	4.36106532002721e-07\\
3165	4.79892465102942e-07\\
3170	6.30362775828502e-07\\
3175	6.12545082711557e-07\\
3180	8.84883603002238e-07\\
3185	4.52828157990815e-06\\
3190	1.07937869395826e-06\\
3195	6.81443251206973e-07\\
3200	5.50361099786029e-07\\
3205	4.79905483790068e-07\\
3210	4.43108668077129e-07\\
3215	4.22379583236642e-07\\
3220	3.85878816837271e-07\\
3225	3.69303916537856e-07\\
3230	3.5506969623316e-07\\
3235	3.48395309893648e-07\\
3240	3.38124197086032e-07\\
3245	3.33091760306266e-07\\
3250	3.30760076264556e-07\\
3255	3.20923887525821e-07\\
3260	3.17948593705921e-07\\
3265	3.1316447270708e-07\\
3270	3.09482819736651e-07\\
3275	3.06338656768103e-07\\
3280	3.05162010559583e-07\\
3285	3.00896935447395e-07\\
3290	3.00071911934228e-07\\
3295	3.00424492473907e-07\\
3300	3.06250419481759e-07\\
3305	3.00195005092165e-07\\
3310	3.02177170318569e-07\\
3315	3.09745647398653e-07\\
3320	3.39096458866229e-07\\
3325	4.64534665017349e-07\\
3330	3.01452488838072e-07\\
3335	2.85875595905383e-07\\
3340	2.8285471627383e-07\\
3345	2.74831971863416e-07\\
3350	2.79731435368412e-07\\
3355	2.74526917308729e-07\\
3360	2.68666438572002e-07\\
3365	2.68394668596592e-07\\
3370	2.64142117882906e-07\\
3375	2.62777232364274e-07\\
3380	2.61332890770885e-07\\
3385	2.62125196869693e-07\\
3390	2.72880055006293e-07\\
3395	3.07444923267393e-07\\
3400	3.84518823129494e-07\\
3405	4.13188165279322e-07\\
3410	3.10419784734485e-07\\
3415	2.86368061335979e-07\\
3420	2.74668389474996e-07\\
3425	2.73370663661946e-07\\
3430	2.78583124236367e-07\\
3435	4.91972436391603e-07\\
3440	2.66612663078605e-07\\
3445	2.57118014575451e-07\\
3450	2.53473214681215e-07\\
3455	2.50476375951908e-07\\
3460	2.73074089683177e-07\\
3465	2.50709173347898e-07\\
3470	4.3005965775679e-07\\
3475.00000000001	2.46239434383696e-07\\
3480	2.49188330591153e-07\\
3485	2.58613929452871e-07\\
3490	2.86117866145414e-07\\
3495	2.46222831818912e-07\\
3500	2.46282328233434e-07\\
3505	2.42316799230999e-07\\
3510	2.41336180863371e-07\\
3515	2.40404281979632e-07\\
3520.00000000001	2.43316450026262e-07\\
3525	2.42230600497269e-07\\
3530	2.6417315204557e-07\\
3535	2.41759565123265e-07\\
3540	2.41451995877597e-07\\
3545	2.40714583106789e-07\\
3550	2.4178914842804e-07\\
3555	2.45738500003816e-07\\
3560	2.44332857305641e-07\\
3565	2.423538272196e-07\\
3570	2.57168502760242e-07\\
3575	2.46231315352885e-07\\
3580	2.46291583114847e-07\\
3585	2.46130155639306e-07\\
3590	2.52925108833489e-07\\
3595	2.54468315928293e-07\\
3600	2.59475654122858e-07\\
3605	2.65658519689074e-07\\
3610	2.78456031535205e-07\\
3615	3.09364179405885e-07\\
3620	4.39999471690459e-07\\
3625	2.86867227388644e-07\\
3630	2.64027719622728e-07\\
3635	2.76835847617077e-07\\
3640	2.76872348016063e-07\\
3645	2.83150461229861e-07\\
3650	2.93572994577329e-07\\
3655	2.98211631198871e-07\\
3660	3.0657100342258e-07\\
3665	3.22428682557474e-07\\
3670	3.44763970824261e-07\\
3675	3.76843974345631e-07\\
3680	4.31985404531394e-07\\
3685	5.36366897914653e-07\\
3690	8.27610715920513e-07\\
3695	4.41422347780741e-06\\
3700	7.65248620441064e-07\\
3705	5.35124673490325e-07\\
3710	4.47724841910657e-07\\
3715	3.98467409665679e-07\\
3720	3.66688560900724e-07\\
3725	3.48398055479126e-07\\
3730	3.49621645764672e-07\\
3735	3.221173940417e-07\\
3740.00000000001	3.12957028760829e-07\\
3745	3.0975541270176e-07\\
3750	3.0720430625997e-07\\
3755	3.06368670317431e-07\\
3760	3.1345983131768e-07\\
3765	3.37080651987863e-07\\
3770.00000000001	2.32392682998801e-06\\
3775	2.5958766594589e-07\\
3780	2.41234685450761e-07\\
3785	2.38707626799744e-07\\
3790	2.3672364459572e-07\\
3795.00000000001	2.34665800090775e-07\\
3800	2.3792203747634e-07\\
3805	2.32543214210518e-07\\
3810	2.32331077946058e-07\\
3815	2.35173288234378e-07\\
3820.00000000001	2.25486959807973e-07\\
3825	2.28463417639094e-07\\
3830	2.25425053795384e-07\\
3835	2.24971852932639e-07\\
3840.00000000001	2.24085821186188e-07\\
3845	2.26725149638123e-07\\
3850	2.28306560384657e-07\\
3855	2.24865028382224e-07\\
3860	2.56702932606183e-07\\
3865	2.2367329073341e-07\\
3870	2.19833848814388e-07\\
3875	2.15790805244073e-07\\
3880	2.13855552712843e-07\\
3885	2.11246213873435e-07\\
3890	2.15955998499112e-07\\
3895	2.1646628262775e-07\\
3900	2.14148213816418e-07\\
3905	2.13903189310892e-07\\
3910	2.21435899094595e-07\\
3915	2.16496946359679e-07\\
3920	2.22551073168196e-07\\
3925	2.33640445648247e-07\\
3930	2.35267871507648e-07\\
3935	2.28840354350829e-07\\
3940	2.07382966228653e-07\\
3945	2.04255314473391e-07\\
3950	2.05345233856538e-07\\
3955	2.01964130063863e-07\\
3960	2.18526007492399e-07\\
3965	2.47903524041615e-07\\
3970	2.51276857140713e-07\\
3975	2.10326768331324e-07\\
3980	2.07733968285544e-07\\
3985	2.06338657089961e-07\\
3990	2.06643331861272e-07\\
3995	2.01672942646149e-07\\
4000	2.0084058088904e-07\\
4005	2.00283196375202e-07\\
4010	1.99207399515665e-07\\
4015	1.99601481531364e-07\\
4020	1.98197453572862e-07\\
4025	1.98142424854883e-07\\
4030.00000000001	1.97882973014614e-07\\
4035	1.96886514974582e-07\\
4040	1.98547335162047e-07\\
4045	1.95258599934099e-07\\
4050	1.95869671753893e-07\\
4055	2.00083182334615e-07\\
4060.00000000001	1.93942902863435e-07\\
4065	1.9342800881676e-07\\
4070	1.94097895956949e-07\\
4075	1.99183900512566e-07\\
4080.00000000001	1.96999153187197e-07\\
4085	1.94781627367494e-07\\
4090	1.97065879358843e-07\\
4095.00000000001	2.11686595629867e-07\\
4100	2.00195169814099e-07\\
4105	2.44228753244845e-07\\
4110	1.91276203978083e-07\\
4115	1.88409906136552e-07\\
4120	1.8789100576039e-07\\
4125	1.88546112442745e-07\\
4130.00000000001	1.87354615123257e-07\\
4135	1.8599971151234e-07\\
4140.00000000001	1.86665226271656e-07\\
4145.00000000001	1.89194692575714e-07\\
4150	2.150184947489e-07\\
4155.00000000001	3.88985491610467e-07\\
4160	2.68995705804499e-07\\
4165	2.11031762742209e-07\\
4170	1.93880193760776e-07\\
4175	1.96221065396998e-07\\
4180	2.00264360545989e-07\\
4185	2.0164309196825e-07\\
4190	2.04833883894444e-07\\
4195.00000000001	2.11206949341973e-07\\
4200	2.15762791417147e-07\\
4205	2.22547991652828e-07\\
4210	2.33575609950967e-07\\
4215	2.42344912378959e-07\\
4220	2.58671756188185e-07\\
4225	2.88602501730974e-07\\
4230	3.37003078241317e-07\\
4235	4.49624649600476e-07\\
4240	1.26770569768289e-06\\
4245.00000000001	4.4162578828007e-07\\
4250	3.04825072936691e-07\\
4255	2.57140664199212e-07\\
4260	2.33505222123276e-07\\
4265.00000000001	2.178805164139e-07\\
4270	2.06870360625643e-07\\
4275	1.98429671279253e-07\\
4280	1.91778839570229e-07\\
4285	1.85685865259317e-07\\
4290.00000000001	1.82452406093025e-07\\
4295.00000000001	1.77208592795494e-07\\
4300.00000000001	1.78713829767932e-07\\
4305	1.82924720868981e-07\\
4310	1.78272415913175e-07\\
4315	1.74360051556014e-07\\
4320	1.7397775954164e-07\\
4325	1.72715327250681e-07\\
4330	1.72965923954856e-07\\
4335	1.72889693413465e-07\\
4340.00000000001	1.74758451009261e-07\\
4345.00000000001	1.70670615609546e-07\\
4350.00000000001	1.69793369102057e-07\\
4355.00000000001	2.00340545795592e-07\\
4360	1.80087151287219e-07\\
4365.00000000001	1.73980823076014e-07\\
4370	1.71663239099379e-07\\
4375	1.72815109304846e-07\\
4380	1.70287379366695e-07\\
4385.00000000001	1.68637594622772e-07\\
4390	1.69304638595889e-07\\
4395.00000000001	1.6892297522452e-07\\
4400	1.69493707681196e-07\\
4405	1.68027002928596e-07\\
4410.00000000001	1.6624894934483e-07\\
4415	1.69216986882164e-07\\
4420	1.68205957537923e-07\\
4425.00000000001	1.70948173118823e-07\\
4430	1.67787013798446e-07\\
4435	1.66468179882148e-07\\
4440	1.66929045579263e-07\\
4445	1.67405737890369e-07\\
4450	1.66974369460746e-07\\
4455	1.65847760378995e-07\\
4460	1.65960098553476e-07\\
4465	1.66300398828449e-07\\
4470	1.66387550774756e-07\\
4475	1.66084226948102e-07\\
4480	1.66601338720009e-07\\
4485	1.6659264178521e-07\\
4490	1.67433047276885e-07\\
4495	1.67047028652619e-07\\
4500	1.67936461404835e-07\\
4505	1.68595532540916e-07\\
4510	1.69570383498438e-07\\
4515	1.71324044350586e-07\\
4520	1.74771942678108e-07\\
4525	1.84711166488508e-07\\
4530	1.93520198265173e-07\\
4535	1.70104179011023e-07\\
4540	1.63042681719401e-07\\
4545	1.68601170956822e-07\\
4550	1.77551186378717e-07\\
4555	1.95152185207504e-07\\
4560	3.09194716991818e-07\\
4565	1.95633475663807e-07\\
4570	1.70672602939842e-07\\
4575	1.8346475813174e-07\\
4580	3.93171107520786e-07\\
4585	1.73088953002215e-07\\
4590	1.60366714719402e-07\\
4595.00000000001	1.57495958430546e-07\\
4600	1.56355210425206e-07\\
4605.00000000001	1.55467440432179e-07\\
4610	1.57496381836797e-07\\
4615	1.56306490682585e-07\\
4620.00000000001	1.65235765097383e-07\\
4625	1.55900864328547e-07\\
4630	1.5544325444505e-07\\
4635	1.53331712930708e-07\\
4640	1.53101369080535e-07\\
4645	1.54185164152648e-07\\
4650.00000000001	1.52845463076825e-07\\
4655	1.52837619985278e-07\\
4660	1.52595902183749e-07\\
4665.00000000001	1.52996107562756e-07\\
4670.00000000001	1.53003629886495e-07\\
4675	1.57877607393917e-07\\
4680.00000000001	1.52535585108905e-07\\
4685	1.54277738748664e-07\\
4690.00000000001	1.54017793348941e-07\\
4695	1.54700728021734e-07\\
4700.00000000001	1.53921751773348e-07\\
4705	1.53869702865865e-07\\
4710	1.56294024242934e-07\\
4715	1.55099063109136e-07\\
4720.00000000001	1.56182177194622e-07\\
4725.00000000001	1.57394267383352e-07\\
4730	1.58988554651041e-07\\
4735	1.59935013894014e-07\\
4740.00000000001	1.62586542134112e-07\\
4745	1.66121591735564e-07\\
4750	2.68311343416168e-07\\
4755	1.6531335302134e-07\\
4760	1.66618487476696e-07\\
4765.00000000001	1.69558681547748e-07\\
4770	1.77559460650071e-07\\
4775	1.78929001090622e-07\\
4780	1.93597701424466e-07\\
4785	1.94877950905702e-07\\
4790	2.02242477829957e-07\\
4795	2.30348565324497e-07\\
4800	2.26791970123944e-07\\
4805	2.46950557502185e-07\\
4810	2.77700327744231e-07\\
4815	3.41049235779856e-07\\
4820	4.96139253054235e-07\\
4825	2.94433939002483e-06\\
4830	5.85585894742171e-07\\
4835	3.75414789245679e-07\\
4840	3.06863625784092e-07\\
4845	2.69957818245634e-07\\
4850	2.46918370345691e-07\\
4855	2.31810962480148e-07\\
4860	2.20882743535082e-07\\
4865	2.12557988536925e-07\\
4870	2.06495575074309e-07\\
4875	1.99928943204004e-07\\
4880	1.94908125998413e-07\\
4885	1.94985554443556e-07\\
4890.00000000001	1.82169907501537e-07\\
4895	1.80758708626673e-07\\
4900	1.78550392380337e-07\\
4905.00000000001	1.76134794591221e-07\\
4910	1.75512140313692e-07\\
4915.00000000001	1.75575226785579e-07\\
4920	1.72916442080891e-07\\
4925	1.7865404736271e-07\\
4930.00000000001	1.81338260403351e-07\\
4935	1.75596939298976e-07\\
4940	1.73374677102945e-07\\
4945	1.72326836655594e-07\\
4950.00000000001	1.73744645243767e-07\\
4955	1.73724762108335e-07\\
4960.00000000001	1.74905487334629e-07\\
4965	1.77694132890138e-07\\
4970	1.82590497060247e-07\\
4975.00000000001	1.92522170905451e-07\\
4980	2.19608486286992e-07\\
4985	5.36571748252505e-07\\
4990	2.07460458563525e-07\\
4995	1.69101702782962e-07\\
5000	1.58989015281033e-07\\
5005.00000000001	1.52073672935661e-07\\
5010	1.49046851621677e-07\\
5015.00000000001	1.44805378725699e-07\\
5020	1.42658309337716e-07\\
5025.00000000001	1.45337983855919e-07\\
5030	1.42778327353847e-07\\
5035	1.4016466674456e-07\\
5040.00000000001	1.38366470224703e-07\\
5045	1.37268472756384e-07\\
5050.00000000001	1.37276130546238e-07\\
5055	1.3582243047194e-07\\
5060.00000000001	1.33827854309157e-07\\
5065	1.34480367484947e-07\\
5070.00000000001	1.35770494148189e-07\\
5075	1.33272534511742e-07\\
5080.00000000001	1.3630153589408e-07\\
5085.00000000001	1.3213576014332e-07\\
5090.00000000001	1.32579682387378e-07\\
5095	1.32375248404034e-07\\
5100	1.31166971386799e-07\\
5105	1.30695090118632e-07\\
5110	1.30627914577767e-07\\
5115	1.31224272881485e-07\\
5120.00000000001	1.29372124136289e-07\\
5125	1.2926511982223e-07\\
5130	1.40786581546774e-07\\
5135.00000000001	1.29679409385819e-07\\
5140.00000000001	1.29241579373295e-07\\
5145	1.29369411293266e-07\\
5150.00000000001	1.2873198944855e-07\\
5155	1.3188977429126e-07\\
5160	1.37923379573792e-07\\
5165	1.34750597056816e-07\\
5170	1.31977885910667e-07\\
5175	1.30581151622633e-07\\
5180	1.31113422019652e-07\\
5185	1.31229877945731e-07\\
5190	1.29560058458816e-07\\
5195.00000000001	1.30070132051623e-07\\
5200	1.29971885464304e-07\\
5205	1.67238705263859e-07\\
5210.00000000001	1.32316777209006e-07\\
5215	1.31370295906664e-07\\
5220	1.39139782855948e-07\\
5225	1.28184953892498e-07\\
5230	1.28755989274463e-07\\
5235	1.31893135534083e-07\\
5240	1.28075671560989e-07\\
5245	1.29359543463425e-07\\
5250	1.29501476430035e-07\\
5255	1.29614883551384e-07\\
5260	1.30839956666265e-07\\
5265	1.31755951713361e-07\\
5270.00000000001	1.34903241204325e-07\\
5275	1.39889879828632e-07\\
5280	1.52313274435466e-07\\
5285	2.39834166452771e-07\\
5290	1.43709267504359e-07\\
5295	1.5464027243414e-07\\
5300	1.31722926063884e-07\\
5305	1.28690392484014e-07\\
5310	1.28913474371964e-07\\
5315	1.26811143305475e-07\\
5320	1.31190239628094e-07\\
5325	1.26016844114308e-07\\
5330	1.26266384624154e-07\\
5335.00000000001	1.26617112917713e-07\\
5340	1.28007425805161e-07\\
5345	1.28634651426118e-07\\
5350	1.29777125784422e-07\\
5355.00000000001	1.30577401976117e-07\\
5360.00000000001	1.32885873291333e-07\\
5365	1.34594760669909e-07\\
5370	1.36731401496212e-07\\
5375	1.37922227343234e-07\\
5380.00000000001	1.41369746893663e-07\\
5385	1.44462168078734e-07\\
5390.00000000001	1.48586121503486e-07\\
5395	1.52731955534821e-07\\
5400.00000000001	1.5928493937622e-07\\
5405.00000000001	1.66404735782747e-07\\
5410	1.75811538385467e-07\\
5415	1.88855923558417e-07\\
5420.00000000001	2.10376617014611e-07\\
5425	2.63546887501476e-07\\
5430	3.27770116147911e-07\\
5435	2.53744078825333e-07\\
5440.00000000001	3.28948727155559e-07\\
5445	6.65360193733969e-07\\
5450	7.11018553087543e-07\\
5455	3.46180988458229e-07\\
5460.00000000001	2.59599312222326e-07\\
5465.00000000001	2.21247763446094e-07\\
5470	2.00372609019113e-07\\
5475	1.84786246728071e-07\\
5480	1.75435826110922e-07\\
5485.00000000001	1.68169027916387e-07\\
5490.00000000001	1.61900244582398e-07\\
5495	1.59913061747383e-07\\
5500	1.51050210932764e-07\\
5505.00000000001	1.48127258760661e-07\\
5510	1.44828438338174e-07\\
5515	1.44955167141071e-07\\
5520.00000000001	1.41243404441253e-07\\
5525	1.3942796036348e-07\\
5530.00000000001	1.36211256515775e-07\\
5535.00000000001	1.35950180685944e-07\\
5540.00000000001	1.36905472518811e-07\\
5545.00000000001	1.2880699588383e-07\\
5550	1.29575432527359e-07\\
5555	1.29917799766767e-07\\
5560	1.28993725473689e-07\\
5565	1.28858547036984e-07\\
5570.00000000001	1.27227914033658e-07\\
5575	1.26280447386588e-07\\
5580	1.27571787460148e-07\\
5585	1.25509708288507e-07\\
5590	1.23595072516107e-07\\
5595	1.24348033860836e-07\\
5600	1.24068725494487e-07\\
5605	1.24043144512062e-07\\
5610	1.23569038320629e-07\\
5615.00000000001	1.23295279963648e-07\\
5620	1.23252418095938e-07\\
5625	1.22039504499983e-07\\
5630	1.22912497695607e-07\\
5635	1.22013146491127e-07\\
5640	1.22320937908269e-07\\
5645	1.23870695983117e-07\\
5650	1.27188362991246e-07\\
5655	6.11015673935986e-07\\
5660.00000000001	1.31776201606221e-07\\
5665	1.20801459554806e-07\\
5670	1.1760273863433e-07\\
5675.00000000001	1.15857909813259e-07\\
5680	1.14201723750973e-07\\
5685	1.14670455851691e-07\\
5690.00000000001	1.13098133448381e-07\\
5695.00000000001	1.14605073548406e-07\\
5700	1.18519573520385e-07\\
5705	1.14125146141522e-07\\
5710	1.1424237374513e-07\\
5715.00000000001	1.14187260539774e-07\\
5720.00000000001	1.14863994919531e-07\\
5725	1.15020870344107e-07\\
5730	1.17504634658046e-07\\
5735	1.14267879311678e-07\\
5740	1.13580620268844e-07\\
5745	1.13279224073253e-07\\
5750	1.13144985602751e-07\\
5755	1.17952849908831e-07\\
5760	1.13077845602232e-07\\
5765	1.1241015222356e-07\\
5770.00000000001	1.14669961383626e-07\\
5775.00000000001	1.16642051079008e-07\\
5780.00000000001	1.12204277741433e-07\\
5785.00000000001	1.12775027801434e-07\\
5790	1.12875505512801e-07\\
5795	1.13203756405185e-07\\
5800.00000000001	1.13708071917539e-07\\
5805	1.15861832173918e-07\\
5810	1.12474836673453e-07\\
5815	1.12506159302286e-07\\
5820.00000000001	1.12222951999448e-07\\
5825	1.12106078304681e-07\\
5830.00000000001	1.12674815403692e-07\\
5835.00000000001	1.12758788493537e-07\\
5840	1.13155380218593e-07\\
5845	1.16006871468094e-07\\
5850.00000000001	1.13302082005463e-07\\
5855	1.14832551398397e-07\\
5860.00000000001	1.1447178492339e-07\\
5865.00000000001	1.24085959780887e-07\\
5870.00000000001	1.16982315805271e-07\\
5875	1.17985635533127e-07\\
5880.00000000001	1.2207517746074e-07\\
5885	1.30467252325642e-07\\
5890	2.26570630835853e-07\\
5895.00000000001	1.32010789243375e-07\\
5900.00000000001	1.17847547268109e-07\\
5905.00000000001	1.12750605791199e-07\\
5910.00000000001	1.11283566982281e-07\\
5915.00000000001	1.09679663627222e-07\\
5920	1.0867072153105e-07\\
5925.00000000001	1.07859184499984e-07\\
5930.00000000001	1.07776447364247e-07\\
5935	1.0733221788849e-07\\
5940	1.07355688717724e-07\\
5945	1.06668340178479e-07\\
5950	1.09261191221501e-07\\
5955	1.14615470675259e-07\\
5960	1.07918008746415e-07\\
5965.00000000001	1.82848971225187e-07\\
5970	1.07364210155819e-07\\
5975	1.0629163156782e-07\\
5980.00000000001	1.06614778214915e-07\\
5985	1.06839863619083e-07\\
5990	1.10787491019514e-07\\
5995.00000000001	1.07798036534341e-07\\
6000	1.07382741025376e-07\\
};
\addplot [color=black, line width=0.9pt, forget plot]
  table[row sep=crcr]{%
200	8.565008204102601e-06\\
6000	2.227232665603490e-08\\
};

\node[] at (axis cs: 650, 1e-7) {\small $\zeta=\rmi$};
\node[] at (axis cs: 650, 3e-8) {\small $v(x)=x$};

\end{axis}

\end{tikzpicture}%

%% file: plot_an_okb.tex
%
%
\begin{tikzpicture}

\begin{axis}[%
width=1.7in,
height=1.7in,
at={(0.769in,0.477in)},
scale only axis,
xmode=log,
xmin=100,
xmax=10000,
xminorticks=true,
ymode=log,
ymin=1e-08,
ymax=1e-4,
yminorticks=true,
axis background/.style={fill=white},
xlabel = {$N$},
ylabel = {Error}
]
\addplot [color=blue, only marks, mark=x, mark options={solid, blue}, forget plot]
  table[row sep=crcr]{%
200	1.96462902315453e-05\\
205	1.88175321005346e-05\\
210	1.80424948381663e-05\\
215	1.73165346288862e-05\\
220	1.66355049899813e-05\\
225	1.59956937006989e-05\\
230	1.5393768937777e-05\\
235	1.48267330604135e-05\\
240	1.42918828505145e-05\\
245	1.37867752090237e-05\\
250	1.33091974423572e-05\\
255	1.28571414519385e-05\\
260	1.24287812890422e-05\\
265	1.20224534785329e-05\\
270	1.16366398326129e-05\\
275	1.12699524148496e-05\\
280	1.0921120174423e-05\\
285	1.05889772745727e-05\\
290	1.02724527044629e-05\\
295	9.97056112028715e-06\\
300	9.68239463761122e-06\\
305	9.40711563224282e-06\\
310	9.14395021256453e-06\\
315	8.89218245303526e-06\\
320	8.65114924719635e-06\\
325	8.42023563007999e-06\\
330	8.19887058245427e-06\\
335	7.9865233573173e-06\\
340	7.78270001733361e-06\\
345	7.58694043812014e-06\\
350	7.39881551936605e-06\\
355	7.21792467528459e-06\\
360	7.0438936110584e-06\\
365	6.87637224183747e-06\\
370	6.71503280003094e-06\\
375	6.55956816641989e-06\\
380	6.40969029586103e-06\\
385	6.26512881440888e-06\\
390	6.12562972834852e-06\\
395	5.99095421804918e-06\\
400	5.86087753573494e-06\\
405.000000000001	5.73518808311135e-06\\
410	5.61368643303694e-06\\
415	5.49618449463551e-06\\
420	5.38250477255531e-06\\
425	5.27247962489596e-06\\
430.000000000001	5.16595062016734e-06\\
435.000000000001	5.06276794887128e-06\\
440	4.96278980977038e-06\\
445	4.86588194981153e-06\\
450	4.77191715075875e-06\\
455	4.68077482507212e-06\\
460	4.59234057537117e-06\\
465.000000000001	4.50650576722112e-06\\
470.000000000001	4.4231673053119e-06\\
475	4.34222716094723e-06\\
480	4.26359216953997e-06\\
485	4.18717365491262e-06\\
490	4.11288729518233e-06\\
495.000000000001	4.04065272174847e-06\\
500	3.97039339938843e-06\\
505.000000000001	3.90203638689357e-06\\
510	3.83551207461252e-06\\
515.000000000001	3.77075410940008e-06\\
520	3.70769909840974e-06\\
525	3.64628656779331e-06\\
530	3.58645865494722e-06\\
535	3.52816011472967e-06\\
540.000000000001	3.4713380960838e-06\\
545	3.41594209318785e-06\\
550	3.36192366701127e-06\\
555	3.30923655322835e-06\\
560	3.25783637933341e-06\\
565	3.20768062467281e-06\\
570	3.1587285409529e-06\\
575	3.11094108962351e-06\\
580.000000000001	3.06428079177579e-06\\
585.000000000001	3.01871167085465e-06\\
590.000000000001	2.97419923267483e-06\\
595	2.93071030599279e-06\\
600	2.88821305982623e-06\\
605.000000000001	2.84667693506435e-06\\
610	2.80607246683218e-06\\
615	2.76637141949365e-06\\
620.000000000001	2.72754658192653e-06\\
625.000000000001	2.68957174309748e-06\\
630.000000000001	2.65242178354441e-06\\
635	2.61607241736072e-06\\
640	2.58050031298751e-06\\
645	2.54568295510182e-06\\
650	2.51159865793937e-06\\
655	2.47822653864915e-06\\
660.000000000001	2.44554644490691e-06\\
665	2.41353896512919e-06\\
670.000000000001	2.3821853192274e-06\\
675.000000000001	2.35146746385695e-06\\
680	2.32136792188697e-06\\
685	2.29186984412877e-06\\
690.000000000001	2.26295693872558e-06\\
695	2.23461348980436e-06\\
700	2.20682431661956e-06\\
705.000000000001	2.17957472159469e-06\\
710	2.15285053517533e-06\\
715.000000000001	2.12663802434676e-06\\
720.000000000001	2.10092390506844e-06\\
725	2.07569533738905e-06\\
730	2.0509399241142e-06\\
735	2.02664563975219e-06\\
740	2.00280081319448e-06\\
745	1.97939419743776e-06\\
750	1.95641489408871e-06\\
755	1.93385230939924e-06\\
760	1.91169622176801e-06\\
765.000000000001	1.88993668492898e-06\\
770.000000000001	1.86856408035396e-06\\
775	1.84756914034523e-06\\
780.000000000001	1.82694272599093e-06\\
785	1.80667613625118e-06\\
790.000000000001	1.78676083173457e-06\\
795	1.7671885803594e-06\\
800	1.74795133256467e-06\\
805	1.72904136475083e-06\\
810	1.71045108965373e-06\\
815	1.6921731966768e-06\\
820	1.67420058172496e-06\\
825.000000000001	1.65652631078927e-06\\
830	1.63914366124729e-06\\
835	1.62204614628791e-06\\
840.000000000001	1.60522737946422e-06\\
845.000000000001	1.58868123101286e-06\\
850	1.57240166132056e-06\\
855	1.556382934087e-06\\
860.000000000001	1.54061929258376e-06\\
865.000000000001	1.52510527406946e-06\\
870.000000000001	1.50983547175798e-06\\
875	1.49480474487263e-06\\
880.000000000001	1.48000797883796e-06\\
885.000000000001	1.46544019896666e-06\\
890.000000000001	1.45109665528054e-06\\
895	1.43697263155218e-06\\
900.000000000001	1.4230636042889e-06\\
905	1.40936506864975e-06\\
910.000000000001	1.39587276670738e-06\\
915	1.38258244408718e-06\\
920.000000000001	1.36949004181375e-06\\
925	1.35659151956347e-06\\
930.000000000001	1.34388301598065e-06\\
935.000000000001	1.33136071500672e-06\\
940.000000000001	1.31902090849678e-06\\
945.000000000001	1.30685998644964e-06\\
950	1.29487444988641e-06\\
955.000000000001	1.28306086377705e-06\\
960.000000000001	1.27141585437584e-06\\
965	1.25993617272613e-06\\
970.000000000001	1.24861864669867e-06\\
975	1.23746013658277e-06\\
980.000000000001	1.22645765809892e-06\\
985.000000000001	1.21560821408906e-06\\
989.999999999999	1.20490890553881e-06\\
995.000000000001	1.19435697509829e-06\\
1000	1.18394958503743e-06\\
1005	1.17368409657814e-06\\
1010	1.16355788293277e-06\\
1015	1.15356834973213e-06\\
1020	1.1437130202907e-06\\
1025	1.13398944634468e-06\\
1030	1.124395223151e-06\\
1035	1.1149280148004e-06\\
1040	1.10558557331331e-06\\
1045	1.09636559386672e-06\\
1050	1.08726597325415e-06\\
1055	1.07828452433623e-06\\
1060	1.069419216293e-06\\
1065	1.06066796368154e-06\\
1070	1.0520287765381e-06\\
1075	1.04349974128226e-06\\
1080	1.03507890969468e-06\\
1085	1.02676446900318e-06\\
1090	1.01855451672961e-06\\
1095	1.01044737332856e-06\\
1100	1.00244119005666e-06\\
1105	9.94534306020256e-07\\
1110	9.86725060325711e-07\\
1115	9.79011811175212e-07\\
1120	9.71392942528126e-07\\
1125	9.63866852554674e-07\\
1130	9.56432073095925e-07\\
1135	9.49087070267752e-07\\
1140	9.41830330170035e-07\\
1145	9.34660516094254e-07\\
1150	9.27576113696205e-07\\
1155	9.20575789375987e-07\\
1160	9.1365819132605e-07\\
1165	9.06821959745231e-07\\
1170	9.00065827202923e-07\\
1175	8.93388498290904e-07\\
1180	8.86788713572172e-07\\
1185	8.80265310421181e-07\\
1190	8.73817007640553e-07\\
1195	8.67442680352327e-07\\
1200	8.61141132180165e-07\\
1205	8.54911317293982e-07\\
1210	8.48752027327038e-07\\
1215	8.42662269295858e-07\\
1220	8.3664089745028e-07\\
1225	8.30686944119919e-07\\
1230	8.24799300414015e-07\\
1235	8.18977019534373e-07\\
1240	8.13219117823394e-07\\
1245	8.07524589863107e-07\\
1250	8.01892519053382e-07\\
1255	7.96321959040114e-07\\
1260	7.90812013651276e-07\\
1265	7.85361773392167e-07\\
1270	7.79970365627492e-07\\
1275	7.74636944367302e-07\\
1280	7.69360631203142e-07\\
1285	7.64140577480532e-07\\
1290	7.58976023362834e-07\\
1295	7.53866153058169e-07\\
1300	7.48810124129307e-07\\
1305	7.43807220704441e-07\\
1310	7.38856729576299e-07\\
1315	7.33957788767726e-07\\
1320	7.29109745023493e-07\\
1325	7.24311883804063e-07\\
1330	7.19563459039563e-07\\
1335	7.14863779727182e-07\\
1340	7.10212203713923e-07\\
1345	7.05608029338833e-07\\
1350	7.01050599349884e-07\\
1355	6.96539334210656e-07\\
1360	6.92073520269788e-07\\
1365	6.87652541131455e-07\\
1370	6.8327583102601e-07\\
1375	6.78942764675839e-07\\
1380	6.74652730126013e-07\\
1385	6.70405174041378e-07\\
1390	6.66199512888711e-07\\
1395	6.6203516801977e-07\\
1400	6.57911644719179e-07\\
1405	6.53828308383452e-07\\
1410	6.49784719808366e-07\\
1415	6.4578032743512e-07\\
1420	6.41814563273612e-07\\
1425	6.37886973464674e-07\\
1430	6.3399703176259e-07\\
1435	6.30144233237928e-07\\
1440	6.26328178654489e-07\\
1445	6.22548280926339e-07\\
1450	6.18804163465824e-07\\
1455	6.15095311573555e-07\\
1460	6.11421313134742e-07\\
1465	6.07781643235938e-07\\
1470	6.04176007001911e-07\\
1475	6.00603851541592e-07\\
1480	5.9706479049737e-07\\
1485	5.93558397099514e-07\\
1490	5.9008429653673e-07\\
1495	5.86642018962635e-07\\
1500	5.8323124951798e-07\\
1505	5.79851542781284e-07\\
1510	5.76502539040292e-07\\
1515	5.73183837726531e-07\\
1520	5.69895076019123e-07\\
1525	5.66635887100375e-07\\
1530	5.63405904596692e-07\\
1535	5.60204763022654e-07\\
1540	5.57032160841687e-07\\
1545	5.53887736565172e-07\\
1550	5.5077107585788e-07\\
1555	5.47681919815801e-07\\
1560	5.44619906950316e-07\\
1565	5.41584762814296e-07\\
1570	5.38576109487821e-07\\
1575	5.35593620121233e-07\\
1580	5.32637045136397e-07\\
1585	5.29706007945663e-07\\
1590	5.26800377986802e-07\\
1595	5.23919581940646e-07\\
1600	5.2106357140147e-07\\
1605	5.18231926260882e-07\\
1610	5.15424397384834e-07\\
1615	5.12640760508276e-07\\
1620	5.09880600407797e-07\\
1625	5.07143782968456e-07\\
1630	5.04429992442823e-07\\
1635	5.01738921965255e-07\\
1640	4.99070402337764e-07\\
1645	4.96424172435894e-07\\
1650	4.93799849454746e-07\\
1655	4.911972983912e-07\\
1660	4.88616266114406e-07\\
1665	4.86056489279463e-07\\
1670	4.83517756499907e-07\\
1675	4.80999768903701e-07\\
1680	4.78502379053225e-07\\
1685	4.76025276086034e-07\\
1690	4.73568298353655e-07\\
1695	4.71131170964867e-07\\
1700	4.68713738932535e-07\\
1705	4.66315751790347e-07\\
1710	4.63936969286039e-07\\
1715	4.61577244426082e-07\\
1720	4.59236331185054e-07\\
1725	4.5691405237136e-07\\
1730	4.54610126432442e-07\\
1735	4.52324417921091e-07\\
1740	4.50056730993964e-07\\
1745	4.47806937309281e-07\\
1750	4.45574717566899e-07\\
1755	4.43359927437825e-07\\
1760	4.41162395503625e-07\\
1765	4.38981987649357e-07\\
1770	4.36818477833611e-07\\
1775	4.34671652449481e-07\\
1780	4.32541335193548e-07\\
1785	4.30427449682469e-07\\
1790	4.28329752111267e-07\\
1795	4.26248071949687e-07\\
1800	4.24182247993343e-07\\
1805	4.22132146571386e-07\\
1810	4.20097524322927e-07\\
1815	4.18078303532354e-07\\
1820	4.16074341647033e-07\\
1825	4.14085410849197e-07\\
1830	4.12111382797065e-07\\
1835	4.10152107388484e-07\\
1840	4.082074469558e-07\\
1845	4.06277273601318e-07\\
1850	4.04361395034414e-07\\
1855	4.02459696680069e-07\\
1860	4.00571954717321e-07\\
1865	3.98698171810707e-07\\
1870	3.96838095717556e-07\\
1875	3.94991613195117e-07\\
1880	3.93158622546963e-07\\
1885	3.91338931926554e-07\\
1890	3.89532515132629e-07\\
1895	3.8773911681389e-07\\
1900	3.85958724535839e-07\\
1905	3.84191131352907e-07\\
1910	3.82436204482417e-07\\
1915	3.80693900847717e-07\\
1920	3.7896395488346e-07\\
1925	3.77246410554477e-07\\
1930	3.75541066688356e-07\\
1935	3.73847805601457e-07\\
1940	3.72166497175641e-07\\
1945	3.70497077906152e-07\\
1950	3.6883936660459e-07\\
1955	3.67193296657575e-07\\
1960	3.65558783688158e-07\\
1965	3.63935677594185e-07\\
1970	3.62323842928447e-07\\
1975	3.60723221071169e-07\\
1980	3.59133631278041e-07\\
1985	3.5755506511137e-07\\
1990	3.55987427980153e-07\\
1995	3.5443054047235e-07\\
2000	3.52884307996959e-07\\
2005	3.51348683924612e-07\\
2010	3.49823507050928e-07\\
2015	3.48308767161853e-07\\
2020	3.46804220008323e-07\\
2025	3.45309928206916e-07\\
2030	3.43825766080386e-07\\
2035	3.42351564874832e-07\\
2040	3.4088728284587e-07\\
2045	3.39432795204431e-07\\
2050	3.37988101506426e-07\\
2055	3.36553038771114e-07\\
2060	3.35127511075228e-07\\
2065	3.33711474009846e-07\\
2070	3.32304752159729e-07\\
2075	3.30907424128668e-07\\
2080	3.29519301622838e-07\\
2085	3.28140226546481e-07\\
2090	3.26770301484203e-07\\
2095	3.25409278190136e-07\\
2100	3.24057167322423e-07\\
2105	3.22713938682995e-07\\
2110	3.21379336920557e-07\\
2115	3.20053405111764e-07\\
2120	3.18736154358845e-07\\
2125	3.17427349738608e-07\\
2130	3.16126958388452e-07\\
2135	3.14834946557596e-07\\
2140	3.1355115170939e-07\\
2145	3.12275642677662e-07\\
2150	3.11008269360258e-07\\
2155	3.09748909188556e-07\\
2160	3.0849759991014e-07\\
2165	3.07254218512298e-07\\
2170	3.06018634876892e-07\\
2175	3.04790925831355e-07\\
2180	3.03570856896585e-07\\
2185	3.02358540871239e-07\\
2190	3.0115380322826e-07\\
2195	2.99956532057166e-07\\
2200	2.9876680995855e-07\\
2205	2.97584406450113e-07\\
2210	2.9640938015163e-07\\
2215	2.95241653791578e-07\\
2220	2.94081120788548e-07\\
2225	2.92927775813467e-07\\
2230	2.91781460770579e-07\\
2235	2.90642213407466e-07\\
2240	2.89509995532455e-07\\
2245	2.88384632618488e-07\\
2250	2.87266141096865e-07\\
2255	2.86154444584242e-07\\
2260	2.85049543080617e-07\\
2265	2.8395133799819e-07\\
2270	2.82859716982386e-07\\
2275	2.8177473021529e-07\\
2280	2.80696165422256e-07\\
2285	2.79624150500979e-07\\
2290	2.78558646815696e-07\\
2295	2.77499402123738e-07\\
2300	2.76446506131123e-07\\
2305	2.75399902438522e-07\\
2310	2.74359449825567e-07\\
2315	2.73325179378503e-07\\
2320	2.72297102199559e-07\\
2325	2.71275034435803e-07\\
2330	2.70259032486564e-07\\
2335	2.692489293743e-07\\
2340	2.68244766843395e-07\\
2345	2.67246501817198e-07\\
2350	2.66254035707903e-07\\
2355	2.65267410259895e-07\\
2360	2.64286483808718e-07\\
2365	2.63311215498163e-07\\
2370	2.62341637746744e-07\\
2375	2.61377630206283e-07\\
2380	2.60419200870388e-07\\
2385	2.59466326202329e-07\\
2390	2.58518883189396e-07\\
2395	2.57576801665493e-07\\
2400	2.56640197537905e-07\\
2405	2.55708952234812e-07\\
2410	2.54782898778672e-07\\
2415	2.53862189047993e-07\\
2420	2.52946661394305e-07\\
2425	2.52036360670615e-07\\
2430	2.51131218487188e-07\\
2435	2.50231104725884e-07\\
2440	2.49336085111907e-07\\
2445	2.48446057504736e-07\\
2450	2.4756104100021e-07\\
2455	2.46680959659073e-07\\
2460	2.45805790832776e-07\\
2465	2.44935550952618e-07\\
2470	2.4407011123273e-07\\
2475	2.43209381967091e-07\\
2480	2.42353505708337e-07\\
2485	2.41502327469334e-07\\
2490	2.40655849914617e-07\\
2495	2.39814091251844e-07\\
2500	2.38976855637674e-07\\
2505	2.3814421901136e-07\\
2510	2.37316200024651e-07\\
2515	2.36492632588181e-07\\
2520	2.35673554449534e-07\\
2525	2.34858954950568e-07\\
2530	2.34048773251061e-07\\
2535	2.33242900993247e-07\\
2540	2.3244152069779e-07\\
2545	2.31644323722691e-07\\
2550	2.30851515237163e-07\\
2555	2.30062904282846e-07\\
2560	2.29278554808587e-07\\
2565	2.2849841396777e-07\\
2570	2.27722409373854e-07\\
2575	2.26950594317543e-07\\
2580	2.26182752083304e-07\\
2585	2.25419117594327e-07\\
2590	2.24659504333147e-07\\
2595	2.23903904306155e-07\\
2600	2.23152168743468e-07\\
2605	2.22404506367013e-07\\
2610	2.21660782173672e-07\\
2615	2.20920953530879e-07\\
2620	2.20185028432241e-07\\
2625	2.19452900296346e-07\\
2630	2.18724558020966e-07\\
2635	2.18000105967064e-07\\
2640	2.172793647226e-07\\
2645	2.165623866901e-07\\
2650	2.15849084828079e-07\\
2655	2.15139504433637e-07\\
2660	2.14433566902983e-07\\
2665	2.13731193188238e-07\\
2670	2.13032523621592e-07\\
2675	2.12337463612044e-07\\
2680	2.11645958980711e-07\\
2685	2.10957904478448e-07\\
2690	2.10273376044512e-07\\
2695	2.09592339484033e-07\\
2700	2.08914780586156e-07\\
2705	2.08240707344487e-07\\
2710	2.07570020283043e-07\\
2715	2.06902633692607e-07\\
2720	2.06238735422914e-07\\
2725	2.05578075895829e-07\\
2730	2.04920806101683e-07\\
2735	2.04266789705088e-07\\
2740	2.03616074223589e-07\\
2745	2.02968637452727e-07\\
2750	2.02324441200829e-07\\
2755	2.01683396650054e-07\\
2760	2.01045550873857e-07\\
2765	2.00410940287554e-07\\
2770	1.99779391252264e-07\\
2775	1.99150961055494e-07\\
2780	1.98525695438434e-07\\
2785	1.97903469612015e-07\\
2790	1.97284249381369e-07\\
2795	1.96668193730432e-07\\
2800	1.96055129020323e-07\\
2805	1.95444962880487e-07\\
2810	1.94837822320437e-07\\
2815	1.94233725103743e-07\\
2820	1.93632500256058e-07\\
2825	1.93034232598421e-07\\
2830	1.92438823543028e-07\\
2835	1.91846281527574e-07\\
2840	1.9125666206321e-07\\
2845	1.90669826594103e-07\\
2850	1.90085909679283e-07\\
2855	1.89504710590427e-07\\
2860	1.88926345234819e-07\\
2865	1.88350702590157e-07\\
2870	1.87777831950342e-07\\
2875	1.87207694235525e-07\\
2880	1.86640274346672e-07\\
2885	1.86075513664008e-07\\
2890	1.85513533423887e-07\\
2895	1.84954195958653e-07\\
2900	1.84397453306673e-07\\
2905	1.83843389844895e-07\\
2910	1.83291899880089e-07\\
2915	1.8274302515664e-07\\
2920	1.82196727926964e-07\\
2925	1.81652990871584e-07\\
2930	1.8111178468061e-07\\
2935	1.80573147989804e-07\\
2940	1.80036975550024e-07\\
2945	1.79503375274947e-07\\
2950	1.78972205056027e-07\\
2955	1.78443483100921e-07\\
2960	1.77917190313792e-07\\
2965	1.77393406630699e-07\\
2970	1.7687202058525e-07\\
2975	1.76353011749342e-07\\
2980	1.75836474713975e-07\\
2985	1.75322248718857e-07\\
2990	1.7481037106748e-07\\
2995	1.7430086662884e-07\\
3000	1.73793622604279e-07\\
3005	1.73288824178997e-07\\
3010	1.72786273289205e-07\\
3015	1.72286004129774e-07\\
3020	1.71787988723082e-07\\
3025	1.71292244832699e-07\\
3030	1.7079875247461e-07\\
3035	1.70307459246288e-07\\
3040	1.69818432649293e-07\\
3045	1.69331583421695e-07\\
3050	1.68846893799923e-07\\
3055	1.6836442506829e-07\\
3060	1.67884154578246e-07\\
3065	1.67406027262729e-07\\
3070	1.6692991833267e-07\\
3075	1.66456081363009e-07\\
3080	1.65984298305943e-07\\
3085	1.65514628669428e-07\\
3090	1.65047047140377e-07\\
3095	1.64581440920131e-07\\
3100	1.64118007628389e-07\\
3105	1.6365663535467e-07\\
3110	1.63197148239647e-07\\
3115	1.62739850928517e-07\\
3120	1.62284535587531e-07\\
3125	1.61831183120853e-07\\
3130	1.61379809071605e-07\\
3135	1.60930408110715e-07\\
3140	1.60483006439449e-07\\
3145	1.60037553431636e-07\\
3150	1.59594037985045e-07\\
3155	1.59152381051797e-07\\
3160	1.58712691877838e-07\\
3165	1.58274927386515e-07\\
3170	1.57839056047493e-07\\
3175	1.57405094292073e-07\\
3180	1.56973016807172e-07\\
3185	1.56542687257399e-07\\
3190	1.56114276173014e-07\\
3195	1.55687734260112e-07\\
3200	1.55263029100183e-07\\
3205	1.54840125166089e-07\\
3210	1.54418941189505e-07\\
3215	1.53999639262992e-07\\
3220	1.53582170980826e-07\\
3225	1.53166391125836e-07\\
3230	1.52752437365678e-07\\
3235	1.52340148051877e-07\\
3240	1.51929708369636e-07\\
3245	1.51520970437247e-07\\
3250	1.51114020408016e-07\\
3255	1.50708643786857e-07\\
3260	1.50305123014505e-07\\
3265	1.49903161439369e-07\\
3270	1.49502918045385e-07\\
3275	1.49104470992256e-07\\
3280	1.4870759024177e-07\\
3285	1.48312391257122e-07\\
3290	1.47918844284334e-07\\
3295	1.47526861837832e-07\\
3300	1.47136617556498e-07\\
3305	1.46747941798253e-07\\
3310	1.46360877195661e-07\\
3315	1.45975407761512e-07\\
3320	1.45591578792903e-07\\
3325	1.45209358315412e-07\\
3330	1.44828609105474e-07\\
3335	1.44449468386654e-07\\
3340	1.44071919283561e-07\\
3345	1.43695867649285e-07\\
3350	1.43321419621145e-07\\
3355	1.42948434866952e-07\\
3360	1.42577025741275e-07\\
3365	1.42207085218615e-07\\
3370	1.41838694567298e-07\\
3375	1.41471826253792e-07\\
3380	1.4110634927178e-07\\
3385	1.4074249010676e-07\\
3390	1.40379957880299e-07\\
3395	1.40018933780794e-07\\
3400	1.39659453335383e-07\\
3405	1.39301326029795e-07\\
3410	1.38944718397482e-07\\
3415	1.38589379083953e-07\\
3420	1.38235614954851e-07\\
3425	1.37883305217912e-07\\
3430	1.37532320199085e-07\\
3435	1.37182795345581e-07\\
3440	1.36834608532865e-07\\
3445	1.36487805946217e-07\\
3450	1.36142461304445e-07\\
3455	1.35798355671568e-07\\
3460	1.35455683114571e-07\\
3465	1.35114469834718e-07\\
3470	1.34774433391272e-07\\
3475.00000000001	1.34435818921474e-07\\
3480	1.34098570914176e-07\\
3485	1.33762712462016e-07\\
3490	1.33427987325519e-07\\
3495	1.33094711696202e-07\\
3500	1.32762715931989e-07\\
3505	1.32432020905071e-07\\
3510	1.32102560890246e-07\\
3515	1.31774460232492e-07\\
3520.00000000001	1.31447560836051e-07\\
3525	1.31121991930882e-07\\
3530	1.30797673580929e-07\\
3535	1.30474682613624e-07\\
3540	1.30152916444359e-07\\
3545	1.2983231645336e-07\\
3550	1.29512997659731e-07\\
3555	1.2919494984942e-07\\
3560	1.2887803180206e-07\\
3565	1.28562427814671e-07\\
3570	1.28248035302647e-07\\
3575	1.27934830729259e-07\\
3580	1.27622881151979e-07\\
3585	1.27311938324936e-07\\
3590	1.27002433458756e-07\\
3595	1.26693977531289e-07\\
3600	1.26386649146326e-07\\
3605	1.2608050825591e-07\\
3610	1.25775645010151e-07\\
3615	1.25471848910763e-07\\
3620	1.25169214548748e-07\\
3625	1.24867729045519e-07\\
3630	1.24567369308437e-07\\
3635	1.24268207724043e-07\\
3640	1.23970139043195e-07\\
3645	1.23673217888865e-07\\
3650	1.2337736121637e-07\\
3655	1.23082664060803e-07\\
3660	1.22788982093169e-07\\
3665	1.22496485843726e-07\\
3670	1.2220504208571e-07\\
3675	1.21914734751982e-07\\
3680	1.21625400417713e-07\\
3685	1.21337258018883e-07\\
3690	1.21050172996462e-07\\
3695	1.20764104938331e-07\\
3700	1.20479234588799e-07\\
3705	1.20195247532706e-07\\
3710	1.19912428875324e-07\\
3715	1.19630470418741e-07\\
3720	1.19349761629195e-07\\
3725	1.19069969883867e-07\\
3730	1.1879127370662e-07\\
3735	1.18513639346674e-07\\
3740.00000000001	1.18236922475035e-07\\
3745	1.17961239443076e-07\\
3750	1.17686525857863e-07\\
3755	1.1741295891099e-07\\
3760	1.17140324995546e-07\\
3765	1.16868558830418e-07\\
3770.00000000001	1.16597931754114e-07\\
3775	1.16328231491991e-07\\
3780	1.16059549970515e-07\\
3785	1.15791821908573e-07\\
3790	1.15525065513822e-07\\
3795.00000000001	1.15259264354961e-07\\
3800	1.14994425093329e-07\\
3805	1.1473056771294e-07\\
3810	1.14467540335283e-07\\
3815	1.14205720880278e-07\\
3820.00000000001	1.13944488067119e-07\\
3825	1.13684464508879e-07\\
3830	1.13425306480508e-07\\
3835	1.13167084592192e-07\\
3840.00000000001	1.12909732230548e-07\\
3845	1.12653322670298e-07\\
3850	1.12397833706979e-07\\
3855	1.12143261343789e-07\\
3860	1.11889669973664e-07\\
3865	1.11636822008876e-07\\
3870	1.11385013656928e-07\\
3875	1.11133996671953e-07\\
3880	1.10883850545917e-07\\
3885	1.10634664984843e-07\\
3890	1.10386277896168e-07\\
3895	1.10138890452305e-07\\
3900	1.0989232457348e-07\\
3905	1.09646596246904e-07\\
3910	1.0940163708284e-07\\
3915	1.09157703764851e-07\\
3920	1.08914615104538e-07\\
3925	1.08672182808078e-07\\
3930	1.08430866063713e-07\\
3935	1.08190202574576e-07\\
3940	1.0795053340118e-07\\
3945	1.07711470853644e-07\\
3950	1.07473406174563e-07\\
3955	1.07236052926396e-07\\
3960	1.06999668236796e-07\\
3965	1.06764065144205e-07\\
3970	1.06529134402678e-07\\
3975	1.06295203305962e-07\\
3980	1.06062041371757e-07\\
3985	1.05829677021774e-07\\
3990	1.05597998789619e-07\\
3995	1.05367238489862e-07\\
4000	1.05137202943695e-07\\
4005	1.04907895259743e-07\\
4010	1.04679490409154e-07\\
4015	1.04451878257805e-07\\
4020	1.04224909147632e-07\\
4025	1.03998819334095e-07\\
4030.00000000001	1.0377359327407e-07\\
4035	1.03548940533216e-07\\
4040	1.03325092926099e-07\\
4045	1.03102114401565e-07\\
4050	1.02879756713747e-07\\
4055	1.02658256118104e-07\\
4060.00000000001	1.02437583748838e-07\\
4065	1.02217481146027e-07\\
4070	1.01998120172198e-07\\
4075	1.01779592753814e-07\\
4080.00000000001	1.0156177987497e-07\\
4085	1.01344634462208e-07\\
4090	1.01128282636864e-07\\
4095.00000000001	1.00912642686524e-07\\
4100	1.00697684857209e-07\\
4105	1.00483501963566e-07\\
4110	1.00270037606265e-07\\
4115	1.00057207408355e-07\\
4120	9.98451961109481e-08\\
4125	9.9633787442599e-08\\
4130.00000000001	9.94232083328938e-08\\
4135	9.92131909960392e-08\\
4140.00000000001	9.90039006332211e-08\\
4145.00000000001	9.87954611453289e-08\\
4150	9.85875288073147e-08\\
4155.00000000001	9.83803469800649e-08\\
4160	9.81738428329494e-08\\
4165	9.79680181423249e-08\\
4170	9.77628240583783e-08\\
4175	9.75583787088396e-08\\
4180	9.7354532435645e-08\\
4185	9.71514064751488e-08\\
4190	9.69489364344156e-08\\
4195.00000000001	9.67472177926255e-08\\
4200	9.65460058566237e-08\\
4205	9.63455528690816e-08\\
4210	9.61457007342404e-08\\
4215	9.59464552252598e-08\\
4220	9.57479850960396e-08\\
4225	9.55500412125331e-08\\
4230	9.53527918845509e-08\\
4235	9.51561847095661e-08\\
4240	9.49602556588048e-08\\
4245.00000000001	9.47648768345744e-08\\
4250	9.45702085530797e-08\\
4255	9.43760904981161e-08\\
4260	9.4182769139195e-08\\
4265.00000000001	9.39898447960275e-08\\
4270	9.37975954684589e-08\\
4275	9.36060033929209e-08\\
4280	9.34150747866625e-08\\
4285	9.32245698415102e-08\\
4290.00000000001	9.30349388639229e-08\\
4295.00000000001	9.28457337678878e-08\\
4300.00000000001	9.26572605308707e-08\\
4305	9.24692331594201e-08\\
4310	9.22819842763545e-08\\
4315	9.20951528371462e-08\\
4320	9.1909109212196e-08\\
4325	9.17234084241158e-08\\
4330	9.15385069966134e-08\\
4335	9.13540225688791e-08\\
4340.00000000001	9.11702846551066e-08\\
4345.00000000001	9.09869357634819e-08\\
4350.00000000001	9.08043404912463e-08\\
4355.00000000001	9.06222217267328e-08\\
4360	9.04407846391564e-08\\
4365.00000000001	9.02598218388563e-08\\
4370	9.00795180669434e-08\\
4375	8.98996201925683e-08\\
4380	8.97204119887362e-08\\
4385.00000000001	8.95418121871216e-08\\
4390	8.93636826759802e-08\\
4395.00000000001	8.91860301166503e-08\\
4400	8.90091063077136e-08\\
4405	8.88327331693971e-08\\
4410.00000000001	8.86566882130069e-08\\
4415	8.84813378121407e-08\\
4420	8.83065487400358e-08\\
4425.00000000001	8.81323245494059e-08\\
4430	8.79585355662016e-08\\
4435	8.77853971736898e-08\\
4440	8.76126859949977e-08\\
4445	8.74405428064051e-08\\
4450	8.72690888442662e-08\\
4455	8.70978928979583e-08\\
4460	8.6927324893793e-08\\
4465	8.67573537455257e-08\\
4470	8.65879785649781e-08\\
4475	8.64189018123795e-08\\
4480	8.62505182830375e-08\\
4485	8.60825184467729e-08\\
4490	8.59151332299745e-08\\
4495	8.5748293354726e-08\\
4500	8.55817332556796e-08\\
4505	8.54159294405576e-08\\
4510	8.5250496884015e-08\\
4515	8.50856434198022e-08\\
4520	8.49212473674754e-08\\
4525	8.47573291551384e-08\\
4530	8.45939389648719e-08\\
4535	8.44310781289437e-08\\
4540	8.42686369573187e-08\\
4545	8.4106736242262e-08\\
4550	8.39453324630313e-08\\
4555	8.37843949774708e-08\\
4560	8.36239189005994e-08\\
4565	8.34639459768028e-08\\
4570	8.33044215831081e-08\\
4575	8.31455211347531e-08\\
4580	8.29869279961314e-08\\
4585	8.28287838317011e-08\\
4590	8.26712489576663e-08\\
4595.00000000001	8.25141430738086e-08\\
4600	8.23574626274138e-08\\
4605.00000000001	8.22013359602636e-08\\
4610	8.20455952066368e-08\\
4615	8.18903589383524e-08\\
4620.00000000001	8.17354752769006e-08\\
4625	8.15811200816086e-08\\
4630	8.14273239946317e-08\\
4635	8.12739768818461e-08\\
4640	8.11209468487561e-08\\
4645	8.09684568281455e-08\\
4650.00000000001	8.08164215548856e-08\\
4655	8.06647859619147e-08\\
4660	8.05136362025393e-08\\
4665.00000000001	8.03627941969864e-08\\
4670.00000000001	8.02125641463646e-08\\
4675	8.00627941721643e-08\\
4680.00000000001	7.99133590412282e-08\\
4685	7.97643946448545e-08\\
4690.00000000001	7.96159116411844e-08\\
4695	7.94678141069483e-08\\
4700.00000000001	7.93201744286875e-08\\
4705	7.91729428684107e-08\\
4710	7.90262215666359e-08\\
4715	7.88797844819555e-08\\
4720.00000000001	7.87338403362981e-08\\
4725.00000000001	7.8588191509965e-08\\
4730	7.84432376832455e-08\\
4735	7.82985161151828e-08\\
4740.00000000001	7.81542479622033e-08\\
4745	7.8010375936799e-08\\
4750	7.78669213552519e-08\\
4755	7.77239654858875e-08\\
4760	7.75813613351773e-08\\
4765.00000000001	7.74391426539012e-08\\
4770	7.72974733109776e-08\\
4775	7.7156030009462e-08\\
4780	7.7015087640575e-08\\
4785	7.68744614632056e-08\\
4790	7.67343291130373e-08\\
4795	7.65946359670977e-08\\
4800	7.64552510190697e-08\\
4805	7.63163026107349e-08\\
4810	7.61777139146602e-08\\
4815	7.60395550969406e-08\\
4820	7.59017697582465e-08\\
4825	7.5764358342667e-08\\
4830	7.56273186297563e-08\\
4835	7.54908877631521e-08\\
4840	7.53546265386262e-08\\
4845	7.5218760553497e-08\\
4850	7.5083246287022e-08\\
4855	7.494827558574e-08\\
4860	7.48136574912906e-08\\
4865	7.46793107353483e-08\\
4870	7.45453205830415e-08\\
4875	7.4411864225965e-08\\
4880	7.42787604757211e-08\\
4885	7.41457979458459e-08\\
4890.00000000001	7.40134797894143e-08\\
4895	7.38814551759504e-08\\
4900	7.3749603313189e-08\\
4905.00000000001	7.36183451977013e-08\\
4910	7.34874339158864e-08\\
4915.00000000001	7.33568867872237e-08\\
4920	7.32265226233153e-08\\
4925	7.30968867657111e-08\\
4930.00000000001	7.29673157451315e-08\\
4935	7.28382385517534e-08\\
4940	7.27095055275129e-08\\
4945	7.25810060941967e-08\\
4950.00000000001	7.24528352868958e-08\\
4955	7.23251902812194e-08\\
4960.00000000001	7.21977997386603e-08\\
4965	7.20708874801801e-08\\
4970	7.19441834995393e-08\\
4975.00000000001	7.18179231640193e-08\\
4980	7.16920318666325e-08\\
4985	7.15664212336264e-08\\
4990	7.1441173865594e-08\\
4995	7.13162475740602e-08\\
5000	7.11916658957535e-08\\
5005.00000000001	7.10674101789266e-08\\
5010	7.09434919698992e-08\\
5015.00000000001	7.0820022735063e-08\\
5020	7.06969021102566e-08\\
5025.00000000001	7.05739835460406e-08\\
5030	7.04513816174312e-08\\
5035	7.03291584969179e-08\\
5040.00000000001	7.02074269831598e-08\\
5045	7.00858633351231e-08\\
5050.00000000001	6.99646260926556e-08\\
5055	6.9843822281257e-08\\
5060.00000000001	6.97232489521583e-08\\
5065	6.96029895941309e-08\\
5070.00000000001	6.94830908365418e-08\\
5075	6.93635673343352e-08\\
5080.00000000001	6.92443169469926e-08\\
5085.00000000001	6.91252903806117e-08\\
5090.00000000001	6.90066870312478e-08\\
5095	6.88885255506478e-08\\
5100	6.87703614055124e-08\\
5105	6.86527501514433e-08\\
5110	6.85353431784109e-08\\
5115	6.84183780741421e-08\\
5120.00000000001	6.83016621039201e-08\\
5125	6.81853644657337e-08\\
5130	6.80692129328974e-08\\
5135.00000000001	6.79535552272625e-08\\
5140.00000000001	6.78380382979071e-08\\
5145	6.77228939593987e-08\\
5150.00000000001	6.76081071127044e-08\\
5155	6.74935103361918e-08\\
5160	6.73791760164022e-08\\
5165	6.72653133193535e-08\\
5170	6.7151647797914e-08\\
5175	6.70383943912611e-08\\
5180	6.69253656937485e-08\\
5185	6.68126221015087e-08\\
5190	6.67001693877012e-08\\
5195.00000000001	6.65879826833305e-08\\
5200	6.64761090618527e-08\\
5205	6.63646257947904e-08\\
5210.00000000001	6.62533050643788e-08\\
5215	6.61424901515772e-08\\
5220	6.60317436285141e-08\\
5225	6.59213128528791e-08\\
5230	6.58113958884599e-08\\
5235	6.57015597482768e-08\\
5240	6.55920238123998e-08\\
5245	6.54828640200833e-08\\
5250	6.537391605832e-08\\
5255	6.52651532817572e-08\\
5260	6.51568416998315e-08\\
5265	6.50487561593139e-08\\
5270.00000000001	6.49409028774528e-08\\
5275	6.48334017583352e-08\\
5280	6.47261235720009e-08\\
5285	6.4619146478151e-08\\
5290	6.45125837195338e-08\\
5295	6.44061519672334e-08\\
5300	6.42999706812474e-08\\
5305	6.41941491075215e-08\\
5310	6.40886055336409e-08\\
5315	6.39832986593092e-08\\
5320	6.38781747497319e-08\\
5325	6.37735251274307e-08\\
5330	6.36689283517455e-08\\
5335.00000000001	6.35646708602166e-08\\
5340	6.34607006944065e-08\\
5345	6.3356970780859e-08\\
5350	6.32535890332519e-08\\
5355.00000000001	6.31504524228888e-08\\
5360.00000000001	6.30476622021092e-08\\
5365	6.29449883327027e-08\\
5370	6.28424805526606e-08\\
5375	6.2740477257961e-08\\
5380.00000000001	6.26386569280158e-08\\
5385	6.25368898887757e-08\\
5390.00000000001	6.24356752965128e-08\\
5395	6.2334513994955e-08\\
5400.00000000001	6.22338429678848e-08\\
5405.00000000001	6.21332669759056e-08\\
5410	6.20329450029544e-08\\
5415	6.19328357487347e-08\\
5420.00000000001	6.18330546764412e-08\\
5425	6.17335333963354e-08\\
5430	6.1634150672063e-08\\
5435	6.1535073481167e-08\\
5440.00000000001	6.14363329098922e-08\\
5445	6.13377668656768e-08\\
5450	6.12394015497842e-08\\
5455	6.11414536777488e-08\\
5460.00000000001	6.10436443615469e-08\\
5465.00000000001	6.09461676681632e-08\\
5470	6.08487877862274e-08\\
5475	6.07516885686721e-08\\
5480	6.06547501114107e-08\\
5485.00000000001	6.05583538870746e-08\\
5490.00000000001	6.04619585509169e-08\\
5495	6.03657546172087e-08\\
5500	6.02700134244571e-08\\
5505.00000000001	6.01743290751245e-08\\
5510	6.00789804572344e-08\\
5515	5.99838805293019e-08\\
5520.00000000001	5.98888942882069e-08\\
5525	5.97943490276975e-08\\
5530.00000000001	5.9699839738414e-08\\
5535.00000000001	5.96056244361876e-08\\
5540.00000000001	5.95116729229517e-08\\
5545.00000000001	5.94179057067379e-08\\
5550	5.93244688928962e-08\\
5555	5.92312132674522e-08\\
5560	5.91380455716716e-08\\
5565	5.90453232973687e-08\\
5570.00000000001	5.89526094607606e-08\\
5575	5.88603521478603e-08\\
5580	5.87681321384537e-08\\
5585	5.86762127774421e-08\\
5590	5.85845119083217e-08\\
5595	5.8493197396814e-08\\
5600	5.84019752558617e-08\\
5605	5.83108401563948e-08\\
5610	5.82200634369201e-08\\
5615.00000000001	5.81296024648737e-08\\
5620	5.80392076621194e-08\\
5625	5.79490961882811e-08\\
5630	5.78592396216493e-08\\
5635	5.7769545147579e-08\\
5640	5.76799656926141e-08\\
5645	5.75908623012822e-08\\
5650	5.75018659354498e-08\\
5655	5.74129410679803e-08\\
5660.00000000001	5.73243270629575e-08\\
5665	5.72361398276655e-08\\
5670	5.71478420141602e-08\\
5675.00000000001	5.70597915583449e-08\\
5680	5.69722429233365e-08\\
5685	5.68845650583682e-08\\
5690.00000000001	5.67972269216454e-08\\
5695.00000000001	5.67102840243195e-08\\
5700	5.6623278954504e-08\\
5705	5.65366082838636e-08\\
5710	5.64500854949301e-08\\
5715.00000000001	5.63638229422736e-08\\
5720.00000000001	5.62777779933299e-08\\
5725	5.61920101560531e-08\\
5730	5.61061934689632e-08\\
5735	5.60207973343552e-08\\
5740	5.59356743146111e-08\\
5745	5.58506072501075e-08\\
5750	5.57656867350431e-08\\
5755	5.56810544338759e-08\\
5760	5.55965677939696e-08\\
5765	5.55123542689274e-08\\
5770.00000000001	5.54282562070796e-08\\
5775.00000000001	5.53445236306516e-08\\
5780.00000000001	5.5260874542995e-08\\
5785.00000000001	5.51775252155551e-08\\
5790	5.50941412491568e-08\\
5795	5.50111272090703e-08\\
5800.00000000001	5.492835697396e-08\\
5805	5.48457341764675e-08\\
5810	5.47632712510904e-08\\
5815	5.46809877377541e-08\\
5820.00000000001	5.45989267131119e-08\\
5825	5.45170864008071e-08\\
5830.00000000001	5.4435281615639e-08\\
5835.00000000001	5.43539067088261e-08\\
5840	5.42726894536827e-08\\
5845	5.41914078056038e-08\\
5850.00000000001	5.41104943074799e-08\\
5855	5.40298548123985e-08\\
5860.00000000001	5.39493707485405e-08\\
5865.00000000001	5.3869094962522e-08\\
5870.00000000001	5.3788968390478e-08\\
5875	5.37088298280252e-08\\
5880.00000000001	5.36291531183508e-08\\
5885	5.35495865428004e-08\\
5890	5.34700639320817e-08\\
5895.00000000001	5.33908428579366e-08\\
5900.00000000001	5.33117856527099e-08\\
5905.00000000001	5.3232939833947e-08\\
5910.00000000001	5.31542148074493e-08\\
5915.00000000001	5.3075770001243e-08\\
5920	5.29974610863349e-08\\
5925.00000000001	5.29193524556604e-08\\
5930.00000000001	5.28413912626036e-08\\
5935	5.27635730662724e-08\\
5940	5.26860195471102e-08\\
5945	5.2608629008688e-08\\
5950	5.25314129973253e-08\\
5955	5.2454326660012e-08\\
5960	5.23775067762244e-08\\
5965.00000000001	5.23008765185296e-08\\
5970	5.22241916378618e-08\\
5975	5.21480028048416e-08\\
5980.00000000001	5.2071824185873e-08\\
5985	5.19957805700244e-08\\
5990	5.19199834236872e-08\\
5995.00000000001	5.18442808683517e-08\\
6000	5.17687976930859e-08\\
};
\addplot [color=black, line width=0.9pt, forget plot]
  table[row sep=crcr]{%
200	8.565008204102601e-06\\
6000	2.227232665603490e-08\\
};

\node[] at (axis cs: 650, 1e-7) {\small $\zeta=1$};
\node[] at (axis cs: 650, 3e-8) {\small $v(x)=\frac{1}{2+\cos(2\pi x)}$};
\end{axis}

\end{tikzpicture}%

%% file: plot_an_nob.tex
%
%
\begin{tikzpicture}

\begin{axis}[%
width=1.7in,
height=1.7in,
at={(0.769in,0.477in)},
scale only axis,
xmode=log,
xmin=100,
xmax=10000,
xminorticks=true,
ymode=log,
ymin=1e-08,
ymax=1e-04,
yminorticks=true,
axis background/.style={fill=white},
xlabel = {$N$},
ylabel = {Error}
]
\addplot [color=blue, only marks, mark=x, mark options={solid, blue}, forget plot]
  table[row sep=crcr]{%
200	1.89938313770455e-05\\
205	1.81940909995681e-05\\
210	1.74458793067434e-05\\
215	1.67447430001744e-05\\
220	1.60869025707111e-05\\
225	1.5468465968782e-05\\
230	1.48867813982712e-05\\
235	1.43384834039306e-05\\
240	1.38212162514861e-05\\
245	1.33326785729526e-05\\
250	1.28705256623452e-05\\
255	1.24331375179754e-05\\
260	1.20185719270083e-05\\
265	1.16251498254449e-05\\
270	1.12517016568336e-05\\
275	1.08966713607828e-05\\
280	1.05587928747042e-05\\
285	1.02371566261805e-05\\
290	9.93061425069719e-06\\
295	9.63813855459072e-06\\
300	9.35894678799621e-06\\
305	9.09231728019045e-06\\
310	8.83733817658162e-06\\
315	8.5933281124273e-06\\
320	8.35980202995578e-06\\
325	8.13607882554913e-06\\
330	7.92154716622306e-06\\
335	7.71570118551779e-06\\
340	7.51820970612549e-06\\
345	7.3285234485354e-06\\
350	7.14619070163458e-06\\
355	6.97082926115478e-06\\
360	6.802165853248e-06\\
365	6.63984604276634e-06\\
370	6.48348251730546e-06\\
375	6.33278203809872e-06\\
380	6.18746940461712e-06\\
385	6.04741107412821e-06\\
390	5.91223767720805e-06\\
395	5.78171437171094e-06\\
400	5.65562633836336e-06\\
405.000000000001	5.533771228583e-06\\
410	5.41606745319001e-06\\
415	5.30222172190342e-06\\
420	5.19206178384302e-06\\
425	5.08542730837114e-06\\
430.000000000001	4.98216668298745e-06\\
435.000000000001	4.88219105876375e-06\\
440	4.78534710524059e-06\\
445	4.69146401416154e-06\\
450	4.6004200722205e-06\\
455	4.51209983043377e-06\\
460	4.42639374198528e-06\\
465.000000000001	4.34325308362738e-06\\
470.000000000001	4.26254544372284e-06\\
475	4.18415103320413e-06\\
480	4.10798065764162e-06\\
485	4.03394941339563e-06\\
490	3.96197648822039e-06\\
495.000000000001	3.89200283001934e-06\\
500	3.82398557552399e-06\\
505.000000000001	3.75780300898576e-06\\
510	3.69338849370848e-06\\
515.000000000001	3.63067846143039e-06\\
520	3.56961217451435e-06\\
525	3.51013163801817e-06\\
530	3.45219542507991e-06\\
535	3.3957782867855e-06\\
540.000000000001	3.34078515629166e-06\\
545	3.28716764563808e-06\\
550	3.23487937703426e-06\\
555	3.18387597153524e-06\\
560	3.13411491603688e-06\\
565	3.08555540717847e-06\\
570	3.03818437941139e-06\\
575	2.99195834463362e-06\\
580.000000000001	2.94681933699082e-06\\
585.000000000001	2.90273291314591e-06\\
590.000000000001	2.85966597468601e-06\\
595	2.81758675546584e-06\\
600	2.77646471613657e-06\\
605.000000000001	2.73627052949088e-06\\
610	2.6969799113985e-06\\
615	2.65859682890835e-06\\
620.000000000001	2.62105854131889e-06\\
625.000000000001	2.58433995181662e-06\\
630.000000000001	2.54841696101238e-06\\
635	2.51326628752935e-06\\
640	2.47886556703492e-06\\
645	2.44519321435099e-06\\
650	2.41222840369204e-06\\
655	2.37995109486633e-06\\
660.000000000001	2.34836757129209e-06\\
665	2.31743941681018e-06\\
670.000000000001	2.28714158390097e-06\\
675.000000000001	2.25745675219713e-06\\
680	2.22836818286609e-06\\
685	2.199859732821e-06\\
690.000000000001	2.17191576057374e-06\\
695	2.14452119795538e-06\\
700	2.11766146884784e-06\\
705.000000000001	2.09132247475097e-06\\
710	2.06550449499687e-06\\
715.000000000001	2.04019278293721e-06\\
720.000000000001	2.01536135890912e-06\\
725	1.99099790609836e-06\\
730	1.96709053490452e-06\\
735	1.94362772054646e-06\\
740	1.92059827419655e-06\\
745	1.89799140182245e-06\\
750	1.87579665755777e-06\\
755	1.85400387464618e-06\\
760	1.83260323161072e-06\\
765.000000000001	1.8115879778513e-06\\
770.000000000001	1.7909651013337e-06\\
775	1.77070595674955e-06\\
780.000000000001	1.75080183728582e-06\\
785	1.73124438762606e-06\\
790.000000000001	1.71202541432436e-06\\
795	1.69313702835794e-06\\
800	1.67457152611128e-06\\
805	1.65632148463324e-06\\
810	1.63837964084479e-06\\
815	1.62073896659009e-06\\
820	1.60339264576592e-06\\
825.000000000001	1.58633404190312e-06\\
830	1.56956553354348e-06\\
835	1.55308075622251e-06\\
840.000000000001	1.53686433224109e-06\\
845.000000000001	1.52091041316638e-06\\
850	1.50521319208785e-06\\
855	1.48976717451177e-06\\
860.000000000001	1.47456684684855e-06\\
865.000000000001	1.45960697750525e-06\\
870.000000000001	1.44488238840168e-06\\
875	1.43038813216201e-06\\
880.000000000001	1.41611930026819e-06\\
885.000000000001	1.40207115340019e-06\\
890.000000000001	1.38823910678099e-06\\
895	1.3746186349195e-06\\
900.000000000001	1.36121109695075e-06\\
905	1.34801505202731e-06\\
910.000000000001	1.33501742927145e-06\\
915	1.32221417437073e-06\\
920.000000000001	1.30960139266279e-06\\
925	1.2971752036961e-06\\
930.000000000001	1.28493190354462e-06\\
935.000000000001	1.27286785445158e-06\\
940.000000000001	1.26097949304516e-06\\
945.000000000001	1.24926337030651e-06\\
950	1.23771611004742e-06\\
955.000000000001	1.22633443311315e-06\\
960.000000000001	1.21511511652628e-06\\
965	1.20405501236043e-06\\
970.000000000001	1.19315111590801e-06\\
975	1.18240038515793e-06\\
980.000000000001	1.17180755032642e-06\\
985.000000000001	1.16136598760086e-06\\
989.999999999999	1.15106878517679e-06\\
995.000000000001	1.14091329028732e-06\\
1000	1.1308968002055e-06\\
1005	1.12101680116439e-06\\
1010	1.11127076074524e-06\\
1015	1.10165618849578e-06\\
1020	1.09217073784862e-06\\
1025	1.08281203425875e-06\\
1030	1.0735777864479e-06\\
1035	1.06446574976715e-06\\
1040	1.05547377970972e-06\\
1045	1.04659967248288e-06\\
1050	1.0378413932699e-06\\
1055	1.02919685773806e-06\\
1060	1.02066410700985e-06\\
1065	1.01224114912313e-06\\
1070	1.00393488700057e-06\\
1075	9.95734922959457e-07\\
1080	9.87638940630475e-07\\
1085	9.79645164989052e-07\\
1090	9.71751819456301e-07\\
1095	9.63957268229621e-07\\
1100	9.56259740503285e-07\\
1105	9.48657664423536e-07\\
1110	9.41149441713306e-07\\
1115	9.33733490082744e-07\\
1120	9.26408269430468e-07\\
1125	9.1917226296978e-07\\
1130	9.12024023191903e-07\\
1135	9.04962134118392e-07\\
1140	8.97985107162213e-07\\
1145	8.91091671118004e-07\\
1150	8.84280371815649e-07\\
1155	8.77549929612087e-07\\
1160	8.70899061311547e-07\\
1165	8.64326771266022e-07\\
1170	8.57838556855839e-07\\
1175	8.51426005477052e-07\\
1180	8.4508789766069e-07\\
1185	8.38823101867448e-07\\
1190	8.32630418168278e-07\\
1195	8.26508776308188e-07\\
1200	8.20457010775045e-07\\
1205	8.14474120369723e-07\\
1210	8.08558979326079e-07\\
1215	8.0271059910153e-07\\
1220	7.96927899671118e-07\\
1225	7.91209955997019e-07\\
1230	7.85555722027098e-07\\
1235	7.79964262731526e-07\\
1240	7.74434644412736e-07\\
1245	7.68965917385955e-07\\
1250	7.63557192140497e-07\\
1255	7.5820757428069e-07\\
1260	7.52916187840568e-07\\
1265	7.47682178170451e-07\\
1270	7.42504694617452e-07\\
1275	7.37382936044639e-07\\
1280	7.32318722773684e-07\\
1285	7.27311658899054e-07\\
1290	7.22357828886188e-07\\
1295	7.17456426491125e-07\\
1300	7.12606685882022e-07\\
1305	7.07807898514546e-07\\
1310	7.03059372941795e-07\\
1315	6.98360308470925e-07\\
1320	6.93710076271615e-07\\
1325	6.89107981122206e-07\\
1330	6.8455331092565e-07\\
1335	6.80045426637577e-07\\
1340	6.75583693876547e-07\\
1345	6.71167446286703e-07\\
1350	6.66796053705454e-07\\
1355	6.62468962131513e-07\\
1360	6.58185489221807e-07\\
1365	6.5394505255334e-07\\
1370	6.49747090353259e-07\\
1375	6.45591028414217e-07\\
1380	6.4147626965827e-07\\
1385	6.37402280512234e-07\\
1390	6.33368521629762e-07\\
1395	6.29374415250794e-07\\
1400	6.25419496858015e-07\\
1405	6.21503163822368e-07\\
1410	6.17628816756977e-07\\
1415	6.13792720205453e-07\\
1420	6.09993640710016e-07\\
1425	6.06231145727775e-07\\
1430	6.0250470568235e-07\\
1435	5.98813870489323e-07\\
1440	5.95158231364579e-07\\
1445	5.91537256511287e-07\\
1450	5.87950544694849e-07\\
1455	5.8439765271423e-07\\
1460	5.80878145806096e-07\\
1465	5.77391556344509e-07\\
1470	5.73937569692262e-07\\
1475	5.7051568647104e-07\\
1480	5.67125522765721e-07\\
1485	5.63766671568544e-07\\
1490	5.60438774943606e-07\\
1495	5.57141386359206e-07\\
1500	5.53874198283566e-07\\
1505	5.50636811036397e-07\\
1510	5.47428845143472e-07\\
1515	5.44249937339814e-07\\
1520	5.41099725692717e-07\\
1525	5.37977875803008e-07\\
1530	5.34884024849802e-07\\
1535	5.31817837545745e-07\\
1540	5.28779027897386e-07\\
1545	5.25767208880979e-07\\
1550	5.22782046319392e-07\\
1555	5.19824953526537e-07\\
1560	5.16895669377959e-07\\
1565	5.13992082673554e-07\\
1570	5.11113835477416e-07\\
1575	5.08260640241787e-07\\
1580	5.05432246944437e-07\\
1585	5.02628302312402e-07\\
1590	4.99848653134904e-07\\
1595	4.97092800921805e-07\\
1600	4.94360646641212e-07\\
1605	4.91651839018558e-07\\
1610	4.88966094280841e-07\\
1615	4.86303224134232e-07\\
1620	4.83662840000676e-07\\
1625	4.81044812428166e-07\\
1630	4.7844880324277e-07\\
1635	4.7587456086795e-07\\
1640	4.73321901894863e-07\\
1645	4.70790591622361e-07\\
1650	4.68280251020304e-07\\
1655	4.65790776837949e-07\\
1660	4.63321884192069e-07\\
1665	4.60873356145087e-07\\
1670	4.58444974649197e-07\\
1675	4.56036468365895e-07\\
1680	4.53647682308045e-07\\
1685	4.51278318935877e-07\\
1690	4.48928246132852e-07\\
1695	4.46597164804885e-07\\
1700	4.44284962819452e-07\\
1705	4.41991374833251e-07\\
1710	4.39716184130745e-07\\
1715	4.3745925348837e-07\\
1720	4.35220340877506e-07\\
1725	4.33001574595693e-07\\
1730	4.30800738193682e-07\\
1735	4.286173285184e-07\\
1740	4.26451099544423e-07\\
1745	4.24302005752608e-07\\
1750	4.22169694092034e-07\\
1755	4.20054048433372e-07\\
1760	4.17954900022721e-07\\
1765	4.15872126069417e-07\\
1770	4.13805493426623e-07\\
1775	4.11754801810104e-07\\
1780	4.09719904670425e-07\\
1785	4.07700732507621e-07\\
1790	4.05697019756346e-07\\
1795	4.03708627194633e-07\\
1800	4.01735394950365e-07\\
1805	3.99777192017225e-07\\
1810	3.97833771481615e-07\\
1815	3.95905118466544e-07\\
1820	3.93991034686181e-07\\
1825	3.92091307199749e-07\\
1830	3.90205845635094e-07\\
1835	3.88334511214339e-07\\
1840	3.86477119862505e-07\\
1845	3.84633601413498e-07\\
1850	3.82803749809924e-07\\
1855	3.80987449144498e-07\\
1860	3.79184523113807e-07\\
1865	3.77394919093277e-07\\
1870	3.75618428583024e-07\\
1875	3.73854967428145e-07\\
1880	3.72104389301242e-07\\
1885	3.70366538104961e-07\\
1890	3.68641371872869e-07\\
1895	3.66928681883039e-07\\
1900	3.65228409737739e-07\\
1905	3.63540387793293e-07\\
1910	3.6186448193476e-07\\
1915	3.60200652638199e-07\\
1920	3.58550400925139e-07\\
1925	3.56912259080744e-07\\
1930	3.55285764142011e-07\\
1935	3.53670798647343e-07\\
1940	3.5206725979009e-07\\
1945	3.50475049426535e-07\\
1950	3.48894027668578e-07\\
1955	3.47324105698377e-07\\
1960	3.45765214904148e-07\\
1965	3.44217226500021e-07\\
1970	3.42679991938155e-07\\
1975	3.41153472804834e-07\\
1980	3.39637450386121e-07\\
1985	3.38131988852907e-07\\
1990	3.36636945652558e-07\\
1995	3.35152169572694e-07\\
2000	3.33677567576629e-07\\
2005	3.32213085485478e-07\\
2010	3.30758593625191e-07\\
2015	3.29314056024543e-07\\
2020	3.27879261297071e-07\\
2025	3.26454259624853e-07\\
2030	3.25038921111797e-07\\
2035	3.2363309898642e-07\\
2040	3.22236732408499e-07\\
2045	3.20849740553797e-07\\
2050	3.19472099885587e-07\\
2055	3.18103637209077e-07\\
2060	3.16744291684046e-07\\
2065	3.15394015792947e-07\\
2070	3.1405264122597e-07\\
2075	3.12720217943152e-07\\
2080	3.11396585406243e-07\\
2085	3.10081614829372e-07\\
2090	3.08775373047965e-07\\
2095	3.0747763379857e-07\\
2100	3.06188415732933e-07\\
2105	3.04907662895815e-07\\
2110	3.03635163678706e-07\\
2115	3.02370958493725e-07\\
2120	3.01115012923958e-07\\
2125	2.99867154884837e-07\\
2130	2.9862734196584e-07\\
2135	2.9739551243857e-07\\
2140	2.96171546398938e-07\\
2145	2.94956490565213e-07\\
2150	2.93749718549563e-07\\
2155	2.92550590419438e-07\\
2160	2.91359125936807e-07\\
2165	2.90175210526656e-07\\
2170	2.88998743158686e-07\\
2175	2.87829767575687e-07\\
2180	2.8666806639599e-07\\
2185	2.85513777731339e-07\\
2190	2.84366691083449e-07\\
2195	2.83226733843733e-07\\
2200	2.82093961301299e-07\\
2205	2.80968179833252e-07\\
2210	2.79849414752675e-07\\
2215	2.78737631864701e-07\\
2220	2.77632681289219e-07\\
2225	2.76534596777012e-07\\
2230	2.75443209130088e-07\\
2235	2.74358572083244e-07\\
2240	2.73280607254733e-07\\
2245	2.72209209839502e-07\\
2250	2.71144342756102e-07\\
2255	2.70085961373568e-07\\
2260	2.69034049926731e-07\\
2265	2.6798853647314e-07\\
2270	2.66949294225327e-07\\
2275	2.65916364705632e-07\\
2280	2.64889586265582e-07\\
2285	2.63869032846032e-07\\
2290	2.62854683796831e-07\\
2295	2.61846312410441e-07\\
2300	2.60843991073401e-07\\
2305	2.598476749327e-07\\
2310	2.58857233870202e-07\\
2315	2.57872671438619e-07\\
2320	2.56894014505349e-07\\
2325	2.55921097869205e-07\\
2330	2.54953957279369e-07\\
2335	2.53992441301421e-07\\
2340	2.53036598563128e-07\\
2345	2.52086372221072e-07\\
2350	2.51141686113954e-07\\
2355	2.50202562224189e-07\\
2360	2.49268875096575e-07\\
2365	2.48340585429219e-07\\
2370	2.4741772075565e-07\\
2375	2.46500166722896e-07\\
2380	2.45587933322966e-07\\
2385	2.4468099457664e-07\\
2390	2.43779244346598e-07\\
2395	2.42882632894848e-07\\
2400	2.41991194416258e-07\\
2405	2.41104885168042e-07\\
2410	2.40224057757032e-07\\
2415	2.39348910913151e-07\\
2420	2.38478691017008e-07\\
2425	2.3761347600626e-07\\
2430	2.36753135762768e-07\\
2435	2.35897600786572e-07\\
2440	2.35046885954659e-07\\
2445	2.34200954629671e-07\\
2450	2.33359781498521e-07\\
2455	2.32523307053256e-07\\
2460	2.3169150398239e-07\\
2465	2.30864392269936e-07\\
2470	2.30041833804151e-07\\
2475	2.29223801273548e-07\\
2480	2.284103328698e-07\\
2485	2.27601383295806e-07\\
2490	2.2679685773852e-07\\
2495	2.25996835689912e-07\\
2500	2.25201159942401e-07\\
2505	2.24409835380968e-07\\
2510	2.23622896200482e-07\\
2515	2.22840206509645e-07\\
2520	2.22061784738159e-07\\
2525	2.21287617785393e-07\\
2530	2.20517673898968e-07\\
2535	2.19751826957548e-07\\
2540	2.18990254818863e-07\\
2545	2.18232657944739e-07\\
2550	2.17479234176921e-07\\
2555	2.1672981342924e-07\\
2560	2.15984467422103e-07\\
2565	2.15243091128414e-07\\
2570	2.14505671891629e-07\\
2575	2.13772227697362e-07\\
2580	2.13042573582456e-07\\
2585	2.12316947578728e-07\\
2590	2.11595107879603e-07\\
2595	2.10877127315712e-07\\
2600	2.10162814040516e-07\\
2605	2.09452369004381e-07\\
2610	2.08745679630695e-07\\
2615	2.0804270461916e-07\\
2620	2.07343433755725e-07\\
2625	2.06647780220947e-07\\
2630	2.05955744903008e-07\\
2635	2.05267387975993e-07\\
2640	2.04582584428792e-07\\
2645	2.0390135002657e-07\\
2650	2.03223614159143e-07\\
2655	2.02549436334465e-07\\
2660	2.01878728844918e-07\\
2665	2.01211405981283e-07\\
2670	2.00547622286606e-07\\
2675	1.99887250973418e-07\\
2680	1.99230273389972e-07\\
2685	1.98576594501176e-07\\
2690	1.97926263156845e-07\\
2695	1.97279256708427e-07\\
2700	1.96635568494585e-07\\
2705	1.95995182750153e-07\\
2710	1.95358028198811e-07\\
2715	1.94724042890115e-07\\
2720	1.94093342065216e-07\\
2725	1.93466487941052e-07\\
2730	1.92842920299086e-07\\
2735	1.92222447736867e-07\\
2740	1.91605123545102e-07\\
2745	1.90990885329256e-07\\
2750	1.90379737974311e-07\\
2755	1.89771590664023e-07\\
2760	1.89166476927127e-07\\
2765	1.88564430958493e-07\\
2770	1.87965313092064e-07\\
2775	1.87369141313454e-07\\
2780	1.867759840124e-07\\
2785	1.86185699746488e-07\\
2790	1.85598297619549e-07\\
2795	1.85013868225781e-07\\
2800	1.8443230676013e-07\\
2805	1.83853472668361e-07\\
2810	1.83277528043036e-07\\
2815	1.82704498197239e-07\\
2820	1.82134161974545e-07\\
2825	1.815666450522e-07\\
2830	1.81001863941432e-07\\
2835	1.80439795549603e-07\\
2840	1.7988048028883e-07\\
2845	1.79323823124023e-07\\
2850	1.78769942626999e-07\\
2855	1.7821864828349e-07\\
2860	1.77670032686095e-07\\
2865	1.77124033218234e-07\\
2870	1.76580648103552e-07\\
2875	1.76039856913945e-07\\
2880	1.75501657873056e-07\\
2885	1.74966003685384e-07\\
2890	1.74432937649627e-07\\
2895	1.73902430011808e-07\\
2900	1.73374372414159e-07\\
2905	1.72848855672925e-07\\
2910	1.72325780756211e-07\\
2915	1.71805206283793e-07\\
2920	1.71287077410653e-07\\
2925	1.707713732646e-07\\
2930	1.70258067200279e-07\\
2935	1.69747227385386e-07\\
2940	1.69238726366317e-07\\
2945	1.68732664729276e-07\\
2950	1.68228892150068e-07\\
2955	1.67727471467316e-07\\
2960	1.67228353387117e-07\\
2965	1.66731611406234e-07\\
2970	1.66237148269133e-07\\
2975	1.65744941327261e-07\\
2980	1.65255093831362e-07\\
2985	1.64767451904524e-07\\
2990	1.64282019987638e-07\\
2995	1.63798860919328e-07\\
3000	1.63317849910527e-07\\
3005	1.62839140172011e-07\\
3010	1.62362602029731e-07\\
3015	1.61888193739301e-07\\
3020	1.61415945276744e-07\\
3025	1.60945872407225e-07\\
3030	1.60477916066881e-07\\
3035	1.60012067818016e-07\\
3040	1.5954835852483e-07\\
3045	1.59086727347102e-07\\
3050	1.58627132318401e-07\\
3055	1.58169669806085e-07\\
3060	1.57714284965138e-07\\
3065	1.5726096336266e-07\\
3070	1.56809532914082e-07\\
3075	1.5636024852661e-07\\
3080	1.55912939892033e-07\\
3085	1.55467638096596e-07\\
3090	1.55024311165875e-07\\
3095	1.54582884714927e-07\\
3100	1.54144061514927e-07\\
3105	1.53707222949606e-07\\
3110	1.53272184055808e-07\\
3115	1.5283920506981e-07\\
3120	1.5240807904604e-07\\
3125	1.51978861939739e-07\\
3130	1.51551519333992e-07\\
3135	1.51126026359805e-07\\
3140	1.50702423429294e-07\\
3145	1.50280682786885e-07\\
3150	1.49860774900645e-07\\
3155	1.49442634711505e-07\\
3160	1.49026348816861e-07\\
3165	1.48611897898832e-07\\
3170	1.48199249094816e-07\\
3175	1.47788392634851e-07\\
3180	1.47379347836818e-07\\
3185	1.46971929515516e-07\\
3190	1.46566343062204e-07\\
3195	1.46162536518446e-07\\
3200	1.45760475023238e-07\\
3205	1.45360091519109e-07\\
3210	1.44961385117881e-07\\
3215	1.44564428206095e-07\\
3220	1.44169218119217e-07\\
3225	1.43775606753493e-07\\
3230	1.43383731554536e-07\\
3235	1.42993447305173e-07\\
3240	1.42604892117148e-07\\
3245	1.42217956744517e-07\\
3250	1.41832740441217e-07\\
3255	1.41448991852755e-07\\
3260	1.41067021175445e-07\\
3265	1.40686516880706e-07\\
3270	1.40307640394965e-07\\
3275	1.39930469433836e-07\\
3280	1.3955477662364e-07\\
3285	1.3918070562724e-07\\
3290	1.38808165184301e-07\\
3295	1.38437149521664e-07\\
3300	1.38067749455573e-07\\
3305	1.37699848634654e-07\\
3310	1.37333475036527e-07\\
3315	1.36968631103685e-07\\
3320	1.36605306177984e-07\\
3325	1.36243536230651e-07\\
3330	1.35883162499795e-07\\
3335	1.35524297339984e-07\\
3340	1.3516694030713e-07\\
3345	1.34811015684022e-07\\
3350	1.3445661073419e-07\\
3355	1.34103598448121e-07\\
3360	1.33752087849714e-07\\
3365	1.33401932833621e-07\\
3370	1.33053286166529e-07\\
3375	1.32706050370856e-07\\
3380	1.32360157945044e-07\\
3385	1.32015786968864e-07\\
3390	1.31672692527118e-07\\
3395	1.31331028052628e-07\\
3400	1.30990815527809e-07\\
3405	1.30651893526235e-07\\
3410	1.30314402602139e-07\\
3415	1.29978144247644e-07\\
3420	1.29643337176688e-07\\
3425	1.29309917218379e-07\\
3430	1.28977774460637e-07\\
3435	1.28647009267624e-07\\
3440	1.28317511061127e-07\\
3445	1.27989332465717e-07\\
3450	1.27662528548456e-07\\
3455	1.27336916344589e-07\\
3460	1.27012651951475e-07\\
3465	1.26689748025655e-07\\
3470	1.26368005171074e-07\\
3475.00000000001	1.26047582815758e-07\\
3480	1.2572847141179e-07\\
3485	1.25410658746716e-07\\
3490	1.25093946090615e-07\\
3495	1.24778594789987e-07\\
3500	1.24464454742679e-07\\
3505	1.24151549485418e-07\\
3510	1.23839827503858e-07\\
3515	1.23529370954501e-07\\
3520.00000000001	1.23220071923669e-07\\
3525	1.22912030997568e-07\\
3530	1.22605192887093e-07\\
3535	1.22299588234398e-07\\
3540	1.21995156643351e-07\\
3545	1.21691979826366e-07\\
3550	1.21390271390354e-07\\
3555	1.21089761329074e-07\\
3560	1.20790352831079e-07\\
3565	1.20492169797259e-07\\
3570	1.20195134512002e-07\\
3575	1.19899222994491e-07\\
3580	1.19604494530634e-07\\
3585	1.19310715973597e-07\\
3590	1.19018291222517e-07\\
3595	1.18726896092269e-07\\
3600	1.18436537688282e-07\\
3605	1.18147310157468e-07\\
3610	1.17859291437483e-07\\
3615	1.17572274582756e-07\\
3620	1.17286378609194e-07\\
3625	1.17001535793193e-07\\
3630	1.16717792320031e-07\\
3635	1.16435164621009e-07\\
3640	1.16153559437393e-07\\
3645	1.15873062922489e-07\\
3650	1.15593576488493e-07\\
3655	1.15315160753582e-07\\
3660	1.15037715797683e-07\\
3665	1.14761411040831e-07\\
3670	1.14486082836152e-07\\
3675	1.14211838209144e-07\\
3680	1.13938529722191e-07\\
3685	1.13666330348039e-07\\
3690	1.133951326171e-07\\
3695	1.13124896117256e-07\\
3700	1.1285579892828e-07\\
3705	1.12587536182929e-07\\
3710	1.12320364786811e-07\\
3715	1.12054030498854e-07\\
3720	1.11788865941875e-07\\
3725	1.11524583568112e-07\\
3730	1.11261312829569e-07\\
3735	1.10999084368402e-07\\
3740.00000000001	1.10737681024986e-07\\
3745	1.10477301085154e-07\\
3750	1.10217796001066e-07\\
3755	1.09959391592085e-07\\
3760	1.09701872252899e-07\\
3765	1.09445173146483e-07\\
3770.00000000001	1.09189569164059e-07\\
3775	1.08934824050166e-07\\
3780	1.08681049271198e-07\\
3785	1.08428168887898e-07\\
3790	1.08176226643053e-07\\
3795.00000000001	1.07925170356182e-07\\
3800	1.07675016680631e-07\\
3805	1.07425812245765e-07\\
3810	1.07177394514935e-07\\
3815	1.06930090115753e-07\\
3820.00000000001	1.06683379907935e-07\\
3825	1.06437795244219e-07\\
3830	1.06193034365987e-07\\
3835	1.0594915322848e-07\\
3840.00000000001	1.05706107200732e-07\\
3845	1.05463944466422e-07\\
3850	1.05222645041536e-07\\
3855	1.04982229132133e-07\\
3860	1.04742720719031e-07\\
3865	1.04503943276768e-07\\
3870	1.04266125289243e-07\\
3875	1.04029065139954e-07\\
3880	1.03792855643547e-07\\
3885	1.03557526331954e-07\\
3890	1.0332295063975e-07\\
3895	1.0308934528247e-07\\
3900	1.02856483330527e-07\\
3905	1.02624437614551e-07\\
3910	1.02393122425326e-07\\
3915	1.02162768023106e-07\\
3920	1.01933204765814e-07\\
3925	1.01704288990589e-07\\
3930	1.01476377523113e-07\\
3935	1.01249156836403e-07\\
3940	1.01022811005436e-07\\
3945	1.00797100222039e-07\\
3950	1.00572304706503e-07\\
3955	1.00348175324783e-07\\
3960	1.00124973423377e-07\\
3965	9.9902513817085e-08\\
3970	9.96806954756125e-08\\
3975	9.94598030601424e-08\\
3980	9.92396442800469e-08\\
3985	9.90202286832443e-08\\
3990	9.88015025349398e-08\\
3995	9.85835972855398e-08\\
4000	9.8366415235418e-08\\
4005	9.81499059804492e-08\\
4010	9.79342675844209e-08\\
4015	9.7719355052206e-08\\
4020	9.75050693519109e-08\\
4025	9.72916152086611e-08\\
4030.00000000001	9.70789642007475e-08\\
4035	9.68668878442714e-08\\
4040	9.66555404602332e-08\\
4045	9.64450284079987e-08\\
4050	9.6235099888986e-08\\
4055	9.60259867177626e-08\\
4060.00000000001	9.58176333831773e-08\\
4065	9.56098387128179e-08\\
4070	9.54027785660117e-08\\
4075	9.51964573836505e-08\\
4080.00000000001	9.49908571801216e-08\\
4085	9.47858549427138e-08\\
4090	9.45815985531339e-08\\
4095.00000000001	9.43781615081463e-08\\
4100	9.41755087247031e-08\\
4105	9.39736393146262e-08\\
4110	9.37724082827885e-08\\
4115	9.35717796579638e-08\\
4120	9.33719066509298e-08\\
4125	9.31726349406858e-08\\
4130.00000000001	9.29741224009462e-08\\
4135	9.27761818481087e-08\\
4140.00000000001	9.25788774530644e-08\\
4145.00000000001	9.23823990639505e-08\\
4150	9.21864062863876e-08\\
4155.00000000001	9.19910874142005e-08\\
4160	9.17964362301404e-08\\
4165	9.16024223140966e-08\\
4170	9.14090059200845e-08\\
4175	9.12162980704068e-08\\
4180	9.10241411133939e-08\\
4185	9.08326829307527e-08\\
4190	9.06418340385072e-08\\
4195.00000000001	9.04517103439417e-08\\
4200	9.02620291842738e-08\\
4205	9.00731125241804e-08\\
4210	8.98847296593175e-08\\
4215	8.96969363228805e-08\\
4220	8.95098495323764e-08\\
4225	8.93232778853558e-08\\
4230	8.91373734823731e-08\\
4235	8.89520590519056e-08\\
4240	8.87673783367404e-08\\
4245.00000000001	8.85832498465077e-08\\
4250	8.83997512968194e-08\\
4255	8.82168083027323e-08\\
4260	8.80345834008978e-08\\
4265.00000000001	8.78527690595377e-08\\
4270	8.76715688935547e-08\\
4275	8.74909993342499e-08\\
4280	8.73110481691697e-08\\
4285	8.71315166683928e-08\\
4290.00000000001	8.69527772007218e-08\\
4295.00000000001	8.67744733845654e-08\\
4300.00000000001	8.65968319274658e-08\\
4305	8.64196212368995e-08\\
4310	8.62431477344216e-08\\
4315	8.6067075910634e-08\\
4320	8.58917315049723e-08\\
4325	8.57167370416078e-08\\
4330	8.55424868717591e-08\\
4335	8.53686126234266e-08\\
4340.00000000001	8.519547400887e-08\\
4345.00000000001	8.50226795634513e-08\\
4350.00000000001	8.48505958828128e-08\\
4355.00000000001	8.46789622865885e-08\\
4360	8.45079846101271e-08\\
4365.00000000001	8.43374599046598e-08\\
4370	8.41675218410387e-08\\
4375	8.39980212052893e-08\\
4380	8.38291425164783e-08\\
4385.00000000001	8.36608187171351e-08\\
4390	8.34929669846218e-08\\
4395.00000000001	8.33255693333257e-08\\
4400	8.31588342631306e-08\\
4405	8.29926363188349e-08\\
4410.00000000001	8.2826745906317e-08\\
4415	8.26615105253836e-08\\
4420	8.24968162671524e-08\\
4425.00000000001	8.23326293808436e-08\\
4430	8.21688772578711e-08\\
4435	8.20057248773765e-08\\
4440	8.18429801707765e-08\\
4445	8.16807912418227e-08\\
4450	8.15192131575771e-08\\
4455	8.13579130731767e-08\\
4460	8.11971789804744e-08\\
4465	8.10370373027781e-08\\
4470	8.08774580640659e-08\\
4475	8.07181481654595e-08\\
4480	8.05594844166535e-08\\
4485	8.04012136867982e-08\\
4490	8.02434947377861e-08\\
4495	8.00862909322575e-08\\
4500	7.99293700115556e-08\\
4505	7.97731494195375e-08\\
4510	7.96172656691851e-08\\
4515	7.94619436916832e-08\\
4520	7.9307059142053e-08\\
4525	7.91526202359449e-08\\
4530	7.89986978055879e-08\\
4535	7.88452363398306e-08\\
4540	7.86922065287854e-08\\
4545	7.85396880864652e-08\\
4550	7.83876157317565e-08\\
4555	7.82360067841382e-08\\
4560	7.8084831045544e-08\\
4565	7.79340918466432e-08\\
4570	7.77838244925277e-08\\
4575	7.76341084751663e-08\\
4580	7.7484717309062e-08\\
4585	7.73357164973732e-08\\
4590	7.71873265303924e-08\\
4595.00000000001	7.70393153715077e-08\\
4600	7.68917338689335e-08\\
4605.00000000001	7.67446399763116e-08\\
4610	7.65979311090347e-08\\
4615	7.64517029683276e-08\\
4620.00000000001	7.63058027875019e-08\\
4625	7.61604004306804e-08\\
4630	7.60155078882718e-08\\
4635	7.58710569925825e-08\\
4640	7.57269220663659e-08\\
4645	7.55832614274255e-08\\
4650.00000000001	7.54400415470258e-08\\
4655	7.52972189044244e-08\\
4660	7.51548328015162e-08\\
4665.00000000001	7.50127697735081e-08\\
4670.00000000001	7.48712392084628e-08\\
4675	7.47301722725524e-08\\
4680.00000000001	7.45894406239955e-08\\
4685	7.44491064352815e-08\\
4690.00000000001	7.43092607446983e-08\\
4695	7.41697623318771e-08\\
4700.00000000001	7.403070290124e-08\\
4705	7.3892012064647e-08\\
4710	7.37538365935819e-08\\
4715	7.36159022629579e-08\\
4720.00000000001	7.34784564304647e-08\\
4725.00000000001	7.33412861553262e-08\\
4730	7.32047455986874e-08\\
4735	7.30684455163555e-08\\
4740.00000000001	7.29325693171745e-08\\
4745	7.27970610459039e-08\\
4750	7.2661968220089e-08\\
4755	7.25273243684653e-08\\
4760	7.23930200230427e-08\\
4765.00000000001	7.225914000486e-08\\
4770	7.21258925917567e-08\\
4775	7.19928874293174e-08\\
4780	7.18603241356419e-08\\
4785	7.17280985718105e-08\\
4790	7.15963095476724e-08\\
4795	7.14649488475771e-08\\
4800	7.13338632607474e-08\\
4805	7.12031957839088e-08\\
4810	7.10728789155012e-08\\
4815	7.09429535117323e-08\\
4820	7.08133847115988e-08\\
4825	7.06841756237253e-08\\
4830	7.05552845037261e-08\\
4835	7.04269953466508e-08\\
4840	7.02988445233643e-08\\
4845	7.01711053707754e-08\\
4850	7.00436693090722e-08\\
4855	6.99167488349417e-08\\
4860	6.97901640922538e-08\\
4865	6.96638358110846e-08\\
4870	6.95378472581609e-08\\
4875	6.94123523103941e-08\\
4880	6.92871857665977e-08\\
4885	6.9162191529415e-08\\
4890.00000000001	6.90377472967185e-08\\
4895	6.89135928322315e-08\\
4900	6.87896193340976e-08\\
4905.00000000001	6.86662093851709e-08\\
4910	6.85431242875013e-08\\
4915.00000000001	6.84203642631331e-08\\
4920	6.82977996380174e-08\\
4925	6.81758729470517e-08\\
4930.00000000001	6.80540366282401e-08\\
4935	6.79326830343996e-08\\
4940	6.78116327534894e-08\\
4945	6.76908120667008e-08\\
4950.00000000001	6.75703382135851e-08\\
4955	6.74503064512777e-08\\
4960.00000000001	6.73305198262142e-08\\
4965	6.72112074884268e-08\\
4970	6.70920581313794e-08\\
4975.00000000001	6.69733735136901e-08\\
4980	6.68549902105297e-08\\
4985	6.67368671436464e-08\\
4990	6.66191224407698e-08\\
4995	6.65016552936493e-08\\
5000	6.63845209913916e-08\\
5005.00000000001	6.62677239748888e-08\\
5010	6.61512156163724e-08\\
5015.00000000001	6.60351258119362e-08\\
5020	6.59193692964521e-08\\
5025.00000000001	6.58038203926736e-08\\
5030	6.5688546380116e-08\\
5035	6.55736294152831e-08\\
5040.00000000001	6.54591720827825e-08\\
5045	6.53448786192001e-08\\
5050.00000000001	6.52308980164662e-08\\
5055	6.51173461818644e-08\\
5060.00000000001	6.50039917449164e-08\\
5065	6.48909412870324e-08\\
5070.00000000001	6.47782252283236e-08\\
5075	6.46658375735853e-08\\
5080.00000000001	6.45537452381718e-08\\
5085.00000000001	6.44418443052075e-08\\
5090.00000000001	6.43303446068444e-08\\
5095	6.42192510280638e-08\\
5100	6.41081878693938e-08\\
5105	6.39976307503787e-08\\
5110	6.38872696967496e-08\\
5115	6.37772894496181e-08\\
5120.00000000001	6.36675676624066e-08\\
5125	6.35582406705026e-08\\
5130	6.34490719964021e-08\\
5135.00000000001	6.33403369754148e-08\\
5140.00000000001	6.32317469495547e-08\\
5145	6.31234988723862e-08\\
5150.00000000001	6.3015586082571e-08\\
5155	6.29078953373607e-08\\
5160	6.28004159786144e-08\\
5165	6.26933609471081e-08\\
5170	6.25865377301693e-08\\
5175	6.24800804427394e-08\\
5180	6.23738198868296e-08\\
5185	6.22678435480139e-08\\
5190	6.21621485397127e-08\\
5195.00000000001	6.20566946718526e-08\\
5200	6.19515554411976e-08\\
5205	6.18467541624313e-08\\
5210.00000000001	6.17421243020999e-08\\
5215	6.163795451819e-08\\
5220	6.15338593412674e-08\\
5225	6.14300668111412e-08\\
5230	6.13267470139789e-08\\
5235	6.12235087071866e-08\\
5240	6.11205677181203e-08\\
5245	6.10179586857384e-08\\
5250	6.09155574871068e-08\\
5255	6.08133561286194e-08\\
5260	6.07115353545851e-08\\
5265	6.06099432864937e-08\\
5270.00000000001	6.0508587473862e-08\\
5275	6.04075578447549e-08\\
5280	6.03067227267218e-08\\
5285	6.02061545063038e-08\\
5290	6.01060234917128e-08\\
5295	6.00059946176401e-08\\
5300	5.99061926731537e-08\\
5305	5.98067271262437e-08\\
5310	5.97075651143086e-08\\
5315	5.96085725224072e-08\\
5320	5.95097882083451e-08\\
5325	5.9411436659218e-08\\
5330	5.93131250781199e-08\\
5335.00000000001	5.92151689904341e-08\\
5340	5.91174511566095e-08\\
5345	5.90199775718503e-08\\
5350	5.8922796863925e-08\\
5355.00000000001	5.88258759481874e-08\\
5360.00000000001	5.87292767750824e-08\\
5365	5.86327946194842e-08\\
5370	5.85364536842548e-08\\
5375	5.84405985826208e-08\\
5380.00000000001	5.83449200064479e-08\\
5385	5.82492876155528e-08\\
5390.00000000001	5.81541643729366e-08\\
5395	5.80591192900215e-08\\
5400.00000000001	5.79645051956846e-08\\
5405.00000000001	5.78699832498586e-08\\
5410	5.77757182096406e-08\\
5415	5.76816603370389e-08\\
5420.00000000001	5.75879033348769e-08\\
5425	5.74943872511113e-08\\
5430	5.74009944021014e-08\\
5435	5.73079002030851e-08\\
5440.00000000001	5.72151064304194e-08\\
5445	5.71225009515785e-08\\
5450	5.70300628943699e-08\\
5455	5.69380502746242e-08\\
5460.00000000001	5.68461422378874e-08\\
5465.00000000001	5.67545526131142e-08\\
5470	5.66630498077814e-08\\
5475	5.65718225598034e-08\\
5480	5.64807445257998e-08\\
5485.00000000001	5.6390144331786e-08\\
5490.00000000001	5.6299583661712e-08\\
5495	5.6209197740742e-08\\
5500	5.61192565751156e-08\\
5505.00000000001	5.60293440532433e-08\\
5510	5.59397521637805e-08\\
5515	5.58504045233832e-08\\
5520.00000000001	5.57611627982624e-08\\
5525	5.56723180888952e-08\\
5530.00000000001	5.55835264481885e-08\\
5535.00000000001	5.54950223552453e-08\\
5540.00000000001	5.54067449698437e-08\\
5545.00000000001	5.53186554341778e-08\\
5550	5.52308638823718e-08\\
5555	5.5143253741008e-08\\
5560	5.50557632816861e-08\\
5565	5.4968625429197e-08\\
5570.00000000001	5.48815273226921e-08\\
5575	5.47948437734647e-08\\
5580	5.47081864255006e-08\\
5585	5.4621876133254e-08\\
5590	5.45357237147214e-08\\
5595	5.44499016985611e-08\\
5600	5.43642304506875e-08\\
5605	5.4278619820991e-08\\
5610	5.41933342645962e-08\\
5615.00000000001	5.41084435035088e-08\\
5620	5.40236420043528e-08\\
5625	5.3939096522626e-08\\
5630	5.38548297068787e-08\\
5635	5.3770700114697e-08\\
5640	5.36866973099848e-08\\
5645	5.36030957398737e-08\\
5650	5.35196131856707e-08\\
5655	5.34362143422839e-08\\
5660.00000000001	5.33530861712706e-08\\
5665	5.32703630096165e-08\\
5670	5.31875516962544e-08\\
5675.00000000001	5.31050086127749e-08\\
5680	5.30228461137483e-08\\
5685	5.29406312121949e-08\\
5690.00000000001	5.28587136283675e-08\\
5695.00000000001	5.27771504277297e-08\\
5700	5.26955403756802e-08\\
5705	5.261427071801e-08\\
5710	5.25330814404868e-08\\
5715.00000000001	5.24521890366003e-08\\
5720.00000000001	5.23714751565763e-08\\
5725	5.2291040164576e-08\\
5730	5.22105569888964e-08\\
5735	5.21304430733949e-08\\
5740	5.20506295842438e-08\\
5745	5.19708296398136e-08\\
5750	5.18911964508817e-08\\
5755	5.18118163927995e-08\\
5760	5.17325944304759e-08\\
5765	5.16535982875155e-08\\
5770.00000000001	5.15747446971915e-08\\
5775.00000000001	5.14961941977532e-08\\
5780.00000000001	5.14177456167886e-08\\
5785.00000000001	5.13395781442938e-08\\
5790	5.12613840264464e-08\\
5795	5.11835074323841e-08\\
5800.00000000001	5.11058804164577e-08\\
5805	5.10283928445432e-08\\
5810	5.09510513779787e-08\\
5815	5.08738768889572e-08\\
5820.00000000001	5.07969317720125e-08\\
5825	5.07201807220525e-08\\
5830.00000000001	5.06434787439503e-08\\
5835.00000000001	5.05671420292231e-08\\
5840	5.04909656307006e-08\\
5845	5.04147665836286e-08\\
5850.00000000001	5.03388863926091e-08\\
5855	5.02632426790939e-08\\
5860.00000000001	5.01877397418582e-08\\
5865.00000000001	5.01124730600821e-08\\
5870.00000000001	5.00373453782288e-08\\
5875	4.99621946037365e-08\\
5880.00000000001	4.98874648258152e-08\\
5885	4.98128334136539e-08\\
5890	4.97382768305244e-08\\
5895.00000000001	4.96639820379841e-08\\
5900.00000000001	4.95898446750686e-08\\
5905.00000000001	4.95159118152344e-08\\
5910.00000000001	4.94420815400076e-08\\
5915.00000000001	4.93685301528047e-08\\
5920	4.92951111041862e-08\\
5925.00000000001	4.92218519276833e-08\\
5930.00000000001	4.91487390785749e-08\\
5935	4.90757663396124e-08\\
5940	4.90030653832463e-08\\
5945	4.89304654571753e-08\\
5950	4.88580615964906e-08\\
5955	4.87857905184796e-08\\
5960	4.87137528093484e-08\\
5965.00000000001	4.86418960665702e-08\\
5970	4.85699809260609e-08\\
5975	4.84985440696306e-08\\
5980.00000000001	4.84271014400406e-08\\
5985	4.83558140196294e-08\\
5990	4.82847293259426e-08\\
5995.00000000001	4.8213760761584e-08\\
6000	4.81429820453627e-08\\
};
\addplot [color=black, line width=0.9pt, forget plot]
  table[row sep=crcr]{%
200	8.565008204102601e-06\\
6000	2.227232665603490e-08\\
};
\node[] at (axis cs: 650, 1e-7) {\small $\zeta=1$};
\node[] at (axis cs: 650, 3e-8) {\small $v(x)=x$};
\end{axis}
\end{tikzpicture}%

%% file: plot_per1.tex
%
%
\begin{tikzpicture}

\begin{axis}[%
width=1.5in,
height=1.7in,
at={(0.758in,0.481in)},
scale only axis,
xmode=log,
xmin=1,
xmax=10^4,
xminorticks=true,
xlabel={$N$},
ymode=log,
ymin=1e-16,
ymax=1e-1,
ylabel={Error},
yminorticks=true,
title = {$r=3/4$},
axis background/.style={fill=white},
legend style={legend cell align=left, align=left, draw=white!15!black, font=\footnotesize, at={(0.92,0.55)}}
]
\addplot [color=red, only marks,line width=1pt, mark size = 3.5pt, mark=o, mark options={solid, red}]
  table[row sep=crcr]{%
4	0.00477857974206447\\
16	0.000571379294527858\\
64.0000000000001	5.35299291255972e-05\\
256	4.83934445509629e-06\\
1024	4.6014125023438e-07\\
};
\addlegendentry{CEQR2 T}
\addplot [color=black, line width=0.8pt, forget plot]
  table[row sep=crcr]{%
4	0.00753895424384009\\
1024	4.6014125023438e-07\\
};

\addplot [color=blue, only marks, line width=1pt, mark size=3.5pt, mark=x, mark options={solid, blue}]
  table[row sep=crcr]{%
4	1.0778315928077e-15\\
16	3.11060637828084e-15\\
64.0000000000001	7.40611131540354e-15\\
256	2.34998445313485e-14\\
1024	7.31462408192199e-14\\
};
\addlegendentry{EQRF2 T}

\end{axis}

\end{tikzpicture}%

%% file: plot_per2.tex
%
%
\begin{tikzpicture}

\begin{axis}[%
width=1.5in,
height=1.7in,
at={(0.758in,0.481in)},
scale only axis,
xmode=log,
xmin=1,
xmax=10^4,
xminorticks=true,
xlabel={$N$},
ymode=log,
ymin=1e-16,
ymax=1e-1,
ylabel={Error},
yminorticks=true,
title = {$r=1/2$},
axis background/.style={fill=white},
legend style={legend cell align=left, align=left, draw=white!15!black, font=\footnotesize, at={(0.92,0.4)}}
]
\addplot [color=red, only marks,line width=1pt, mark size = 3.5pt, mark=o, mark options={solid, red}]
  table[row sep=crcr]{%
4	0.0143918454776753\\
16	0.00229650918644059\\
64.0000000000001	0.000282393666343389\\
256	3.60335537206172e-05\\
1024	4.62824347354523e-06\\
};
\addplot [color=black, dashed, line width=0.8pt, forget plot]
  table[row sep=crcr]{%
4	0.0189572852676413\\
1024	4.62824347354523e-06\\
};

\addplot [color=blue, only marks, line width=1pt, mark size=3.5pt, mark=x, mark options={solid, blue}]
  table[row sep=crcr]{%
4	1.15511223068214e-15\\
16	3.33806235505671e-15\\
64.0000000000001	7.48557119891195e-15\\
256	2.21889813637079e-14\\
1024	6.91163991687798e-14\\
};

\end{axis}

\end{tikzpicture}%

%% file: plot_per3.tex
%
%
\begin{tikzpicture}

\begin{axis}[%
width=1.5in,
height=1.7in,
at={(0.758in,0.481in)},
scale only axis,
xmode=log,
xmin=1,
xmax=10^4,
xminorticks=true,
xlabel={$N$},
ymode=log,
ymin=1e-16,
ymax=1e-1,
ylabel={Error},
yminorticks=true,
title = {$r=1/4$},
axis background/.style={fill=white},
legend style={legend cell align=left, align=left, draw=white!15!black, font=\footnotesize, at={(0.92,0.4)}}
]
\addplot [color=red, only marks,line width=1pt, mark size = 3.5pt, mark=o, mark options={solid, red}]
  table[row sep=crcr]{%
4	0.0347814716122707\\
16	0.007568597628358\\
64.0000000000001	0.001270067489167\\
256	0.000232537541375689\\
1024	4.1331875499331e-05\\
};
\addplot [color=black, dotted, line width=0.8pt, forget plot]
  table[row sep=crcr]{%
4	0.0423238405113149\\
1024	4.1331875499331e-05\\
};

\addplot [color=blue, only marks, line width=1pt, mark size=3.5pt, mark=x, mark options={solid, blue}]
  table[row sep=crcr]{%
4	1.00534970772086e-15\\
16	3.60480714057903e-15\\
64.0000000000001	7.12707841441686e-15\\
256	1.97409080377525e-14\\
1024	7.62107268627205e-14\\
};

\end{axis}

\end{tikzpicture}%

%% file: plot_periodicrad1.tex
%
%
\begin{tikzpicture}

\begin{axis}[%
width=1.5in,
height=1.7in,
at={(0.758in,0.481in)},
scale only axis,
xmode=log,
xmin=10,
xmax=200,
xminorticks=true,
xlabel={$N$},
xtick = {10,50,200},
xticklabels = {{10},{50},{200}},
ymode=log,
ymin=1e-9,
ymax=1e-1,
ylabel={Error},
yminorticks=true,
title = {$r=3/4$},
axis background/.style={fill=white},
legend style={legend cell align=left, at={(0.975,0.25)}, align=left, draw=white!15!black, font=\footnotesize}
]
\addplot [color=red, only marks,line width=1pt, mark size = 3.5pt, mark=o, mark options={solid, red}]
  table[row sep=crcr]{%
20	0.00171977853206307\\
40	0.000433463895182722\\
60	0.000193754610285997\\
80	0.000109057006046902\\
100	6.98350204351073e-05\\
};
\addlegendentry{EQRF2 T}
\addplot [color=black, dashed, line width=0.8pt, forget plot]
  table[row sep=crcr]{%
20	0.00174587551087768\\
100	6.98350204351073e-05\\
};

\addplot [color=magenta, only marks, line width=1pt, mark size=3.5pt, mark=+, mark options={solid, magenta}]
  table[row sep=crcr]{%
20	6.7854264681226e-05\\
40	9.51572116867524e-06\\
60	3.70747311667576e-06\\
80	1.68376456349138e-06\\
100	9.34084869686216e-07\\
};
\addlegendentry{EQRF2 GR}
\addplot [color=black, line width=0.8pt, forget plot]
  table[row sep=crcr]{%
20	5.22169316343103e-05\\
100	9.34084869686216e-07\\
};

\end{axis}

\end{tikzpicture}%

%% file: plot_periodicrad2.tex
%
%
\begin{tikzpicture}

\begin{axis}[%
width=1.5in,
height=1.7in,
at={(0.758in,0.481in)},
scale only axis,
xmode=log,
xmin=10,
xmax=200,
xminorticks=true,
xlabel={$N$},
xtick = {10,50,200},
xticklabels = {{10},{50},{200}},
ymode=log,
ymin=1e-9,
ymax=1e-1,
ylabel={Error},
yminorticks=true,
title = {$r=1/2$},
axis background/.style={fill=white},
legend style={legend cell align=left, align=left, draw=white!15!black, font=\footnotesize}
]
\addplot [color=red, only marks,line width=1pt, mark size = 3.5pt, mark=o, mark options={solid, red}]
  table[row sep=crcr]{%
20	0.00146688396520143\\
40	0.000414498181503756\\
60	0.000196561806214586\\
80	0.000114286565712038\\
100	7.48557031462078e-05\\
};
\addplot [color=black, dashed, line width=0.8pt]
  table[row sep=crcr]{%
20	0.00187139257865519\\
100	7.48557031462078e-05\\
};

\addplot [color=magenta, only marks, line width=1pt, mark size=3.5pt, mark=+, mark options={solid, magenta}]
  table[row sep=crcr]{%
20	0.000289497391722791\\
40	6.50507598472707e-05\\
60	3.05663931449049e-05\\
80	1.62580472441726e-05\\
100	1.00010169393171e-05\\
};
\addplot [color=black, dashed, line width=0.8pt]
  table[row sep=crcr]{%
20	0.000250025423482927\\
100	1.00010169393171e-05\\
};

\end{axis}

\end{tikzpicture}%

%% file: plot_periodicrad3.tex
%
%
\begin{tikzpicture}

\begin{axis}[%
width=1.5in,
height=1.7in,
at={(0.758in,0.481in)},
scale only axis,
xmode=log,
xmin=10,
xmax=200,
xminorticks=true,
xlabel={$N$},
xtick = {10,50,200},
xticklabels = {{10},{50},{200}},
ymode=log,
ymin=1e-9,
ymax=1e-1,
ylabel={Error},
yminorticks=true,
title = {$r=1/4$},
axis background/.style={fill=white},
legend style={legend cell align=left, align=left, draw=white!15!black, font=\footnotesize}
]
\addplot [color=red, only marks,line width=1pt, mark size = 3.5pt, mark=o, mark options={solid, red}]
  table[row sep=crcr]{%
20	0.00247277231517805\\
40	0.000950949598629419\\
60	0.000529583900835424\\
80	0.000344329695639994\\
100	0.000245016450619068\\
};
\addplot [color=black, dotted, line width=0.8pt]
  table[row sep=crcr]{%
20	0.00273936719594978\\
100	0.000245016450619068\\
};

\addplot [color=magenta, only marks, line width=1pt, mark size=3.5pt, mark=+, mark options={solid, magenta}]
  table[row sep=crcr]{%
20	0.00116518374184963\\
40	0.000389000916036611\\
60	0.000212836231380704\\
80	0.000132392111148238\\
100	9.10738660559894e-05\\
};
\addplot [color=black, dotted, line width=0.8pt]
  table[row sep=crcr]{%
20	0.00101823677737451\\
100	9.10738660559894e-05\\
};

\end{axis}

\end{tikzpicture}%

%% file: plot_periodiclob1.tex
%
%
\begin{tikzpicture}

\begin{axis}[%
width=1.5in,
height=1.7in,
at={(0.758in,0.481in)},
scale only axis,
xmode=log,
xmin=10,
xmax=200,
xminorticks=true,
xlabel={$N$},
xtick = {10,50,200},
xticklabels = {{10},{50},{200}},
ymode=log,
ymin=1e-9,
ymax=1e-1,
ylabel={Error},
yminorticks=true,
title = {$r=3/4$},
axis background/.style={fill=white},
legend style={legend cell align=left, align=left, draw=white!15!black, font=\footnotesize}
]
\addplot [color=green, only marks,line width=1pt, mark size = 3.5pt, mark=triangle, mark options={solid, green}]
  table[row sep=crcr]{%
20	8.16427691369955e-06\\
40	9.95095297695362e-07\\
60	3.02809543293357e-07\\
80	1.31177648417562e-07\\
100	6.77729275460979e-08\\
};
\addlegendentry{EQRF3 NC}
\addplot [color=black, dashed, line width=0.8pt, forget plot]
  table[row sep=crcr]{%
20	8.47161594326224e-06\\
100	6.77729275460979e-08\\
};

\addplot [color=violet, only marks, line width=1pt, mark size=3.5pt, mark=asterisk, mark options={solid, violet}]
  table[row sep=crcr]{%
20	1.42692348454141e-06\\
40	1.32282372626167e-07\\
60	3.44634450640606e-08\\
80	1.27708427830807e-08\\
100	6.53725922265487e-09\\
};
\addlegendentry{EQRF3 GL}
\addplot [color=black, line width=0.8pt, forget plot]
  table[row sep=crcr]{%
20	1.22193532639052e-06\\
100	6.53725922265487e-09\\
};

\end{axis}

\end{tikzpicture}%

%% file: plot_periodiclob2.tex
%
%
\begin{tikzpicture}

\begin{axis}[%
width=1.5in,
height=1.7in,
at={(0.758in,0.481in)},
scale only axis,
xmode=log,
xmin=10,
xmax=200,
xminorticks=true,
xlabel={$N$},
xtick = {10,50,200},
xticklabels = {{10},{50},{200}},
ymode=log,
ymin=1e-9,
ymax=1e-1,
ylabel={Error},
yminorticks=true,
title = {$r=1/2$},
axis background/.style={fill=white},
legend style={legend cell align=left, align=left, draw=white!15!black, font=\footnotesize}
]
\addplot [color=green, only marks,line width=1pt, mark size = 3.5pt, mark=triangle, mark options={solid, green}]
  table[row sep=crcr]{%
20	7.92869176773069e-06\\
40	8.93879049351703e-07\\
60	3.16532735748157e-07\\
80	1.59751110881895e-07\\
100	8.84468883783463e-08\\
};
\addplot [color=black, dotted, line width=0.8pt, forget plot]
  table[row sep=crcr]{%
20	1.07090380543422e-05\\
100	1.9156909650513e-07\\
};

\addplot [color=violet, only marks, line width=1pt, mark size=3.5pt, mark=asterisk, mark options={solid, violet}]
  table[row sep=crcr]{%
20	1.16188254743319e-05\\
40	1.93424242571304e-06\\
60	6.92972729655206e-07\\
80	3.30709372143842e-07\\
100	1.9156909650513e-07\\
};
\addplot [color=black, dotted, line width=0.8pt, forget plot]
  table[row sep=crcr]{%
20	4.94433137030796e-06\\
100	8.84468883783463e-08\\
};

\end{axis}

\end{tikzpicture}%

%% file: plot_periodiclob3.tex
%
%
\begin{tikzpicture}

\begin{axis}[%
width=1.5in,
height=1.7in,
at={(0.758in,0.481in)},
scale only axis,
xmode=log,
xmin=10,
xmax=200,
xminorticks=true,
xlabel={$N$},
xtick = {10,50,200},
xticklabels = {{10},{50},{200}},
ymode=log,
ymin=1e-9,
ymax=1e-1,
ylabel={Error},
yminorticks=true,
title = {$r=1/4$},
axis background/.style={fill=white},
legend style={legend cell align=left, align=left, draw=white!15!black, font=\footnotesize}
]
\addplot [color=green, only marks,line width=1pt, mark size = 3.5pt, mark=triangle, mark options={solid, green}]
  table[row sep=crcr]{%
20	3.33212758011389e-05\\
40	8.63192304060295e-06\\
60	4.19949717282434e-06\\
80	2.62591170191901e-06\\
100	1.9456258326123e-06\\
};
\addplot [color=black, dashdotted, line width=0.8pt, forget plot]
  table[row sep=crcr]{%
20	7.05407129858787e-05\\
100	4.21931876759639e-06\\
};

\addplot [color=violet, only marks, line width=1pt, mark size=3.5pt, mark=asterisk, mark options={solid, violet}]
  table[row sep=crcr]{%
20	6.73940525635097e-05\\
40	1.95969950621468e-05\\
60	9.70111956025368e-06\\
80	5.98570928533983e-06\\
100	4.21931876759639e-06\\
};
\addplot [color=black, dashdotted,line width=0.8pt, forget plot]
  table[row sep=crcr]{%
20	3.25279603167789e-05\\
100	1.9456258326123e-06\\
};

\end{axis}

\end{tikzpicture}%

%% file: plot_homdirg_err.tex
%
%
%
\begin{tikzpicture}

\begin{axis}[%
width=1.7in,
height=1.7in,
at={(0.758in,0.481in)},
scale only axis,
xmode=log,
xmin=10,
xmax=400,
xminorticks=true,
ymode=log,
ymin=1e-9,
ymax=1e-05,
yminorticks=true,
axis background/.style={fill=white},
legend style={legend cell align=left, align=left, draw=white!15!black, font=\footnotesize},
xlabel = {$N$\phantom{(}},
ylabel = {Error}
]
\addplot [color=red, only marks, line width=1pt, mark size=3.5pt, mark=triangle, mark options={solid, red}]
  table[row sep=crcr]{%
20	1.0024703346323e-06\\
40	3.19085029620858e-07\\
60	1.59602859817198e-07\\
80	9.7213336380797e-08\\
100	6.60789259976369e-08\\
};
\addlegendentry{CEQR2 G}

\addplot [color=blue, only marks, line width=1pt, mark size=3.5pt,mark=x, mark options={solid, blue}]
  table[row sep=crcr]{%
20	2.19619727737096e-07\\
40	3.20239492701546e-08\\
60	1.098343171968e-08\\
80	5.20906517920139e-09\\
100	2.93394386563506e-09\\
};
\addlegendentry{EQRF2 G (I)}

\addplot [color=magenta, only marks, line width=1pt, mark size=3.5pt, mark=o, mark options={solid, magenta}]
  table[row sep=crcr]{%
20	2.21060592831002e-07\\
40	3.24374891391699e-08\\
60	1.11837348271138e-08\\
80	5.32910504613683e-09\\
100	3.0147286889104e-09\\
};
\addlegendentry{EQRF2 G (F)}

\addplot [color=black, line width=0.8pt, forget plot]
  table[row sep=crcr]{%
20	1.64012448143213e-07\\
100	2.93394386563506e-09\\
};
\addplot [color=black, dashed, line width=0.8pt, forget plot]
  table[row sep=crcr]{%
20	1.10474102810435e-06\\
100	6.60789259976369e-08\\
};
\end{axis}

\end{tikzpicture}%

%% file: plot_homdirg_cpu.tex
%
%
%
\begin{tikzpicture}

\begin{axis}[%
width=1.7in,
height=1.7in,
at={(0.758in,0.481in)},
scale only axis,
xmode=log,
xmin=0.01,
xmax=10,
xminorticks=true,
ymode=log,
ymin=1e-9,
ymax=1e-05,
yminorticks=true,
axis background/.style={fill=white},
legend style={legend cell align=left, align=left, draw=white!15!black},
xlabel = {Wall-clock time (s)},
ylabel = {Error}
]
\addplot [color=blue, mark=x, line width=1pt, mark size=3.5pt, mark options={solid, blue}]
  table[row sep=crcr]{%
0.02332698	2.19619727737096e-07\\
0.0434716400000001	3.20239492701546e-08\\
0.0641650300000001	1.098343171968e-08\\
0.08588043	5.20906517920139e-09\\
0.1088252	2.93394386563506e-09\\
};

\addplot [color=magenta, mark=o, line width=1pt, mark size=3.5pt, mark options={solid, magenta}]
  table[row sep=crcr]{%
1.485074	2.21060592831002e-07\\
2.907695	3.24374891391699e-08\\
4.37532000000001	1.11837348271138e-08\\
5.78190700000001	5.32910504613683e-09\\
7.24922400000001	3.0147286889104e-09\\
};

\addplot [color=red, mark=triangle, line width=1pt, mark size=3.5pt, mark options={solid, red}]
  table[row sep=crcr]{%
0.02096156	1.0024703346323e-06\\
0.04095637	3.19085029620858e-07\\
0.06248952	1.59602859817198e-07\\
0.08187862	9.7213336380797e-08\\
0.10561907	6.60789259976369e-08\\
};

\end{axis}

\end{tikzpicture}%